\documentclass[a4paper]{article}
\usepackage[
  a4paper,
  left=3cm,
  right=3cm,
  top=3cm,
  bottom=3cm
]{geometry}

\usepackage{subcaption}
\usepackage{multirow}
\usepackage{mathrsfs}
\usepackage{amsmath}
\usepackage{hyperref}
\usepackage{graphicx} 
\usepackage{subcaption}    
\usepackage{caption}       
\usepackage{float}         
\usepackage{amssymb,amsthm}
\usepackage{verbatim}
\usepackage{hyperref}
\usepackage[shortlabels]{enumitem}
\usepackage{parskip}
\usepackage{xcolor}
\numberwithin{equation}{section}
\theoremstyle{plain}
\newtheorem{theorem}{Theorem}[section]
\newtheorem{lemma}[theorem]{Lemma}
\newtheorem{proposition}[theorem]{Proposition}
\newtheorem{corollary}[theorem]{Corollary}
\newtheorem{assumption}{Assumption}
\theoremstyle{definition}
\newtheorem{definition}[theorem]{Definition}
\newtheorem{remark}[theorem]{Remark}

\newcommand{\N}{\mathbb{N}}

 \title{A priori regularity of the reverse heat flow
and dimension-dependent complexity
of higher-order diffusion samplers}

\author{Xixian Wang\thanks{Division of Mathematical Sciences, School of Physical and Mathematical Sciences, Nanyang Technological University, Singapore. xixian001@e.ntu.edu.sg}, Zhongjian Wang\thanks{Division of Mathematical Sciences, School of Physical and Mathematical Sciences, Nanyang Technological University, Singapore. zhongjian.wang@ntu.edu.sg (Corresponding)}}
\begin{document}
\maketitle

\begin{abstract}
We establish arbitrary-order a priori regularity estimates for
the reverse heat flow in the Ornstein--Uhlenbeck setting, that is,
for the probability-flow ODE of score-based diffusion models.
For compactly supported, possibly singular targets satisfying
a dimension-uniform doubling condition on supporting-cap masses,
we bound the iterated material derivatives of the flow and their
first spatial derivatives pointwise, with constants independent of
the dimension $d$ and explicit in time and support radius.
The proof rests on two ingredients: a spatially uniform bound on
normal-direction posterior fluctuations, which is the only point
where the geometry of the target enters, and a centered
posterior-moment normal form closed under material differentiation.
This algebraic structure, together with a weighted regularity
calculus, yields quantitative bounds at every order, including the
mixed space--time directional derivatives required by Runge--Kutta
stages.
As an application, for exact-score sampling stopped at forward time
$\delta>0$, order-$p$ Taylor and explicit Runge--Kutta schemes reach
accuracy $\varepsilon$ in $\widetilde O(d^{1/p}\varepsilon^{-1/p})$
steps in total variation, and the Taylor scheme in
$\widetilde O(d^{1/(2p)}\varepsilon^{-1/p})$ steps in $W_2$.
\end{abstract}

\textbf{keywords:}{
reverse heat flow, probability-flow ODE, higher-order samplers,
Runge--Kutta methods, a priori estimates, doubling condition,
iteration complexity}

\section{Introduction}

Sampling from high-dimensional probability distributions is a
fundamental problem in computational mathematics and lies at the
heart of modern generative modeling. Diffusion-based generative
models provide a prominent approach for this problem
\cite{ho2020denoising,song2020score}.
Starting from data $X_0\sim P_0$, where $P_0$ is the target distribution,
a forward diffusion gradually transforms the data into a tractable prior
distribution.
In the Ornstein--Uhlenbeck
setting considered in this work, the forward process admits the explicit representation
\begin{equation*}
        X_t
    =
    e^{-t/2}X_0+\sqrt{1-e^{-t}}Z,
    \qquad
    Z\sim N(0,I_d),
\end{equation*}
where $Z$ is independent of $X_0$. The law of $X_t$, denoted as $P_t$, is a Gaussian smoothing of
$P_0$ and $P_t$ converges to the standard Gaussian distribution as $t\to\infty$. The regularizing properties and long-time behavior of the forward
OU semigroup are standard topics in the theory of diffusion
semigroups. In particular, the Bakry--{\'E}mery framework relates
curvature bounds to gradient estimates, logarithmic Sobolev
inequalities, and hypercontractivity
\cite{bakry1985diffusions,bakry2014analysis}. 

Sampling is based on reversing the forward evolution of
the marginal laws, from $P_T$ to $P_0$, with $P_T$
approximated in practice by the standard Gaussian distribution. This can be accomplished either by the stochastic reverse-time
diffusion \cite{haussmann1986time,And_80}
or by a deterministic probability-flow ODE
\cite{song2020score}. In the OU setting, the marginals $P_t$ are the
Ornstein--Uhlenbeck heat flow of $P_0$ relative to the Gaussian
reference measure, and the deterministic flow coincides, up to a
constant time rescaling, with the reverse heat-flow transport of
\cite{kim2012generalization}; see Remark~\ref{rem:reverse-heat-flow}.
We therefore use the terms \emph{reverse heat flow} and
\emph{reverse probability flow} interchangeably.
 Writing $p_t$ for the density of $P_t$, we obtain the
probability-flow ODE by rewriting the forward Fokker--Planck
equation as a continuity equation and reversing time.
The identity $\nabla\cdot(p_tD_x\log p_t)=\Delta p_t$
expresses the diffusion term as a transport term involving
the \emph{score function} $D_x\log p_t$;
see Section~\ref{sec:Preliminaries}.
Under the Gaussian smoothing above, Tweedie's formula
expresses this score in terms of the posterior mean
$\mathbb E[X_0\mid X_t=x]$.
This posterior mean minimizes the expected squared error
in predicting $X_0$ from $X_t$.
The resulting least-squares denoising problem provides
an equivalent formulation of denoising score matching
\cite{meng2021estimating,lu2023mathematical}.
In this work, this posterior representation provides
the bridge from geometric properties of $P_0$ to regularity
estimates for the exact probability flow.

A practical sampler replaces three idealized ingredients: the terminal law
$P_T$ is approximated by a tractable prior distribution $N(0,I_d)$, the exact score
is replaced by a learned approximation by neural networks, and the continuous reverse dynamics are approximated by a finite-step numerical
integrator. The resulting sampling error is therefore
naturally decomposed into initialization, score-estimation, and
discretization errors. In this work, we focus on the discretization error of the
exact probability-flow ODE.
Such focus is natural. The initialization error, that is, the
discrepancy between $P_T$ and the standard Gaussian prior, is retained
explicitly in our final bounds and decays exponentially in $T$. The
additional error incurred by representing the Gaussian prior through
finitely many samples is governed by the convergence rate of empirical
measures \cite{fournier2015rate} and is propagated through the numerical
flow by the stability estimates of
Section~\ref{sec:high-order-convergence}. The score-estimation error is
related to the approximation theory of prescribed neural network
structures and to practical training procedures, for instance
\cite{neufeld2026universal}, and lies outside the scope of the present
analysis.

The discretization analysis also yields bounds on the number
of time steps needed to achieve a prescribed accuracy.
Although these bounds depend on several problem parameters,
we pay particular attention to the ambient dimension $d$,
since applications of diffusion models typically involve
high-dimensional data. For example, even a $256\times256$ RGB
image has $196\,608$ scalar coordinates.
A convergence bound that leaves its dimension dependence
implicit may therefore give little guidance on the number
of steps sufficient for accurate sampling at the data
dimensions encountered in practice.

A substantial convergence theory has been developed for first-order
discretizations of reverse ODEs and SDEs. An early result of De Bortoli \cite{de2022convergence} established quantitative
convergence guarantees in Wasserstein-1 distance under the manifold hypothesis, allowing, in
particular, for target distributions without a density with respect to
Lebesgue measure.
Subsequent works established and sharpened Wasserstein-2 convergence
guarantees under Gaussian, log-concave,
weakly log-concave, or other structural assumptions on the data distribution
\cite{pierretdiffusion,gao2025wasserstein,xixianwasserstein,
silveribeyond,beyler2025convergence,bruno2025wasserstein,meng2025pathway}. Complementary convergence guarantees have also been obtained in total
variation and Kullback--Leibler divergence
\cite{lee2022convergence,chen2022sampling,chen2023improved,benton2024nearly,
conforti2024klconvergenceguaranteesscore}.

In the Wasserstein-2 setting, the convergence rate has reached a sharp first-order benchmark. Under the Gaussian-tail assumption of \cite{xixianwasserstein} or
similar regularity assumptions, Euler-type discretizations of the
reverse flow achieve, up to logarithmic and fixed
target-dependent factors, an exact-score discretization error of order
$
    \sqrt d\,h.
$
Consequently, the number $N$ of time steps required to reach $W_2$ accuracy
$\varepsilon$ scales as
$
    N
    =
    \widetilde O
    \left(
        \sqrt d\,\varepsilon^{-1}
    \right).
$
The corresponding scaling is sharp for first-order probability-flow
discretizations, even for Gaussian targets
\cite{gao2025wasserstein}. Thus, further improvement in
the dependence on the accuracy, and potentially in the polynomial dependence
on the ambient dimension, 
cannot be obtained merely by sharpening the
analysis of the same first-order discretization. 

A natural way to move beyond this first-order barrier is to use higher-order
numerical schemes. On a fixed time interval, a stable numerical method of order $p$, applied to
a sufficiently regular vector field, has a global discretization error of the
form
$
    C_p h^p,
$
where $h$ is the time-step size and $C_p$ depends on higher-order derivatives
and stability properties of the underlying dynamics
\cite{hairer1993solving,kloeden1992numerical}. In particular, establishing an order-$p$ global
error generally requires control of the derivatives entering the local
truncation error, together with spatial regularity sufficient to propagate
the error between successive steps.
For the reverse probability flow, deriving the corresponding
quantitative estimates directly from structural assumptions
on the target distribution is a separate, model-specific problem. In addition, the Gaussian reference law is reached only as
$t\to\infty$ in the OU time scale.
Reducing the initialization error therefore requires increasing
the terminal time $T$, so the analysis must also control the
dependence of regularity and stability bounds on the time horizon.

Convergence analyses of these higher-order diffusion samplers are more
recent and considerably more limited. Such analyses have been developed for
interpolation-based schemes~\cite{li2025faster}, Runge--Kutta methods\cite{huang2025convergence,huang2025fast}, Heun-type methods~\cite{beyler2025convergence}, and
stochastic high-order discretizations~\cite{wu2024stochastic,pfarr2026higherorder}.
In much of this literature, the required regularity is imposed directly as
an assumption on the score, its approximation, or related derivative
quantities, rather than derived from structural properties of the target
distribution. For example, Huang, Huang and Lin \cite{huang2025convergence} analyze $p$-th order
Runge--Kutta schemes under bounds on mixed space--time derivatives of the
estimated score. The later analysis in \cite{huang2025fast} substantially
weakens the regularity required for the learned-score perturbation, reducing
it to the first two spatial derivatives, while the high-order exact-score
discretization branch still relies on higher derivatives of the exact
diffusion field. Such score-level assumptions can be difficult to verify
from properties of $P_0$ and do not generally follow from standard $L^2$
score-estimation guarantees.

Even when the required regularity is available, a second and distinct issue
remains: the prefactor $C_p$ may grow rapidly with the ambient dimension.
Thus, the numerical order alone does not determine whether a higher-order
method becomes genuinely more efficient in high dimension. This phenomenon
is visible in the exact-score Runge--Kutta analyses of
\cite{huang2025convergence,huang2025fast}, where the leading discretization
term in total variation has the form
$
    d(dh)^p
    =
    d^{p+1}h^p.
$
With $p$ and the remaining problem parameters fixed, this corresponds to a
step complexity of order
\(
    \widetilde O
    \left(
        d^{1+1/p}\varepsilon^{-1/p}
    \right).
\)
Higher order therefore improves the dependence on the target accuracy and
reduces part of the dimensional overhead, but the polynomial exponent of
$d$ remains at least $1$ and approaches $1$ as $p$ increases. Consequently,
the numerical order and the dimension dependence of the associated
regularity estimates must be analyzed jointly.

\medskip
Our \textbf{main contribution} is then to establish a priori regularity
estimates at arbitrary order for the exact reverse probability
flow directly from structural assumptions on the target
distribution. As an application, we use these estimates in standard
convergence arguments to obtain high-order error bounds.
The resulting step-complexity bounds have polynomial
dimension exponents that decrease as the numerical
order increases.
\medskip

For the deterministic probability-flow ODE considered here,
high-order discretization requires control of derivatives
along the reverse flow, referred to as \emph{material derivatives}
\cite{dziuk2013finite}.
In the OU setting, the material differentiation along the reverse flow
is represented in the variables $(t,x)$ by
\[
    L
    :=
    -\partial_t
    +
    \frac12
    \bigl(x+D_x\log p_t(x)\bigr)\cdot D_x.
\]
The term $-\partial_t$ accounts for the decrease of forward
time along the reverse trajectory, while the spatial term
differentiates in the direction of the reverse-flow velocity.
Successive applications of $L$ give the time derivatives
entering the Taylor expansion, represented by $L^j x$. In particular, the
local truncation error of an order-$p$ Taylor scheme is governed by
$L^{p+1}x$, while stability and global error propagation additionally require
control of the spatial derivatives $D_xL^j x$. The central regularity problem
is therefore to establish quantitative bounds for these quantities at
arbitrary order. Gaussian smoothing ensures that the density $p_t$ is smooth
for every positive diffusion time, but this qualitative smoothness does not
by itself provide the global-in-space and dimension-explicit estimates
required by the numerical analysis.

Our approach starts from the posterior representation of the
exact score provided by Tweedie's formula.
Writing
$
    \mu_{x,t}:=\mathcal L(X_0\mid X_t=x),
$
repeated spatial and time differentiation
produce expressions involving \emph{centered posterior moments}.
These are conditional expectations, given $X_t=x$, of
products of components of
$X_0-\mathbb E[X_0\mid X_t=x]$.
For example, the second-order centered moment is the
posterior covariance $
    \mathbb E_{\mu_{x,t}}\!\left[
        \left(
            X_0-\mathbb E_{\mu_{x,t}}[X_0]
        \right)^{\otimes2}
    \right].
$
Higher-order Tweedie identities express spatial derivatives of
the score in terms of posterior moments and have been used to
estimate these derivatives from samples
\cite{meng2021estimating}. Here we need iterated material
derivatives, which combine time differentiation, spatial
differentiation, and contraction with the probability-flow velocity.
This setting departs from the classical analysis of ODE and SDE
solvers, which assumes a velocity field with bounded derivatives---on a
compact neighbourhood of the trajectory
\cite{hairer1993solving,kloeden1992numerical}, globally
\cite{talay1990expansion}, or on a periodic domain
\cite{mattingly2010convergence}. The reverse velocity $\tfrac12\hat s$
grows linearly at infinity and $P_t$ has full support, so no such
assumption is available, and the regularity of the field must be
derived from the target distribution itself.
A central obstruction in the analysis is the \emph{normal-direction fluctuation} created by repeated material differentiation
$
    \langle x,X_0-\mathbb E[X_0\mid X_t=x]\rangle.
$
If $P_0$ is supported in $B_R(0)$, a direct compact-support estimate gives
only
$
    \left|
        \left\langle
            x,
            X_0-\mathbb E[X_0\mid X_t=x]
        \right\rangle
    \right|
    \leq
    2R|x|.
$
This estimate is inadequate for stability because the reverse state is not
spatially bounded: the distribution $P_t$ from Gaussian smoothing has full
support. Controlling the expression only through moments of $|x|$ along the
exact flow may suffice for parts of the local consistency analysis, but it
does not provide the global pointwise control needed to compare perturbed
and numerical trajectories. Moreover, repeated use of such a bound would
introduce additional dimension dependence through the typical size
$|x|\asymp\sqrt d$. This is one of the two mechanisms behind the dimension
exponents of existing Runge--Kutta bounds; see
Remark~\ref{rem:dimension-sources} for a detailed discussion. 

To overcome this obstruction, we impose a dimension-uniform
\emph{doubling condition} on the supporting-cap masses of the compactly
supported target $P_0$
(Assumption~\ref{ass:uniform-edge-doubling}).
In each direction, this condition bounds the mass within depth
$2r$ of a supporting hyperplane by a fixed multiple of the mass
within depth $r$, with the same constant for all directions,
depths $r>0$, and ambient dimensions.
The condition allows $P_0$ to be singular and yields moment bounds
for the centered normal projection above that are uniform in $x$
and independent of $d$
(Theorem~\ref{thm:sec-3-normal-fluctuation-two-scale}).
We also establish the converse
(Lemma~\ref{lem:variance-implies-doubling} and
Corollary~\ref{cor:doubling-necessary}): a spatially uniform $L^2$ bound
on the centered normal fluctuation at any fixed positive time forces the
doubling condition, with an explicit constant. For compactly supported
targets, the doubling condition is therefore necessary and sufficient
for such a bound.

Under the doubling condition above, we establish
dimension-independent pointwise estimates for $L^j x$ and
$D_xL^j x$ for every $j\geq1$, with explicit dependence on
time and the support radius
(Theorem~\ref{thm:weighted-estimates-Ljx}). The time weights in these estimates decay exponentially
as $t\to\infty$.
To prove these estimates, we construct an algebraic structure
from centered posterior moments: scalar moment expressions
form a commutative algebra, and vector moment expressions form
a module over it.
The associated \emph{centered posterior-moment normal form},
extended to include $x$ and the posterior mean, is closed
under $L$
(Lemma~\ref{lem:centered-moment-normal-form-Ljx}).
Every explicit occurrence of $x$ in the centered remainder
appears through a normal-direction insertion controlled by
the normal-fluctuation estimate.
Combining this algebraic representation with a weighted
regularity calculus yields the stated bounds.
We also establish the mixed directional derivative estimates
needed for explicit Runge--Kutta schemes
(Lemma~\ref{lem:sec3-weighted-directional-bound}).

As an application of these a priori regularity estimates,
we derive high-order convergence bounds for discretizations
of the exact probability-flow ODE.
Since $P_0$ may be singular, we stop the reverse flow at forward
time $\delta>0$ and compare the numerical output with the
smoothed target $P_\delta$. The exponential time decay in our regularity estimates yields
stability and discretization bounds that are uniform in the
terminal time $T$ for fixed $\delta>0$.
For the order-$p$ truncated Taylor scheme, a coupling and discrete
Gr\"onwall argument gives a $W_2$ discretization error of order
$\sqrt d\,h^p$
(Theorem~\ref{thm:global convergence truncated Taylor}).
In total variation, a comparison of the exact and numerical
transport maps gives a discretization error of order $d\,h^p$
for both the order-$p$ truncated Taylor method and fixed explicit
order-$p$ Runge--Kutta schemes
(Theorem~\ref{thm:weighted-exact-tv-bounds}).
For fixed target and early-stopping time $\delta$, these bounds
give step complexities
\[
    \widetilde O\left(
        d^{1/(2p)}\varepsilon^{-1/p}
    \right)
    \quad\text{in } W_2,
    \qquad
    \widetilde O\left(
        d^{1/p}\varepsilon^{-1/p}
    \right)
    \quad\text{in total variation}.
\]

Table~\ref{tab:intro-dimension-comparison} compares the leading
dimension dependence of these applications with existing
high-order Runge--Kutta bounds. The comparison isolates the exact-score
discretization error and suppresses logarithmic, order-dependent, and fixed
target-dependent factors. In the existing Runge--Kutta bounds, the
dimension exponent is $1+1/p$ and therefore approaches $1$ as $p$ increases.
By contrast, the exponents obtained here are $1/p$ in total variation and
$1/(2p)$ in $W_2$, both of which approach $0$. Thus, at the level of
polynomial dimension dependence, the existing bounds approach linear
dependence on $d$, whereas the bounds obtained here approach $O(1)$
dependence on the ambient dimension (Remark~\ref{rem:dimension-sources}).

We complement these convergence results with numerical experiments
on a two-point target embedded in $\mathbb R^d$, whose score is
available in closed form
(Section~\ref{subsec:numerical-illustration}).
For truncated Taylor schemes of orders $p=1,2,3,4$, the measured
coupled $L^2$ errors are consistent with the predicted
$h^p$ time-step convergence and $\sqrt d$ dependence on the
ambient dimension. Runge--Kutta schemes of the corresponding
orders exhibit similar behavior.
\begin{table}[t]
\centering
\small
\begin{tabular}{lccc}
\hline
method & metric & discretization term & step complexity \\
\hline
\cite{huang2025convergence,huang2025fast}, order-$p$ RK
& TV
& $d^{p+1}h^p$
& $\widetilde O(d^{1+1/p}\varepsilon^{-1/p})$
\\
this work, order-$p$ RK / Taylor
& TV
& $dh^p$
& $\widetilde O(d^{1/p}\varepsilon^{-1/p})$
\\
this work, order-$p$ Taylor
& $W_2$
& $\sqrt dh^p$
& $\widetilde O(d^{1/(2p)}\varepsilon^{-1/p})$
\\
\hline
\end{tabular}
\caption{Leading dimension dependence of the exact-score discretization analysis for
fixed $p$, $R$, $\delta$, and target-regularity constants. Logarithmic factors are
suppressed. The first row records the leading Runge--Kutta discretization dependence
in \cite{huang2025convergence,huang2025fast}; the last two rows are proved in
Sections~\ref{sec:tv-comparison} and~\ref{sec:global-convergence-taylor}, respectively.}
\label{tab:intro-dimension-comparison}
\end{table}

\medskip
The rest of the work is organized as follows.
Section~\ref{sec:Preliminaries} introduces the OU probability flow, the material
derivative, and the truncated Taylor scheme. Section~\ref{sec:sec3-centered-posterior-moment-normal-form}
establishes the normal-direction fluctuation bounds, develops
the algebraic structure and closure of the centered
posterior-moment normal form, and combines these ingredients
with a weighted calculus to obtain the regularity estimates. Section~\ref{sec:high-order-convergence} applies these
regularity estimates to obtain Wasserstein bounds for
Taylor schemes and total-variation bounds for Taylor and
Runge--Kutta schemes, together with the corresponding
step complexities, and concludes with a numerical illustration.
Appendix~\ref{app:edge-doubling} collects sufficient conditions and
examples for the doubling assumption, including a family of targets for
which it fails, and Appendix~\ref{app:recursive-material-representation}
gives a recursive conditional-expectation representation of the material
derivatives $L^jx$.

\section{Preliminaries}
\label{sec:Preliminaries}
This section prescribes the notation for the exact probability flow and its
high-order discretization. We first introduce a general affine Gaussian
perturbation and derive the associated posterior representation of the
score, forward diffusion, and reverse probability flow. We then specialize
to the Ornstein--Uhlenbeck process and define the material derivative $L$.
Finally, we show that the iterated derivatives $L^j x$ are precisely the
Taylor coefficients of the reverse flow and use them to define the
truncated Taylor scheme. These identities identify the regularity
quantities that will be estimated in
Section~\ref{sec:sec3-centered-posterior-moment-normal-form}.

\subsection{Affine Gaussian perturbations and probability flow}
We begin by specifying the one-time marginal distributions of the forward
perturbation. Let $X_0\sim P_0$ and $Z\sim N(0,I_d)$ be independent, and set
\[
X_t
=
\alpha(t)X_0+\beta(t)Z.
\]
We assume that $\alpha(t)$ and $\beta(t)^2$ are continuously differentiable and
satisfy
\[
\alpha(0)=1,
\qquad
\beta(0)=0,
\qquad
\alpha(t)>0,
\qquad
\beta(t)>0
\quad\text{for }t>0.
\]
Thus, the perturbation starts from the target distribution $P_0$. For the
schedules used in diffusion models, $\alpha(t)$ becomes small while
$\beta(t)$ approaches a prescribed terminal scale, so that $P_T$ is close
to the corresponding centered Gaussian distribution for sufficiently
large $T$.

Let $P_t$ denote the law of $X_t$. For every $t>0$, it has the smooth
density
\[
p_t(x)
=
\frac{1}{(2\pi\beta(t)^2)^{d/2}}
\int_{\mathbb R^d}
\exp\left(
-\frac{|x-\alpha(t)y|^2}{2\beta(t)^2}
\right)
\,dP_0(y).
\]
Associated with $(x,t)$, we introduce the posterior probability measure
\[
    \mu_{x,t}(dy)
    =
    \frac{
        \exp\left(
            -\frac{|x-\alpha(t)y|^2}{2\beta(t)^2}
        \right)dP_0(y)
    }{
        \displaystyle
        \int_{\mathbb R^d}
        \exp\left(
            -\frac{|x-\alpha(t)y'|^2}{2\beta(t)^2}
        \right)dP_0(y')
    }.
\]
It is the conditional law of $X_0$ given $X_t=x$. Differentiating the
Gaussian kernel with respect to $x$ gives the posterior-mean representation of the score 
\[
    D_x\log p_t(x)
    =
    -\frac{1}{\beta(t)^2}x
    +
    \frac{\alpha(t)}{\beta(t)^2}
    \mathbb E_{\mu_{x,t}}[Y],\qquad Y\sim \mu_{x,t}.
\]
Repeated spatial differentiation of the posterior mean then gives rise to centered moments of $\mu_{x,t}$, which will serve as the basic building blocks for our derivative estimates.
\paragraph{Forward diffusion and the Fokker--Planck equation.}
The same
family of marginal distributions $(P_t)_{t\in[0,T]}$ can be realized by a
linear diffusion
\[
    dX_t=f(t)X_t\,dt+g(t)\,dW_t,
\]
where
\[
    f(t)=\frac{\alpha'(t)}{\alpha(t)},\qquad
    g(t)^2
    =
    \frac{d}{dt}\beta(t)^2
    -
    2f(t)\beta(t)^2.
\]
Here we assume that the right-hand side defining
$g(t)^2$ is nonnegative.
For $t>0$, the density $p_t$ satisfies the Fokker--Planck equation
\[
    \partial_t p_t
    =
    -\operatorname{div}\bigl(f(t)x\,p_t\bigr)
    +
    \frac{g(t)^2}{2}\Delta p_t.
\]
We can rewrite this equation as
\[
\partial_t p_t+\operatorname{div}(v_tp_t)=0,
\qquad
v_t(x)
=
f(t)x-\frac{g(t)^2}{2}D_x\log p_t(x).
\]
This continuity-equation representation identifies a deterministic velocity
field that transports the same one-time marginals as the forward
diffusion.

\paragraph{Reverse-time dynamics and probability flow.}
We now traverse the family $(P_t)$ in the reverse direction, from $P_T$
toward the target distribution. To accommodate target distributions that
may be singular, we fix a terminal forward time $\delta\in[0,T)$ and
consider the reverse dynamics over the forward-time interval $[\delta,T]$.
When the score and the probability flow extend regularly to time zero, we
may take $\delta=0$ and recover $P_0$ exactly. Otherwise, we take
$\delta>0$ and stop at the smoothed distribution $P_\delta$. Set
\[
    \widetilde{X}_u=X_{T-u},
    \qquad
    u\in[0,T-\delta].
\]
Under the usual regularity assumptions, the reverse-time process
satisfies~\cite{And_80,haussmann1986time},
\[
    d\widetilde{X}_u
    =
    \left[
        -f(T-u)\widetilde{X}_u
        +
        g(T-u)^2D_x\log p_{T-u}(\widetilde{X}_u)
    \right]du
    +
    g(T-u)\,d\overline W_u.
\]
When initialized with $\widetilde X_0\sim P_T$, it satisfies
$\widetilde X_u\sim P_{T-u}$ and therefore reaches $P_\delta$ at
$u=T-\delta$.

The same family of marginals can also be transported deterministically.
In forward time, the probability-flow ODE is
\begin{equation*}
\frac{d}{dt}\overline X_t
=
f(t)\overline X_t
-
\frac{g(t)^2}{2}
D_x\log p_t(\overline X_t).
\end{equation*}
Its transport equation is precisely the continuity equation derived above.
For generative sampling, the ODE is initialized at time $T$ and integrated
backward. Equivalently, setting
\begin{equation*}
Y_u=\overline X_{T-u},
\qquad
0\leq u\leq T-\delta,
\end{equation*}
gives
\begin{equation*}
\frac{d}{du}Y_u
=
-f(T-u)Y_u
+
\frac{g(T-u)^2}{2}
D_x\log p_{T-u}(Y_u).
\end{equation*}
If $Y_0\sim P_T$, then $Y_u\sim P_{T-u}$. The reverse-time diffusion and
the reverse probability-flow ODE therefore have the same one-time
marginals, while their score coefficients differ by a factor of two. The
probability flow is deterministic once its initial state is fixed.

In practical samplers, the exact score is replaced by a learned
approximation. Here we work with the exact score and study the regularity
and discretization of the resulting deterministic probability flow.

\paragraph{Reduction to the Ornstein--Uhlenbeck normalization.}
Let $c(t):=\sqrt{\alpha(t)^2+\beta(t)^2}$ and
$\tau(t):=\log\bigl(1+\beta(t)^2/\alpha(t)^2\bigr)$. Then
\begin{equation*}
    X_t\overset{d}{=}c(t)\,\widetilde X_{\tau(t)},
    \qquad
    \widetilde X_\tau:=e^{-\tau/2}X_0+\sqrt{1-e^{-\tau}}\,Z,
\end{equation*}
where $\widetilde X$ is the Ornstein--Uhlenbeck process started from
$P_0$, whose marginals we denote by $\widetilde P_\tau$ with densities
$\tilde p_\tau$. Equivalently,
$P_t=(c(t)\,\cdot)_\#\widetilde P_{\tau(t)}$ and
\begin{equation*}
    D_x\log p_t(x)
    =
    c(t)^{-1}D_x\log\tilde p_{\tau(t)}\bigl(x/c(t)\bigr).
\end{equation*}
Whenever $\beta/\alpha$ is increasing, $\tau$ is a change of time, and
the probability-flow ODE of the schedule $(\alpha,\beta)$ is the image
of the Ornstein--Uhlenbeck reverse heat flow under the linear
space--time transformation $(t,x)\mapsto(\tau(t),x/c(t))$. The
regularity estimates of
Section~\ref{sec:sec3-centered-posterior-moment-normal-form} are
therefore expected to transfer to every affine Gaussian schedule, with
constants depending on the schedule through $c$, $\tau$, and their
derivatives up to the order considered; we do not pursue this here.
For this reason we work with the Ornstein--Uhlenbeck normalization from
now on.

\subsection{The reverse probability flow for the OU process}
\label{subsec:reverse probability flow for the OU process}
By the reduction above, it suffices to treat the
Ornstein--Uhlenbeck process
\begin{equation}
    \label{eq:standardOU}
     dX_t=-\frac12X_t\,dt+dW_t,\qquad
    t\in[0,T].
\end{equation}
Its one-time marginals admit the representation
$
X_t
\overset{d}{=}
e^{-t/2}X_0+\sqrt{1-e^{-t}}Z.
$ Thus, in the notation of the preceding
subsection,
\begin{equation*}
\alpha(t)=e^{-t/2},
\qquad
\beta(t)=\sqrt{1-e^{-t}},
\qquad
f(t)=-\frac12,
\qquad
g(t)=1.
\end{equation*}
Accordingly, the score identity becomes
\begin{equation*}
    D_x\log p_t(x)
    =
    -\frac{1}{1-e^{-t}}x
    +
    \frac{e^{-t/2}}{1-e^{-t}}\mathbb E_{\mu_{x,t}}[Y].
\end{equation*}
For the OU process, it is natural to introduce the modified score
$
\hat s(t,x)
:=
x+D_x\log p_t(x).
$
Indeed, the forward probability-flow velocity is
$-\frac12\hat s(t,x)$ and the exact reverse probability flow satisfies
\begin{equation}\label{eq:backward-ode}
        \frac{d}{du}Y_u
=
\frac12\hat s(T-u,Y_u),
\qquad
Y_0\sim P_T,
\qquad
0\leq u\leq T-\delta.
\end{equation}

\begin{remark}[Reverse heat flow]
\label{rem:reverse-heat-flow}
Let $\gamma_d=N(0,I_d)$ and let $(S_\tau)_{\tau\geq0}$ denote the
Ornstein--Uhlenbeck semigroup with generator $\Delta-x\cdot D_x$,
which is the heat flow relative to $\gamma_d$ used in
\cite{kim2012generalization}. The process
\eqref{eq:standardOU} is this diffusion run at half speed, so
$P_t=P_0S_{t/2}$, and for $t>0$
\begin{equation*}
    \hat s(t,x)
    =
    x+D_x\log p_t(x)
    =
    D_x\log\frac{dP_t}{d\gamma_d}(x).
\end{equation*}
Thus $-\frac12\hat s$ is the velocity field transporting the
relative densities $dP_t/d\gamma_d$ along the heat flow, and
\eqref{eq:backward-ode} is the reverse heat-flow transport of
\cite{kim2012generalization} from $P_T$ toward $P_\delta$, up to the
time change $\tau=t/2$. Equivalently, by
$X_t=e^{-t/2}\bigl(X_0+\sqrt{e^t-1}\,Z\bigr)$, the law $P_t$ is the
Euclidean heat flow of $P_0$ at time $e^t-1$ followed by the dilation
$x\mapsto e^{-t/2}x$. We use the names \emph{reverse heat flow} and
\emph{reverse probability flow} interchangeably.
\end{remark}

We emphasize the roles of the two time variables. The variable $t$ denotes
forward diffusion time, while $u$ denotes increasing reverse time. Along
the reverse trajectory, the corresponding forward time is $t=T-u$ and
therefore decreases as $u$ increases. Hence, for a smooth function
$\varphi(t,x)$ written in the forward-time variables, the chain rule gives
\begin{equation*}
\begin{aligned}
\frac{d}{du}\varphi(T-u,Y_u)
&=
-\partial_t\varphi(T-u,Y_u)
+
D_x\varphi(T-u,Y_u)
\left[
\frac{d}{du}Y_u
\right]
\\&=\left[
-\partial_t\varphi
+
\frac12D_x\varphi
\bigl[\hat s(t,x)\bigr]
\right](T-u,Y_u).
\end{aligned}
\end{equation*}
This leads to the material derivative
\begin{equation}
    L
    :=
    -\partial_t
    +
    \frac12\hat s(t,x)\cdot D_x.
\end{equation}
Although $L$ is written in the forward-time variables $(t,x)$, it
represents differentiation along the reverse trajectory
$u\mapsto(T-u,Y_u)$; since $L$ is the generator of the probability-flow
ODE, we also call it the \emph{flow derivative}.
Successive applications of $L$ therefore represent repeated
differentiation in reverse time and lead directly to the
Taylor expansion developed in the next subsection.

\subsection{Taylor expansion along the reverse flow}

For every sufficiently smooth vector- or tensor-valued function $\varphi(t,x)$, iterating the chain rule gives
\begin{equation*}
    \frac{d^j}{du^j}\varphi(T-u,Y_u)
    =
    (L^j\varphi)(T-u,Y_u), \qquad j\in \N_+.
\end{equation*}
In particular,
\begin{equation}
    \frac{d^j}{du^j}Y_u
    =
    (L^jx)(T-u,Y_u),
    \qquad
    j\in \N_+.
\end{equation}
Consequently, the time derivatives entering the Taylor expansion of the
exact reverse flow are completely characterized by the vector fields
$L^m x$. The main task is therefore to control these iterated flow
derivatives uniformly along the reverse trajectory.

\begin{proposition}[Taylor expansion along the backward flow]\label{prop:taylor-flow}
Let \(p\ge 0\), and let \(\varphi(t,x)\) be smooth enough so that \(L^{p+1}\varphi\) is well-defined. Then for every \(u\ge 0\) and \(h>0\) with \(u+h\le T-\delta\),
\begin{equation}
\varphi(T-(u+h),Y_{u+h})
=
\sum_{j=0}^{p}\frac{h^j}{j!}(L^j \varphi)(T-u,Y_u)
+
R_{p+1}^\varphi(u,h),
\end{equation}
where the remainder is given by
\begin{equation}
R_{p+1}^\varphi(u,h)
=
\int_0^h \frac{(h-r)^p}{p!}\,
(L^{p+1}\varphi)(T-(u+r),Y_{u+r})\,dr.
\end{equation}
\end{proposition}
Applying Proposition~\ref{prop:taylor-flow} to the identity map gives
\begin{equation}\label{eq:taylor-x}
\begin{aligned}
    Y_{u+h}
&=
Y_u+\sum_{j=1}^{p}\frac{h^j}{j!}(L^j x)(T-u,Y_u)
+
R_{p+1}^x(u,h),\\
R_{p+1}^x(u,h)
&=
\int_0^h \frac{(h-r)^p}{p!}\,
(L^{p+1}x)(T-(u+r),Y_{u+r})\,dr.
\end{aligned}
\end{equation}
The local remainder therefore depends only on the next material derivative
$L^{p+1}x$.

\paragraph{The $p$-th order truncated Taylor scheme.}
Let
$
u_n=nh
$
denote the reverse-time grid, and let
$
t_n:=T-u_n
$
be the corresponding forward diffusion time.
Motivated by \eqref{eq:taylor-x}, we define the exact-score truncated \(p\)-th order Taylor scheme by
\begin{equation}\label{eq:taylor-scheme}
\begin{aligned}
Y_{n+1}^h
&:=
\Phi_h^{(p)}(u_n,Y_n^h),
\qquad
Y_0^h\sim N(0,I_d),
\\
\Phi_h^{(p)}(u,x)
&:=
x+
\sum_{j=1}^p
\frac{h^j}{j!}
(L^jx)(T-u,x).
\end{aligned}
\end{equation}
For the exact solution, the one-step defect is exactly
$R_{p+1}^x(u_n,h)$.
Stability of the Taylor map instead depends on the spatial
derivative:
\begin{equation*}
D_x\Phi_h^{(p)}(u,x)
=
I+
\sum_{j=1}^p
\frac{h^j}{j!}
D_x(L^jx)(T-u,x).
\end{equation*}
At this point, the numerical problem reduces to a regularity problem.
The local truncation error is controlled by $L^{p+1}x$, whereas stability
of the truncated Taylor map requires pointwise bounds on $D_x(L^j x)$.
Thus the two quantities to be controlled, at arbitrary order, are
\begin{equation*}
    L^j x
    \qquad\text{and}\qquad
    D_x(L^j x).
\end{equation*}
Directly expanding these iterated derivatives quickly becomes unwieldy.
Instead, Section~\ref{sec:sec3-centered-posterior-moment-normal-form} develops a common algebraic normal form consisting of
deterministic affine terms and centered posterior moments. The class is
closed under the material derivative $L$, and a weighted bookkeeping
argument then yields quantitative bounds for both families above.

For completeness, we note that the material derivatives also admit a
recursive conditional-expectation representation under the forward Gaussian
coupling, in a spirit similar to the higher-order Tweedie identities for
spatial score derivatives \cite{meng2021estimating}.
This representation provides a starting point for designing squared-error
learning objectives, with recursively defined quantities serving as
regression targets for material derivatives at successive orders.
While this representation is not used in the regularity or convergence
analysis, we record it separately in
Appendix~\ref{app:recursive-material-representation}.

\section{Target-derived regularity of the probability-flow ODE}
\label{sec:sec3-centered-posterior-moment-normal-form}

This section establishes the quantitative regularity estimates for the
reverse heat flow needed for the high-order numerical analysis. The
case $j=1$ follows from the posterior-covariance representation of
$D_x\hat s$ and compact support alone; the difficulty, and the content of
this section, lies in $j\geq2$, where the centered factor
$\langle Z,x\rangle$ first appears. The
argument separates into three steps. In Section~\ref{subsec:sec3.1-normal-fluctuation-centered-insertion}, we
control the normal-direction posterior fluctuation that appears when the
ambient variable $x$ enters a centered posterior moment. Then, in Section~\ref{subsec:sec3.2-algebraic-normal-form-closure}, we
construct a centered posterior-moment normal form that is closed under the
material derivative $L$. Finally, in Section~\ref{subsec:subsec3.3-Weighted estimates for the centered normal form} we combine this algebraic closure with
the fluctuation estimate through a weighted bookkeeping argument to obtain
bounds for $L^j x$ and $D_x(L^j x)$ at arbitrary order. 

Throughout this section, $t$ denotes forward diffusion time, consistently
with Section~\ref{subsec:reverse probability flow for the OU process}. All
posterior expectations are taken with respect to $\mu_{x,t}$. To simplify
the notation, within this section we write
\begin{equation*}
\bar y_{x,t}
:=
\mathbb E_{\mu_{x,t}}[Y],
\qquad
Z
:=
Y-\bar y_{x,t}.
\end{equation*}
For the algebraic bookkeeping below, we decompose the modified score as
\begin{equation*}
\hat s(t,x)
=
a(t)x+\lambda(t)\bar y_{x,t},
\qquad
a(t)
:=
-\frac{e^{-t}}{1-e^{-t}},
\qquad
\lambda(t)
:=
\frac{e^{-t/2}}{1-e^{-t}}.
\end{equation*}
This decomposition separates the explicit dependence on the ambient
variable $x$ from the posterior contribution. It is the starting point for
the constructed  centered normal form.

\subsection{Normal-direction fluctuation bounds for centered insertions}
\label{subsec:sec3.1-normal-fluctuation-centered-insertion}
Repeated material differentiation introduces factors of the form
$\langle Z,x\rangle$ into the posterior-moment expressions. A direct compact-support estimate would give a bound
of order $R|x|$, which is insufficient for the spatially uniform estimates needed
later. The purpose of this subsection is to show that the posterior exponential
tilt, together with the doubling condition introduced below, removes this
growth in $|x|$. More precisely, under the geometric
condition introduced below, we prove
\[
\mathbb E_{\mu_{x,t}}
\bigl[
|\langle Z,x\rangle|^q
\bigr]
\leq
C_q
\left(
e^{-qt/2}R^{2q}
+
e^{qt/2}
\right),
\qquad q\geq1,
\]
for targets supported in a ball of radius $R$. This estimate is the only place where the geometry of $P_0$ enters
directly into the subsequent high-order regularity analysis.

To formulate the assumption, suppose that $P_0$ has compact support
\begin{equation*}
M:=\operatorname{supp}P_0.
\end{equation*}
For $\theta\in\mathbb S^{d-1}$, define
\begin{equation*}
\begin{aligned}
U_\theta(y)
&:=
\sup_{y'\in M}\langle\theta,y'\rangle
-
\langle\theta,y\rangle,
\
F_\theta(u)
&:=
P_0\bigl(U_\theta(Y)\leq u\bigr),
\qquad u>0.
\end{aligned}
\end{equation*}
Here $U_\theta(y)$ is the depth of $y$ below the supporting hyperplane of
$M$ with outward normal $\theta$, and $F_\theta(u)$ is the $P_0$-mass of the
corresponding supporting cap of thickness $u$. Since
$M=\operatorname{supp}P_0$, one has $F_\theta(u)>0$ for every $u>0$.

\begin{assumption}[Dimension-uniform doubling]
\label{ass:uniform-edge-doubling}
There exists a constant $D_{\mathrm{UD}}\geq1$ such that
\begin{equation}
    F_\theta(2u)
    \leq
    D_{\mathrm{UD}}F_\theta(u),
    \qquad
    \forall \theta\in\mathbb S^{d-1},
    \quad \forall u>0.
\end{equation}
When a family of target distributions $P_0^{(d)}$ in different ambient
dimensions is considered, the same constant $D_{\mathrm{UD}}$ is required
to work for every $d$.
\end{assumption}
Geometrically, the assumption ensures that halving the thickness
of a supporting cap reduces its mass by at most a fixed factor. It is formulated directly in terms of cap mass and therefore also
allows singular targets. Examples and sufficient geometric conditions are
collected in Appendix~\ref{app:edge-doubling}.

The normal-direction bound follows from the following one-dimensional
estimate, which gives a uniform bound on the rescaled centered fluctuation.

\begin{lemma}[Centered exponential-tilt estimate]
\label{lem:centered-exponential-tilt-doubling}
Let $\nu$ be a finite nonzero measure supported on $[0,L]$. Set
$
    F_\nu(u):=\nu([0,u]),
$
and suppose that
\begin{equation*}
    F_\nu(2u)\leq D F_\nu(u),
    \qquad u>0,
\end{equation*}
for some $D\geq1$. For $\eta>0$, define the exponentially tilted
probability measure
\begin{equation*}
    \Pi_\eta(du)
    :=
    \frac{e^{-\eta u}\nu(du)}
    {\int e^{-\eta v}\nu(dv)}.
\end{equation*}
Let $U$ be a random variable with law $\Pi_\eta$. Then, for every $q\geq1$,
\begin{equation}\label{eq:centered exponential tilt estimate}
    \left\lVert
        \eta\bigl(U-\mathbb E_{\Pi_\eta}U\bigr)
    \right\rVert_{L^q(\Pi_\eta)}
    \leq
    4\left(
        \log D+\Gamma(q+1)^{1/q}
    \right),
\end{equation}
where $\Gamma$ denotes the Gamma function,
\(
\Gamma(s):=\int_0^\infty r^{s-1}e^{-r}dr.
\)

\end{lemma}

\begin{proof}
Let
$
    Z_\nu(\eta)
    :=
    \int e^{-\eta u}\nu(du).
$
Extending $F_\nu$ constantly beyond $L$ and integrating by parts gives
\begin{equation*}
    Z_\nu(\eta)
    =
    \eta\int_0^\infty e^{-\eta u}F_\nu(u)\,du.
\end{equation*}
After a change of variables, the doubling condition yields
\begin{equation*}
\begin{aligned}
    Z_\nu(\eta/2)
    &=
    \eta\int_0^\infty e^{-\eta v}F_\nu(2v)\,dv
    \\
    &\leq
    D\eta\int_0^\infty e^{-\eta v}F_\nu(v)\,dv
    =
    DZ_\nu(\eta).
\end{aligned}
\end{equation*}
It follows that, for every $z\geq0$,
\begin{equation*}
\begin{aligned}
    \Pi_\eta(\eta U\geq z)&=\frac{\int_{z/\eta}^\infty e^{-\eta u}\,\nu (du)}{Z_\nu(\eta)}\\
    &\leq e^{-z/2}\frac{\int_{z/\eta}^\infty e^{-\frac{\eta u}{2}}\,\nu (du)}{Z_\nu(\eta)}\leq
    e^{-z/2}
    \frac{Z_\nu(\eta/2)}{Z_\nu(\eta)}
\leq
    De^{-z/2}.
\end{aligned}
\end{equation*}
Thus $X:=\eta U$ is stochastically dominated by
$2\log D+E$, where $E\sim\operatorname{Exp}(1/2)$ is the exponential distribution with rate $1/2$. 
Since
$
    \lVert{}E\rVert_{L^q}
    =
    2\Gamma(q+1)^{1/q},
$
we obtain
\begin{equation*}
    \lVert{}X\rVert_{L^q}
    \leq
    2\log D+\lVert{}E\rVert_{L^q}
    =
    2\left(
        \log D+\Gamma(q+1)^{1/q}
    \right).
\end{equation*}
The estimate~\eqref{eq:centered exponential tilt estimate} now follows from
$\lVert{}X-\mathbb EX\rVert_{L^q}\leq2\lVert{}X\rVert_{L^q}$.
\end{proof}
We now apply this one-dimensional estimate to the projected depth $U_\theta(y)$. The posterior contains one additional radial Gaussian weight.
On bounded support, however, the
oscillation of this additional weight changes the doubling constant by only
a controlled factor.

\begin{theorem}[Normal-direction fluctuation bound]
\label{thm:sec-3-normal-fluctuation-two-scale}
Suppose that $P_0$ is supported in $B_R(0)$ and satisfies
Assumption~\ref{ass:uniform-edge-doubling}. Define
\begin{equation*}
    \Delta_0
    :=
    \sup_{y\in M}|y|^2-
    \inf_{y\in M}|y|^2
    \leq R^2.
\end{equation*}
Then, for every $q\geq1$, $t>0$, and 
$x\in\mathbb R^d$,
\begin{equation}\label{eq:normal fluctuation sharper}
\begin{aligned}
    &\left\lVert
        \langle x,Y-\bar y_{x,t}\rangle
    \right\rVert_{L^q(\mu_{x,t})}\leq
    4e^{-t/2}
    \left[
        (e^t-1)
        \left(
            \log D_{\mathrm{UD}}
            +\Gamma(q+1)^{1/q}
        \right)
        +\frac{\Delta_0}{2}
    \right].
\end{aligned}
\end{equation}
In particular, there exists a constant $C_q=C(q,D_{\mathrm{UD}})$ such that,
\begin{equation}\label{eq:normal fluctuation}
    \mathbb E_{\mu_{x,t}}
    \left[
        |\langle x,Y-\bar y_{x,t}\rangle|^q
    \right]
    \leq
    C_q
    \left(
        e^{qt/2}
        +
        e^{-qt/2}R^{2q}
    \right).
\end{equation}
\end{theorem}

\begin{proof}
The assertion is immediate when $x=0$. Suppose that $x\neq0$, and write
\begin{equation*}
    x=r\theta,
    \qquad
    r:=|x|,
    \qquad
    \theta:=\frac{x}{|x|},
    \qquad
    s_t:=e^t-1, \qquad
    \eta_{x,t}=\frac{e^{t/2}r}{s_t}.
\end{equation*}
Expanding the square in the posterior density gives
\begin{equation*}
\begin{aligned}
    \mu_{x,t}(dy)
\propto
    e^{-\eta_{x,t}U_\theta(y)}
    e^{-|y|^2/(2s_t)}P_0(dy).
\end{aligned}
\end{equation*}
Let $\widetilde\nu_{\theta,t}$ be the pushforward under $U_\theta$ of the
finite measure
$
    e^{-|y|^2/(2s_t)}P_0(dy),
$
and denote its distribution function by
\begin{equation*}
    \widetilde F_{\theta,t}(u)
    :=
    \int_{\{U_\theta(y)\leq u\}}
    e^{-|y|^2/(2s_t)}P_0(dy).
\end{equation*}
The oscillation of the Gaussian weight on $M$ gives
\begin{equation*}
    e^{-\sup_M|y|^2/(2s_t)}F_\theta(u)
    \leq
    \widetilde F_{\theta,t}(u)
    \leq
    e^{-\inf_M|y|^2/(2s_t)}F_\theta(u).
\end{equation*}
Therefore Assumption~\ref{ass:uniform-edge-doubling} implies
\begin{equation}
    \widetilde F_{\theta,t}(2u)
    \leq
    D_{\mathrm{UD}}
    \exp\left(
        \frac{\Delta_0}{2s_t}
    \right)
    \widetilde F_{\theta,t}(u).
\end{equation}
Under $\mu_{x,t}$, the random variable $U_\theta(Y)$ is precisely the
$\eta_{x,t}$-exponential tilt of $\widetilde\nu_{\theta,t}$. Applying
Lemma~\ref{lem:centered-exponential-tilt-doubling} therefore gives
\begin{equation*}
\begin{aligned}
    &\left\lVert
        \eta_{x,t}
        \left(
            U_\theta(Y)
            -\mathbb E_{\mu_{x,t}}U_\theta(Y)
        \right)
    \right\rVert_{L^q(\mu_{x,t})}
\leq
    4\left[
        \log D_{\mathrm{UD}}
        +\frac{\Delta_0}{2s_t}
        +\Gamma(q+1)^{1/q}
    \right].
\end{aligned}
\end{equation*}
Finally,
\begin{equation*}
\begin{aligned}
    \langle x,Y-\bar y_{x,t}\rangle
    =
    -r
    \left(
        U_\theta(Y)
        -
        \mathbb E_{\mu_{x,t}}U_\theta(Y)
    \right),\qquad
    r
    =
    s_te^{-t/2}\eta_{x,t}.
\end{aligned}
\end{equation*}
This proves \eqref{eq:normal fluctuation sharper}. Since
$s_te^{-t/2}=e^{t/2}-e^{-t/2}\leq e^{t/2}$ and
$\Delta_0\leq R^2$, \eqref{eq:normal fluctuation} follows.
\end{proof}
The key feature of \eqref{eq:normal fluctuation} is its uniformity in $x$. It allows the explicit ambient-variable insertion $\langle Z,x\rangle$
to be estimated in
the same way as the other centered posterior factors in the normal form
constructed in the next subsection.

\paragraph{Necessity of the doubling condition.}
Assumption~\ref{ass:uniform-edge-doubling} is not merely a sufficient
condition: a uniform bound on the rescaled variance also implies a
doubling condition. We first prove this for one-dimensional exponential
tilts and then transfer it to $P_0$.
\begin{lemma}[Doubling from a uniform rescaled variance bound]
\label{lem:variance-implies-doubling}
Let $\nu$ be a finite nonzero measure supported on $[0,L]$,
with $0\in\operatorname{supp}\nu$. For $\eta>0$, let
$\Pi_\eta(du)\propto e^{-\eta u}\nu(du)$ and let $U$ have
law $\Pi_\eta$. Suppose that
\begin{equation*}
    A_2
    :=
    \sup_{\eta>0}
    \eta^2\operatorname{Var}_{\Pi_\eta}(U)
    <\infty.
\end{equation*}
Then $F_\nu(u):=\nu([0,u])$ satisfies
\begin{equation}\label{eq:variance implies doubling}
    F_\nu(2u)
    \leq
    2e^{4A_2}F_\nu(u),
    \qquad u>0.
\end{equation}
If $A_2=0$, the doubling constant can be taken to be one.
\end{lemma}

\begin{proof}
If $A_2=0$, then $\nu$ is concentrated at zero, and the
conclusion is immediate. Assume that $A_2>0$.
Writing $m(\eta):=\mathbb E_{\Pi_\eta}U$, differentiation gives
\begin{equation*}
    m'(\eta)
    =
    -\operatorname{Var}_{\Pi_\eta}(U).
\end{equation*}
Since $0\in\operatorname{supp}\nu$, one has
$m(\eta)\to0$ as $\eta\to\infty$. Hence
\begin{equation*}
    m(\eta)
    =
    \int_\eta^\infty
    \operatorname{Var}_{\Pi_s}(U)\,ds
    \leq
    \frac{A_2}{\eta}.
\end{equation*}
For $u>0$, choose $\eta=2A_2/u$. Markov's inequality yields
$\Pi_\eta(U>u)\leq m(\eta)/u\leq1/2$. Therefore, with
$Z_\nu(\eta):=\int e^{-\eta v}\nu(dv)$,
\begin{equation*}
\begin{aligned}
    F_\nu(u)
    &\geq
    \int_{[0,u]}e^{-\eta v}\nu(dv)
    \geq
    \frac12 Z_\nu(\eta),
    \\
    F_\nu(2u)
    &\leq
    e^{2\eta u}
    \int_{[0,2u]}e^{-\eta v}\nu(dv)
    \leq
    e^{4A_2}Z_\nu(\eta).
\end{aligned}
\end{equation*}
Combining these inequalities proves
\eqref{eq:variance implies doubling}.
\end{proof}

\begin{corollary}[Doubling is necessary for uniform normal fluctuations]
\label{cor:doubling-necessary}
Suppose that $P_0$ is supported in $B_R(0)$ and that, for some fixed
$t>0$,
\begin{equation*}
    C_t
    :=
    \sup_{x\in\mathbb R^d}
    \left\lVert
        \langle x,Y-\bar y_{x,t}\rangle
    \right\rVert_{L^2(\mu_{x,t})}
    <\infty.
\end{equation*}
Then Assumption~\ref{ass:uniform-edge-doubling} holds with
\begin{equation*}
    D_{\mathrm{UD}}
    =
    2\exp\left(
        \frac{4e^t C_t^2}{(e^t-1)^2}
        +
        \frac{\Delta_0}{2(e^t-1)}
    \right),
\end{equation*}
where $\Delta_0\le R^2$ is as in
Theorem~\ref{thm:sec-3-normal-fluctuation-two-scale}. Consequently,
for compactly supported targets, the doubling condition is necessary
and sufficient for a spatially uniform $L^2$ normal-fluctuation bound
at any fixed positive time. For families of targets with uniformly
bounded support radii and fluctuation constants $C_t$, the equivalence
preserves dimension-uniformity.
\end{corollary}

\begin{proof}
Recall $s_t=e^t-1$ and $\eta_{x,t}=e^{t/2}r/s_t$ for $x=r\theta$.
For the weighted projected measure $\widetilde\nu_{\theta,t}$ used in
the proof of Theorem~\ref{thm:sec-3-normal-fluctuation-two-scale},
\begin{equation*}
\begin{aligned}
    \eta_{x,t}^2
    \operatorname{Var}_{\mu_{x,t}}\bigl(U_\theta(Y)\bigr)
    =
    \frac{e^t}{s_t^2}
    \left\lVert
        \langle x,Y-\bar y_{x,t}\rangle
    \right\rVert_{L^2(\mu_{x,t})}^2&\leq
    \frac{e^t C_t^2}{s_t^2}.
\end{aligned}
\end{equation*}
As $r$ ranges over $(0,\infty)$, $\eta_{x,t}$ ranges over all
positive tilt parameters. Applying
Lemma~\ref{lem:variance-implies-doubling} and transferring the
resulting doubling bound through the Gaussian weight gives
\begin{equation*}
    F_\theta(2u)
    \leq
    2\exp\left(
        \frac{4e^t C_t^2}{(e^t-1)^2}
        +
        \frac{\Delta_0}{2(e^t-1)}
    \right)
    F_\theta(u),
    \qquad
    \theta\in\mathbb S^{d-1},\quad u>0,
\end{equation*}
which is the claimed constant. The equivalence follows by combining
this with Theorem~\ref{thm:sec-3-normal-fluctuation-two-scale}, whose
constants depend only on $q$, $R$, $t$, and $D_{\mathrm{UD}}$.
\end{proof}

\subsection{Algebraic normal form and closure}
\label{subsec:sec3.2-algebraic-normal-form-closure}

We now organize the terms generated by repeated applications of $L$ into
a class built from centered posterior moments. We first define this class
and then prove that it is preserved by the time and spatial derivatives
appearing in $L$. At this algebraic stage, we allow arbitrary smooth
time-dependent coefficients in
$
\mathcal C:=C^\infty((0,\infty)).
$
The resulting expressions will be estimated in the next subsection using
the normal-direction fluctuation bound.

\paragraph{Centered atoms.}
We construct the admissible vector insertions recursively, starting from
the posterior mean $\bar y_{x,t}$. Given previously constructed admissible
insertions $U_1,\ldots,U_\ell$, define the vector and scalar atoms
\begin{equation*}
\begin{aligned}
\mathcal M_{k,m}[U_1,\ldots,U_\ell]
&:=
\mathbb E_{\mu_{x,t}}
\left[
    Z
    \langle Z,x\rangle^k
    |Z|^m
    \prod_{i=1}^{\ell}
    \langle Z,U_i\rangle
\right],
\\
\mathcal S_{k,m}[U_1,\ldots,U_\ell]
&:=
\mathbb E_{\mu_{x,t}}
\left[
    \langle Z,x\rangle^k
    |Z|^m
    \prod_{i=1}^{\ell}
    \langle Z,U_i\rangle
\right].
\end{aligned}
\end{equation*}
where $k,\ell\in\mathbb N_0$ and $m\in2\mathbb N_0$.
Each insertion $U_i=U_i(t,x)$ is held fixed in the outer posterior
expectation, even when it is itself defined by a posterior expectation.
When $\ell=0$, the product is interpreted as $1$, and we write
$[\varnothing]$ for the empty insertion list.
Every vector atom is itself an admissible vector insertion, and only
expressions obtained after finitely many recursive steps are allowed.

To record this recursion, we denote the \emph{insertion depth} of the vector and scalar atoms. More precisely, we set
$\operatorname{dep}(\bar y_{x,t})=0$ and
\begin{equation*}
\operatorname{dep}
\left(
    \mathcal M_{k,m}[U_1,\ldots,U_\ell]
\right)
=
\operatorname{dep}
\left(
    \mathcal S_{k,m}[U_1,\ldots,U_\ell]
\right)
:=
1+\max_{1\leq i\leq\ell}\operatorname{dep}(U_i),
\end{equation*}
where the maximum over the empty set is $-1$. Thus an atom without inserted
vectors has depth zero. The restriction $m\in2\mathbb N_0$ guarantees that
differentiating $|Z|^m$ produces only nonnegative smooth powers of $|Z|$.

Let $\mathcal A$ be the commutative $\mathcal C$-algebra generated by the
scalar atoms, and let $\mathfrak M$ be the $\mathcal A$-module generated by
the vector atoms. Their elements are finite sums of expressions of the
forms
\begin{equation*}
    c(t)\prod_{\nu=1}^N\mathcal S_\nu,
    \qquad
    c(t)
    \left(
        \prod_{\nu=1}^N\mathcal S_\nu
    \right)
    \mathcal M,
    \qquad
    c\in\mathcal C,
\end{equation*}
respectively, with the empty scalar product interpreted as $1$.
We call $\mathfrak M$ the centered posterior moment normal-form class.
The fields $x$ and $\bar y_{x,t}$ do not in general belong to
$\mathfrak M$, but both already appear in
$Lx=\frac12(a(t)x+\lambda(t)\bar y_{x,t})$. We therefore introduce the
extended class
\begin{equation*}
    \mathcal G
    :=
    \mathcal Cx
    +
    \mathcal C\bar y_{x,t}
    +
    \mathfrak M.
\end{equation*}

\begin{remark}
The normal-form representation need not be unique. For example,
$\mathcal S_{0,0}[\varnothing]=1$ and
$\mathcal M_{0,0}[\varnothing]=0$, and the atoms are invariant under
permutations of their inserted vectors. We omit these constant and zero
atoms from the generating sets.
\end{remark}

The key structural point is that every explicit occurrence of the ambient
variable $x$ inside an atom appears through the centered insertion
$\langle Z,x\rangle$. These are precisely the factors controlled by the
normal-direction estimate in
Theorem~\ref{thm:sec-3-normal-fluctuation-two-scale}. The next step is to verify that this recursive grammar is preserved by the two operations entering the material derivative $L$: time differentiation and spatial differentiation along an admissible vector field.

\paragraph{Posterior differentiation identities.}
Let $v=v(t,x)$ be a vector field, held fixed in the posterior integral.
For a sufficiently regular scalar- or vector-valued function
$F=F(t,x,Y)$, differentiating the normalized posterior expectation gives
\begin{equation*}
D_x\mathbb E_{\mu_{x,t}}[F][v]
=
\mathbb E_{\mu_{x,t}}[D_xF[v]]
+
\lambda(t)
\mathbb E_{\mu_{x,t}}
\left[
F\langle Z,v\rangle
\right].
\end{equation*}
The first term differentiates the integrand at fixed $Y$, while the second
accounts for the change in the posterior distribution.

Taking $F=Y$ expresses the derivative of the posterior mean through the
posterior covariance:
\begin{equation*}
\mathsf K_{x,t}
:=
D_x\bar y_{x,t}
=
\lambda(t)
\mathbb E_{\mu_{x,t}}
\left[
Z\otimes Z
\right].
\end{equation*}
Consequently,
\begin{equation*}
D_xZ[v]=-\mathsf K_{x,t}v,
\qquad
\mathsf K_{x,t}v
=
\lambda(t)
\mathbb E_{\mu_{x,t}}
\left[
Z\langle Z,v\rangle
\right].
\end{equation*}
We now apply the differentiation identity to a vector atom. Write
\begin{equation*}
    Q
    :=
    \langle Z,x\rangle^k
    |Z|^m
    \prod_{i=1}^{\ell}
    \langle Z,U_i\rangle,
    \qquad
    F:=ZQ.
\end{equation*}
For the product-rule expansion, introduce the abbreviations
\begin{equation*}
\begin{aligned}
    Q_x
    &:=
    \langle Z,x\rangle^{k-1}
    |Z|^m
    \prod_{i=1}^{\ell}
    \langle Z,U_i\rangle,
    && k\geq1,
    \\
    Q_r
    &:=
    \langle Z,x\rangle^k
    |Z|^{m-2}
    \prod_{i=1}^{\ell}
    \langle Z,U_i\rangle,
    && m\geq2,
    \\
    Q_j
    &:=
    \langle Z,x\rangle^k
    |Z|^m
    \prod_{i\neq j}
    \langle Z,U_i\rangle,
    && 1\leq j\leq\ell.
\end{aligned}
\end{equation*}

The spatial derivative decomposes according to the factor that is
differentiated:
\begin{subequations}\label{eq:D_x E[F] form}
\begin{align}
D_x\mathbb E_{\mu_{x,t}}[F][v]
&=
\lambda(t)
\mathbb E_{\mu_{x,t}}
\left[
    F\langle Z,v\rangle
\right],
\label{eq:DxEF-posterior}
\\
&\quad
-
\mathcal S_{k,m}[U_1,\ldots,U_\ell]
\mathsf K_{x,t}v,
\label{eq:DxEF-leading-Z}
\\
&\quad
+
k\,
\mathbb E_{\mu_{x,t}}
\left[
    ZQ_x\langle Z,v\rangle
\right],
\label{eq:DxEF-direct-x}
\\
&\quad
-
k\,
\langle\mathsf K_{x,t}v,x\rangle
\mathbb E_{\mu_{x,t}}
\left[
    ZQ_x
\right],
\label{eq:DxEF-center-x}
\\
&\quad
-
m\,
\mathbb E_{\mu_{x,t}}
\left[
    ZQ_r
    \langle Z,\mathsf K_{x,t}v\rangle
\right],
\label{eq:DxEF-radius}
\\
&\quad
+
\sum_{j=1}^{\ell}
\mathbb E_{\mu_{x,t}}
\left[
    ZQ_j
    \langle Z,D_xU_j[v]\rangle
\right],
\label{eq:DxEF-recursive-insertion}
\\
&\quad
-
\sum_{j=1}^{\ell}
\langle\mathsf K_{x,t}v,U_j\rangle
\mathbb E_{\mu_{x,t}}
\left[
    ZQ_j
\right].
\label{eq:DxEF-center-insertion}
\end{align}
\end{subequations}
The terms multiplied by $k$ or $m$ are absent when $k=0$ or $m=0$,
respectively. All terms involving $\mathsf K_{x,t}v$ come from
differentiating the centering $Z=Y-\bar y_{x,t}$. The  differentiation of the
posterior weight gives~\eqref{eq:DxEF-posterior}, the direct derivative of $x$ gives \eqref{eq:DxEF-direct-x}, while the
derivatives of the inserted fields give
\eqref{eq:DxEF-recursive-insertion}.

For a scalar atom, the same identity holds with the leading factor $Z$
deleted throughout and with
\eqref{eq:DxEF-leading-Z} omitted.

For time differentiation, we denote the centered time score of the posterior by
\begin{equation*}
\widehat\Theta_{x,t}(Y)
:=
\partial_t
\left(
-\frac{|x-\alpha(t)Y|^2}{2\beta(t)^2}
\right)
-
\mathbb E_{\mu_{x,t}}
\left[
\partial_t
\left(
-\frac{|x-\alpha(t)Y|^2}{2\beta(t)^2}
\right)
\right].
\end{equation*}
The subtracted term accounts for the derivative of the posterior
normalizing constant. In particular,
$\mathbb E_{\mu_{x,t}}[\widehat\Theta_{x,t}]=0$, and for every sufficiently
regular integrand $G=G(t,x,Y)$,
\begin{equation*}
\partial_t\mathbb E_{\mu_{x,t}}[G]
=
\mathbb E_{\mu_{x,t}}[\partial_tG]
+
\mathbb E_{\mu_{x,t}}
\left[
G\widehat\Theta_{x,t}
\right].
\end{equation*}
For the OU coefficients, direct differentiation gives
\begin{equation}\label{eq:centered-time-score}
\begin{aligned}
\widehat\Theta_{x,t}
={}&
-c_x(t)\langle Z,x\rangle
+
c_0(t)
\left(
|Z|^2-\mathcal S_{0,2}[\varnothing]
+2\langle Z,\bar y_{x,t}\rangle
\right),
\end{aligned}
\end{equation}
where
$c_x(t):=\lambda(t)/2+e^{-t/2}\lambda(t)^2$ and
$c_0(t):=\lambda(t)^2/2$. When multiplied by an atom integrand and averaged under $\mu_{x,t}$,
this expression produces only atoms and products of atoms.

Taking $G=Y$ and using
$\mathbb E_{\mu_{x,t}}[\widehat\Theta_{x,t}]=0$, we obtain
\begin{equation}\label{eq:explicit expansion of Nt}
\begin{aligned}
N_{x,t}
:=
\partial_t\bar y_{x,t}
&=
\mathbb E_{\mu_{x,t}}
\left[
    Z\widehat\Theta_{x,t}
\right]
\\
&=
-c_x(t)\mathcal M_{1,0}[\varnothing]
+
c_0(t)\mathcal M_{0,2}[\varnothing]
+
2c_0(t)\mathcal M_{0,0}[\bar y_{x,t}].
\end{aligned}
\end{equation}
Thus $\partial_tZ=-N_{x,t}$, with $N_{x,t}\in\mathfrak M$. For $F=ZQ$, time differentiation acts on the posterior weight, the
centering, and the inserted vectors. Using the same abbreviations
$Q_x,Q_r,Q_j$ as above gives
\begin{subequations}\label{eq:D_t E[F] form}
\begin{align}
\partial_t\mathbb E_{\mu_{x,t}}[F]
&=
\mathbb E_{\mu_{x,t}}
\left[
    F\widehat\Theta_{x,t}
\right],
\label{eq:DtEF-posterior}
\\
&\quad
-
N_{x,t}
\mathcal S_{k,m}[U_1,\ldots,U_\ell],
\label{eq:DtEF-leading-Z}
\\
&\quad
-
k\,
\langle N_{x,t},x\rangle
\mathbb E_{\mu_{x,t}}
\left[
    ZQ_x
\right],
\label{eq:DtEF-center-x}
\\
&\quad
-
m\,
\mathbb E_{\mu_{x,t}}
\left[
    ZQ_r
    \langle Z,N_{x,t}\rangle
\right],
\label{eq:DtEF-radius}
\\
&\quad
+
\sum_{j=1}^{\ell}
\mathbb E_{\mu_{x,t}}
\left[
    ZQ_j
    \langle Z,\partial_tU_j\rangle
\right],
\label{eq:DtEF-recursive-insertion}
\\
&\quad
-
\sum_{j=1}^{\ell}
\langle N_{x,t},U_j\rangle
\mathbb E_{\mu_{x,t}}
\left[
    ZQ_j
\right].
\label{eq:DtEF-center-insertion}
\end{align}
\end{subequations}
As in the spatial formula, the terms involving $Q_x$ or $Q_r$ are omitted
when $k=0$ or $m=0$, respectively. For a scalar atom,
the same formula holds with the leading $Z$ deleted and with
\eqref{eq:DtEF-leading-Z} omitted.

The expansions~\eqref{eq:D_x E[F] form} and~\eqref{eq:D_t E[F] form} make the recursive structure explicit. In
\eqref{eq:D_x E[F] form}, the only term containing a spatial derivative of
an inserted vector is
\eqref{eq:DxEF-recursive-insertion}; in
\eqref{eq:D_t E[F] form}, the only corresponding time-recursive term is
\eqref{eq:DtEF-recursive-insertion}. Each such field has smaller insertion depth than the atom containing it and all remaining contributions can be
read directly from already constructed atoms and the two basic vector fields
$\mathsf K_{x,t}v$ and $N_{x,t}$. We use this observation in the following inductive proof of closure.

\begin{proposition}[Algebraic closure]
\label{prop:posterior-differentiation-closure}
For every $v\in\mathcal G$,
\begin{equation*}
\begin{aligned}
D_x\mathcal A[v]
&\subseteq\mathcal A,
&
D_x\mathfrak M[v]
&\subseteq\mathfrak M,
&
D_x\mathcal G[v]
&\subseteq\mathcal G,
\\
\partial_t\mathcal A
&\subseteq\mathcal A,
&
\partial_t\mathfrak M
&\subseteq\mathfrak M,
&
\partial_t\mathcal G
&\subseteq\mathcal G.
\end{aligned}
\end{equation*}
The insertion depth may increase under differentiation, but remains finite.
\end{proposition}

\begin{proof}
For spatial differentiation, the definition of $\mathcal G$ and linearity
in the direction reduce the argument to
$v=x$, $v=\bar y_{x,t}$, and $v=B\mathcal M$, where
$B\in\mathcal A$ and $\mathcal M$ is a vector atom. In these cases,
\begin{equation*}
\begin{aligned}
    \mathsf K_{x,t}x
    &=
    \lambda(t)\mathcal M_{1,0}[\varnothing],
    \\
    \mathsf K_{x,t}\bar y_{x,t}
    &=
    \lambda(t)\mathcal M_{0,0}[\bar y_{x,t}],
    \\
    \mathsf K_{x,t}(B\mathcal M)
    &=
    \lambda(t)B\mathcal M_{0,0}[\mathcal M].
\end{aligned}
\end{equation*}
Together with \eqref{eq:explicit expansion of Nt}, these identities give
$
D_x\bar y_{x,t}[v]
=
\mathsf K_{x,t}v
\in\mathfrak M,
\
\partial_t\bar y_{x,t}
=
N_{x,t}
\in\mathfrak M.
$ This establishes the derivative statements for the terminal insertion
$\bar y_{x,t}$.

We will repeatedly use two immediate consequences of the definitions. If
$V\in\mathfrak M$, then
$\langle V,x\rangle\in\mathcal A$ and
$\langle V,U\rangle\in\mathcal A$ for every admissible insertion $U$.
Moreover, whenever $\langle Z,V\rangle$ appears inside an outer posterior
expectation, expanding $V$ as an $\mathcal A$-linear combination of vector
atoms simply inserts those vector atoms into the outer atom.

Consider first the terms in \eqref{eq:D_x E[F] form} that do not
differentiate an inserted field. The posterior-weight contribution
\eqref{eq:DxEF-posterior} and the direct-$x$ contribution
\eqref{eq:DxEF-direct-x} introduce a factor $\langle Z,v\rangle$ into
the integrand. For the three directions above, this adds a factor
$\langle Z,x\rangle$, inserts $\bar y_{x,t}$, or inserts $\mathcal M$
with the coefficient $B$ taken outside the expectation.
The leading-vector term \eqref{eq:DxEF-leading-Z} belongs to
$\mathfrak M$ because it is the product of a scalar atom and
$\mathsf K_{x,t}v\in\mathfrak M$.
The centering terms \eqref{eq:DxEF-center-x} and
\eqref{eq:DxEF-center-insertion} have coefficients
$\langle\mathsf K_{x,t}v,x\rangle$ and
$\langle\mathsf K_{x,t}v,U_j\rangle$ in $\mathcal A$, respectively.
Finally, the radial term \eqref{eq:DxEF-radius} is handled by inserting
$\mathsf K_{x,t}v\in\mathfrak M$ into the outer atom.
Thus every nonrecursive spatial contribution has the required form.

For time differentiation, the posterior-weight term
\eqref{eq:DtEF-posterior} produces atoms and products of atoms after
expanding $\widehat\Theta_{x,t}$ through
\eqref{eq:centered-time-score}.
The leading-vector term \eqref{eq:DtEF-leading-Z} is an
$\mathcal A$-multiple of $N_{x,t}\in\mathfrak M$.
The contraction rules above handle
\eqref{eq:DtEF-center-x} and \eqref{eq:DtEF-center-insertion}, since
$\langle N_{x,t},x\rangle$ and $\langle N_{x,t},U_j\rangle$ belong to
$\mathcal A$. The radial term \eqref{eq:DtEF-radius} inserts
$N_{x,t}$ into the outer atom and is handled by the insertion rule.
For scalar atoms, the same arguments apply with the leading-vector
contributions \eqref{eq:DxEF-leading-Z} and
\eqref{eq:DtEF-leading-Z} omitted.

We now prove closure for scalar and vector atoms simultaneously by
induction on insertion depth. An atom of depth zero has no inserted
vectors, so the recursive terms
\eqref{eq:DxEF-recursive-insertion} and
\eqref{eq:DtEF-recursive-insertion} are absent, and all its derivative
terms have already been covered.
Suppose the claim holds for atoms of depth less than $r$, and consider
an atom of depth $r\geq1$. Each insertion $U_j$ is either the posterior
mean or a vector atom of smaller depth. The result for the posterior
mean and the induction hypothesis therefore give
\begin{equation*}
    D_xU_j[v]\in\mathfrak M,
    \qquad
    \partial_tU_j\in\mathfrak M.
\end{equation*}
In the spatial recursive term \eqref{eq:DxEF-recursive-insertion},
we apply the insertion rule to $D_xU_j[v]$.
In the time-recursive term \eqref{eq:DtEF-recursive-insertion},
we apply the same rule to $\partial_tU_j$.
Both contributions become finite sums of atoms with coefficients in
$\mathcal A$. This completes the induction.

The product rule extends the atom-level statements to the algebra
$\mathcal A$ and the module $\mathfrak M$, since coefficients in
$\mathcal C$ are independent of $x$ and remain in $\mathcal C$ under
time differentiation.

Finally, $D_x x[v]=v$ and $\partial_t x=0$. Together with the derivatives
of $\bar y_{x,t}$ established above and closure of $\mathfrak M$,
these identities give the assertions for $\mathcal G$.
Every step uses finite sums and finitely many insertions, so the
resulting insertion depth remains finite.
\end{proof}

\begin{lemma}[Algebraic closure under the material derivative $L$]
\label{lem:centered-moment-normal-form-Ljx}
The extended normal-form class is invariant under the material derivative:
$
L\mathcal G\subseteq\mathcal G.
$
Consequently, for every integer $j\geq1$, $L^jx$ admits a representation
\begin{equation}\label{eq:deco for L^j x}
L^jx
=
A_j(t)x
+
\widetilde A_j(t)\bar y_{x,t}
+
R_j(t,x),
\end{equation}
with $A_j,\widetilde A_j\in\mathcal C$ and $R_j\in\mathfrak M$.
\end{lemma}

\begin{proof}
Recall that
\begin{equation*}
    L
    =
    -\partial_t
    +
    \frac12D_x[\cdot][\hat s],
    \qquad
    \hat s(t,x)
    =
    a(t)x+\lambda(t)\bar y_{x,t}
    \in\mathcal G.
\end{equation*}
Proposition~\ref{prop:posterior-differentiation-closure} therefore gives
$L\mathcal G\subseteq\mathcal G$. Since $x\in\mathcal G$, the decomposition
\eqref{eq:deco for L^j x} follows by induction on $j$.
\end{proof}

The result above is qualitative: it identifies the finite collection of
algebraic structures generated by repeated material differentiation, but
does not yet control their size. The next subsection equips the same normal
form with weights that quantify the growth of each term.

\subsection{Weighted estimates for the centered normal form}
\label{subsec:subsec3.3-Weighted estimates for the centered normal form}

In a normal-form term, the decay of the time-dependent coefficient must
be balanced against the growth allowed by the posterior moments.
Each centered ambient insertion $\langle Z,x\rangle$ may contribute a
factor $e^{t/2}$. We therefore begin by counting these insertions and
the powers of the support radius before introducing the time weights.

\paragraph{Insertion count and radius degree.}
The weights below are assigned to a chosen normal-form expression.
For the terminal insertion, set
$K(\bar y_{x,t})=0$ and $\rho(\bar y_{x,t})=1$.
For scalar and vector atoms, define recursively
\begin{equation*}
\begin{aligned}
K\left(\mathcal S_{k,m}[U_1,\ldots,U_\ell]\right)
&
:=
k+\sum_{i=1}^{\ell}K(U_i),
\\
K\left(\mathcal M_{k,m}[U_1,\ldots,U_\ell]\right)
&:=
k+\sum_{i=1}^{\ell}K(U_i),\\
\rho\left(\mathcal S_{k,m}[U_1,\ldots,U_\ell]\right)
&:=
2k+m+\ell+\sum_{i=1}^{\ell}\rho(U_i),
\\
\rho\left(\mathcal M_{k,m}[U_1,\ldots,U_\ell]\right)
&:=
2k+m+\ell+1+\sum_{i=1}^{\ell}\rho(U_i).
\end{aligned}
\end{equation*}
Both quantities are extended additively over products, with time-dependent
coefficients contributing zero.

Thus $K$ counts the factors $\langle Z,x\rangle$, including those inside
nested insertions. The degree $\rho$ records the corresponding worst-case
support-radius degree. Each ambient insertion contributes two radius powers,
each remaining centered factor $Z$ contributes one, and the terminal vector
$\bar y_{x,t}$ contributes one.

\begin{lemma}[Size of centered monomials]
\label{lem:size-centered-atom}
Suppose that $P_0$ is supported in $B_R(0)$ and satisfies
Assumption~\ref{ass:uniform-edge-doubling}. Let
\begin{equation*}
    \mathcal U(t,x)
    =
    \left(
        \prod_{\nu=1}^{N}\mathcal S_\nu(t,x)
    \right)
    \mathcal V(t,x),
\end{equation*}
where each $\mathcal S_\nu$ is a scalar atom and $\mathcal V$ is either
the terminal vector $\bar y_{x,t}$ or a vector atom.
Then, for every $t>0$ and $x\in\mathbb R^d$,
\begin{equation*}
    |\mathcal U(t,x)|
    \leq
    C_{\mathcal U}
    e^{K(\mathcal U)t/2}
    (1+R)^{\rho(\mathcal U)}.
\end{equation*}
The constant depends only on the algebraic structure of $\mathcal U$ and the doubling constant $D_{\mathrm{UD}}$.
\end{lemma}

\begin{proof}
Since $P_0$ is supported in $B_R(0)$,
$|Z|\leq2R$ and $|\bar y_{x,t}|\leq R$. Moreover,
Theorem~\ref{thm:sec-3-normal-fluctuation-two-scale} gives, for every $q\geq1$,
\begin{equation*}
    \mathbb E_{\mu_{x,t}}
    \left[
        |\langle Z,x\rangle|^q
    \right]
    \leq
    C_q e^{qt/2}(1+R)^{2q}.
\end{equation*}
The terminal insertion satisfies the required estimate because
$K(\bar y_{x,t})=0$ and $\rho(\bar y_{x,t})=1$.
We now proceed by induction on insertion depth for scalar and vector
atoms. For a scalar atom,
$\mathcal S_{k,m}[U_1,\ldots,U_\ell]$,
\begin{equation*}
    \left|
        \mathcal S_{k,m}[U_1,\ldots,U_\ell]
    \right|
    \leq
    (2R)^{m+\ell}
    \left(
        \prod_{i=1}^{\ell}|U_i|
    \right)
    \mathbb E_{\mu_{x,t}}
    \left[
        |\langle Z,x\rangle|^k
    \right],
\end{equation*}
Using the bounds for the insertions, the right-hand side is bounded by
\begin{equation*}
    C
    e^{\left(k+\sum_{i=1}^{\ell}K(U_i)\right)t/2}
    (1+R)^{2k+m+\ell+\sum_{i=1}^{\ell}\rho(U_i)},
\end{equation*}
which has exactly the insertion count and radius degree assigned to
the scalar atom. This calculation also covers depth-zero atoms, for
which the insertion list is empty.

For a vector atom, the leading $Z$ contributes one additional factor
$2R$, accounted for by the extra $+1$ in its radius degree.
This proves the bounds for both types of atoms and completes the
induction.

Finally, multiply the estimates for the scalar factors and the vector
factor in $\mathcal U$. The additivity of $K$ and $\rho$ over products gives the stated
bound. 
\end{proof}

\paragraph{Weighted vector normal-form class.}
We now combine the preceding moment bound with estimates on the
time-dependent coefficients. Set
\begin{equation*}
    b(t):=(1-e^{-t})^{-1},
    \qquad
    a(t)=-e^{-t}b(t),
    \qquad
    \lambda(t)=e^{-t/2}b(t).
\end{equation*}
Since $b(t)\sim t^{-1}$ as $t\downarrow0$ and $b(t)\to1$ as
$t\to\infty$, powers of $b(t)$ record the singularity near forward
time zero.
\begin{definition}[Weighted normal-form class]
\label{def:weighted-normal-form-class}
Let $\sigma,\eta,\rho\geq0$. Consider an elementary vector monomial
\begin{equation*}
    \mathcal B(t,x)
    =
    c(t)
    \left(
        \prod_{\nu=1}^{N}\mathcal S_\nu(t,x)
    \right)
    \mathcal V(t,x),
\end{equation*}
where each $\mathcal S_\nu$ is a scalar atom and $\mathcal V$ is either
$\bar y_{x,t}$ or a vector atom.
We require its chosen representation to satisfy
$\rho(\mathcal B)\leq\rho$ and, for every $q\in\mathbb N_0$, a bound
\begin{equation*}
    |\partial_t^qc(t)|
    \leq
    C_q
    e^{-(K(\mathcal B)+\sigma)t/2}
    b(t)^{\eta+q}.
\end{equation*}
with $C_q$ independent of $t$.
A vector field belongs to $\mathfrak W_{\sigma,\eta,\rho}$ if it
admits a finite sum of monomials satisfying these conditions.
\end{definition}
The factor $e^{-K(\mathcal B)t/2}$ compensates for the moment growth
in Lemma~\ref{lem:size-centered-atom}, leaving the decay
$e^{-\sigma t/2}$. The indices $\eta$ and $\rho$ record the singular
power of $b(t)$ and the support-radius degree, respectively.
Under the assumptions of Lemma~\ref{lem:size-centered-atom},
combining its moment bound with the coefficient estimate at $q=0$
gives
\begin{equation}\label{eq:weighted-class-pointwise-bound}
    |\mathcal B(t,x)|
    \leq
    C_{\mathcal B}
    e^{-\sigma t/2}
    b(t)^\eta
    (1+R)^\rho,
    \qquad
    \mathcal B\in\mathfrak W_{\sigma,\eta,\rho}.
\end{equation}
Here $C_{\mathcal B}$ depends on the chosen finite representation,
its coefficient bounds, and $D_{\mathrm{UD}}$.
Since $b(t)\geq1$, the defining conditions also give the inclusion
\begin{equation}\label{eq:weighted-class-monotonicity}
\left.
\begin{aligned}
    \sigma'&\geq\sigma,\\
    \eta'&\leq\eta,\\
    \rho'&\leq\rho
\end{aligned}
\right\}
\qquad\Longrightarrow\qquad
\mathfrak W_{\sigma',\eta',\rho'}
\subseteq
\mathfrak W_{\sigma,\eta,\rho}.
\end{equation}
We will also repeatedly use the following elementary coefficient estimates. For
every $q\in\mathbb N_0$,
\begin{equation}\label{eq:partial t estimate for a and lambda}
|\partial_t^qa(t)|
\leq
C_qe^{-t}b(t)^{q+1},
\qquad
|\partial_t^q\lambda(t)|
\leq
C_qe^{-t/2}b(t)^{q+1}.
\end{equation}

\paragraph{Basic weighted blocks.}
To propagate the weighted bounds through the recursive insertions, we first identify the weighted classes of the two derivatives of the terminal vector $\bar{y}_{x,t}$ that appear in the time and transport directions. Recall \begin{equation*}
c_x(t)
:=
\frac{\lambda(t)}2+e^{-t/2}\lambda(t)^2,
\qquad
c_0(t):=\frac{\lambda(t)^2}{2}.
\end{equation*}
The coefficient estimates~\eqref{eq:partial t estimate for a and lambda} imply that, for every $q\in\mathbb N_0$,
\begin{equation*}
\begin{aligned}
|\partial_t^qc_x(t)|
\leq
C_qe^{-t/2}b(t)^{q+2},
\qquad
|\partial_t^qc_0(t)|
\leq
C_qe^{-t}b(t)^{q+2}.
\end{aligned}
\end{equation*}
The three vector atoms in \eqref{eq:explicit expansion of Nt} have
insertion counts $1,0,0$, respectively, and all have radius degree $3$.
Combining these counts yields
\begin{equation*}
\begin{aligned}
    c_x(t)\mathcal M_{1,0}[\varnothing]
    &\in
    \mathfrak W_{0,2,3},
    \\
    c_0(t)\mathcal M_{0,2}[\varnothing],
    \quad
    2c_0(t)\mathcal M_{0,0}[\bar y_{x,t}]
    &\in
    \mathfrak W_{2,2,3}.
\end{aligned}
\end{equation*}
The expansion \eqref{eq:explicit expansion of Nt}, together with
the inclusion \eqref{eq:weighted-class-monotonicity}, therefore gives
\begin{equation}\label{eq:Nt-weighted-class}
    \partial_t\bar y_{x,t}
    =
    N_{x,t}
    \in
    \mathfrak W_{0,2,3}.
\end{equation}
For spatial differentiation along $\hat s$, we have
\begin{equation}\label{eq:Dx-ybar-score-expansion}
\begin{aligned}
    D_x\bar y_{x,t}[\hat s]=
    \mathsf K_{x,t}\hat s=
    \lambda(t)a(t)\mathcal M_{1,0}[\varnothing]
    +
    \lambda(t)^2
    \mathcal M_{0,0}[\bar y_{x,t}].
\end{aligned}
\end{equation}
The corresponding coefficients satisfy
\begin{equation}
\begin{aligned}
    |\partial_t^q(\lambda(t)a(t))|
    &\leq
    C_qe^{-3t/2}b(t)^{q+2},
    \\
    |\partial_t^q(\lambda(t)^2)|
    &\leq
    C_qe^{-t}b(t)^{q+2},
    \qquad q\in\mathbb N_0.
\end{aligned}
\end{equation}
After accounting for the insertion counts $1$ and $0$, both terms in
\eqref{eq:Dx-ybar-score-expansion} retain the exponential factor
$e^{-t}$ and belong to $\mathfrak W_{2,2,3}$. Hence
\begin{equation}\label{eq:Dx-ybar-score-weighted}
    D_x\bar y_{x,t}[\hat s]
    \in
    \mathfrak W_{2,2,3}.
\end{equation}
Combining \eqref{eq:Nt-weighted-class},
\eqref{eq:Dx-ybar-score-weighted} and using
\eqref{eq:weighted-class-monotonicity}, we obtain
\begin{equation}\label{eq:L-ybar-weighted}
    L\bar y_{x,t}
    =
    -\partial_t\bar y_{x,t}
    +
    \frac12D_x\bar y_{x,t}[\hat s]
    \in
    \mathfrak W_{0,2,3}.
\end{equation}
These bounds provide the terminal case for the insertion-depth
induction below.

\begin{lemma}[Weighted differentiation rules]
\label{lem:weighted-differentiation-rules}
For every $\sigma,\eta,\rho\geq0$,
\begin{align}
\partial_t:
\mathfrak W_{\sigma,\eta,\rho}
&\longrightarrow
\mathfrak W_{\sigma,\eta+2,\rho+2},
\label{eq:weighted-time-rule}
\\
D_x[\cdot][\hat{s}]:
\mathfrak W_{\sigma,\eta,\rho}
&\longrightarrow
\mathfrak W_{\sigma+2,\eta+2,\rho+2}.
\label{eq:weighted-transport-rule}
\end{align}
Consequently,
\begin{equation}\label{eq:count rule for L}
    L:
    \mathfrak W_{\sigma,\eta,\rho}
    \longrightarrow
    \mathfrak W_{\sigma,\eta+2,\rho+2}.
\end{equation}
\end{lemma}

\begin{proof}
By linearity, it suffices to consider an elementary vector monomial
\begin{equation*}
    \mathcal B(t,x)
    =
    c(t)
    \left(
        \prod_{\nu=1}^{N}\mathcal S_\nu(t,x)
    \right)
    \mathcal V(t,x),
\end{equation*}
where $\mathcal V$ is either a vector atom or the terminal vector
$\bar y_{x,t}$. By the product rule, differentiation acts either on the
deterministic coefficient $c(t)$ or on one atom factor. Since $K$ and
$\rho$ are additive over products, it is enough to keep track of the
relative change produced by differentiating a single factor. We use the atom-level identities
\eqref{eq:D_t E[F] form} and \eqref{eq:D_x E[F] form}.
For a scalar atom, the same bookkeeping applies after omitting
\eqref{eq:DtEF-leading-Z} and \eqref{eq:DxEF-leading-Z}, which arise only
from differentiating the leading vector $Z$.

Before treating the two derivatives, we fix the convention used in the
tables below. The ``old factor'' is the factor in the original atom
that is being replaced, while the ``new factor'' records the corresponding
factor after differentiation; all surrounding factors remain unchanged.
If the replacement carries a deterministic coefficient $d(t)$ satisfying
\begin{equation*}
    |\partial_t^q d(t)|
    \leq
    C_q e^{-r t/2}b(t)^{p+q},
\end{equation*}
and changes the insertion count by $\Delta K$, then
$
    \Delta\eta=p, \ 
    \Delta\sigma=r-\Delta K.
$
Indeed, the original coefficient contributes
$e^{-(K+\sigma)t/2}$, whereas after the replacement the atom has insertion
count $K+\Delta K$. Thus the remaining exponential decay is precisely
$e^{-\Delta\sigma t/2}$.
The radius increment $\Delta\rho$ is read directly from the definition of
$\rho$.

\emph{Time differentiation.}
Table~\ref{tab:time-differentiation-replacements} lists all
nonrecursive replacements appearing in
\eqref{eq:D_t E[F] form}. The first four rows come from the posterior-weight
term \eqref{eq:DtEF-posterior}. The last three arise whenever a centered
occurrence of $Z$ is replaced by
\begin{equation*}
    N_{x,t}
    =
    -c_x(t)\mathcal M_{1,0}[\varnothing]
    +
    c_0(t)\mathcal M_{0,2}[\varnothing]
    +
    2c_0(t)\mathcal M_{0,0}[\bar y_{x,t}],
\end{equation*}
and hence occur in
\eqref{eq:DtEF-leading-Z},
\eqref{eq:DtEF-center-x},
\eqref{eq:DtEF-radius}, and
\eqref{eq:DtEF-center-insertion}.

\begin{table}[htbp]
\centering
\small
\setlength{\arraycolsep}{3pt}
\renewcommand{\arraystretch}{1.35}
$
\begin{array}{@{}c|c|c|c|c@{}}
\text{source}
&
\text{old factor}
&
\text{new factor}
&
\text{coefficient bound}
&
(\Delta K,\Delta\sigma,\Delta\eta,\Delta\rho)
\\ \hline

\multirow{4}{*}{\eqref{eq:DtEF-posterior}}
&
1
&
c_x\langle Z,x\rangle
&
|\partial_t^q c_x|
\lesssim e^{-t/2}b^{q+2}
&
(1,0,2,2)
\\

&
1
&
c_0|Z|^2
&
|\partial_t^q c_0|
\lesssim e^{-t}b^{q+2}
&
(0,2,2,2)
\\

&
1
&
c_0\mathcal S_{0,2}[\varnothing]
&
|\partial_t^q c_0|
\lesssim e^{-t}b^{q+2}
&
(0,2,2,2)
\\

&
1
&
2c_0\langle Z,\bar y_{x,t}\rangle
&
|\partial_t^q(2c_0)|
\lesssim e^{-t}b^{q+2}
&
(0,2,2,2)
\\ \hline

\multirow{3}{*}{
    $\begin{gathered}
        \eqref{eq:DtEF-leading-Z}\quad
        \eqref{eq:DtEF-center-x}
        \\[2mm]
        \eqref{eq:DtEF-radius}\quad
        \eqref{eq:DtEF-center-insertion}
    \end{gathered}$
}
&
Z
&
c_x\mathcal M_{1,0}[\varnothing]
&
|\partial_t^q c_x|
\lesssim e^{-t/2}b^{q+2}
&
(1,0,2,2)
\\

&
Z
&
c_0\mathcal M_{0,2}[\varnothing]
&
|\partial_t^q c_0|
\lesssim e^{-t}b^{q+2}
&
(0,2,2,2)
\\

&
Z
&
2c_0\mathcal M_{0,0}[\bar y_{x,t}]
&
|\partial_t^q(2c_0)|
\lesssim e^{-t}b^{q+2}
&
(0,2,2,2)
\end{array}
$
\caption{Nonrecursive factor replacements and index increments
under time differentiation.}
\label{tab:time-differentiation-replacements}
\end{table}

In the last three rows, the symbol $Z$ denotes the centered occurrence
being differentiated. Depending on the source, it is respectively the
leading vector, the centered variable inside $\langle Z,x\rangle$, one of
the centered factors represented by $|Z|^m$, or the centered variable
inside $\langle Z,U_j\rangle$. The surrounding contraction is present both
before and after the replacement and therefore does not alter the relative
increments in the table.
For example, in \eqref{eq:DtEF-center-x}, the replacement
$Z\mapsto c_x\mathcal M_{1,0}[\varnothing]$ converts
$\langle Z,x\rangle$ into
$c_x\langle\mathcal M_{1,0}[\varnothing],x\rangle$.
Relative to the original factor this adds one unit of $K$ and two units of
$\rho$, exactly as recorded in the fifth row. The same calculation applies
to the radial and inserted-contraction terms.
Every nonrecursive time-differentiation term satisfies
\begin{equation*}
    \Delta\rho\leq2,
    \qquad
    \Delta\eta\leq2,
    \qquad
    \Delta\sigma\geq0.
\end{equation*}
Rows with $\Delta\sigma=2$ have stronger residual decay than required and
are included in the target class by monotonicity. This proves the
nonrecursive part of \eqref{eq:weighted-time-rule}.

\emph{Transport differentiation.}
We now set $v=\hat s$ in \eqref{eq:D_x E[F] form}. The posterior-weight
term \eqref{eq:DxEF-posterior} uses
\begin{equation*}
    \lambda(t)\langle Z,\hat s\rangle
    =
    \lambda(t)a(t)\langle Z,x\rangle
    +
    \lambda(t)^2\langle Z,\bar y_{x,t}\rangle.
\end{equation*}
Whenever differentiation of a centered occurrence produces
$\mathsf K_{x,t}\hat s$, we use
\eqref{eq:Dx-ybar-score-expansion},
\begin{equation*}
    \mathsf K_{x,t}\hat s
    =
    \lambda(t)a(t)\mathcal M_{1,0}[\varnothing]
    +
    \lambda(t)^2
    \mathcal M_{0,0}[\bar y_{x,t}].
\end{equation*}
Finally, the direct derivative
\eqref{eq:DxEF-direct-x} of the explicit $x$ in
$\langle Z,x\rangle$ produces
$\langle Z,\hat s\rangle
=
a\langle Z,x\rangle+
\lambda\langle Z,\bar y_{x,t}\rangle$.
These give all nonrecursive transport replacements, summarized in
Table~\ref{tab:transport-replacements}.

\begin{table}[htbp]
\centering
\small
\setlength{\arraycolsep}{3pt}
\renewcommand{\arraystretch}{1.35}
$
\begin{array}{@{}c|c|c|c|c@{}}
\text{source}
&
\text{old factor}
&
\text{new factor}
&
\text{coefficient bound}
&
(\Delta K,\Delta\sigma,\Delta\eta,\Delta\rho)
\\ \hline

\multirow{2}{*}{\eqref{eq:DxEF-posterior}}
&
1
&
\lambda a\langle Z,x\rangle
&
|\partial_t^q(\lambda a)|
\lesssim e^{-3t/2}b^{q+2}
&
(1,2,2,2)
\\

&
1
&
\lambda^2\langle Z,\bar y_{x,t}\rangle
&
|\partial_t^q(\lambda^2)|
\lesssim e^{-t}b^{q+2}
&
(0,2,2,2)
\\ \hline

\multirow{2}{*}{
    $\begin{gathered}
        \eqref{eq:DxEF-leading-Z}\quad
        \eqref{eq:DxEF-center-x}
        \\[2mm]
        \eqref{eq:DxEF-radius}\quad
        \eqref{eq:DxEF-center-insertion}
    \end{gathered}$
}
&
Z
&
\lambda a\mathcal M_{1,0}[\varnothing]
&
|\partial_t^q(\lambda a)|
\lesssim e^{-3t/2}b^{q+2}
&
(1,2,2,2)
\\

&
Z
&
\lambda^2\mathcal M_{0,0}[\bar y_{x,t}]
&
|\partial_t^q(\lambda^2)|
\lesssim e^{-t}b^{q+2}
&
(0,2,2,2)
\\ \hline

\multirow{2}{*}{\eqref{eq:DxEF-direct-x}}
&
\langle Z,x\rangle
&
a\langle Z,x\rangle
&
|\partial_t^q a|
\lesssim e^{-t}b^{q+1}
&
(0,2,1,0)
\\

&
\langle Z,x\rangle
&
\lambda\langle Z,\bar y_{x,t}\rangle
&
|\partial_t^q\lambda|
\lesssim e^{-t/2}b^{q+1}
&
(-1,2,1,0)
\end{array}
$
\caption{Nonrecursive factor replacements and index increments
for the transport term.}
\label{tab:transport-replacements}
\end{table}

The third and fourth rows simultaneously cover
\eqref{eq:DxEF-leading-Z},
\eqref{eq:DxEF-center-x},
\eqref{eq:DxEF-radius}, and
\eqref{eq:DxEF-center-insertion}: in each case one centered occurrence of
$Z$ is replaced by one of the two components of
$\mathsf K_{x,t}\hat s$. As in the time table, the surrounding contraction
does not affect the relative bookkeeping.

The last row is worth noting. Replacing
$\langle Z,x\rangle$ by
$\lambda(t)\langle Z,\bar y_{x,t}\rangle$ removes one centered ambient insertion,
so $\Delta K=-1$. The factor $\lambda(t)$ carries only
$e^{-t/2}$, but after accounting for this decrease in $K$ the residual
decay still gains two units:
$
    \Delta\sigma
    =
    1-(-1)
    =
    2.
$
Thus every nonrecursive transport term satisfies
\begin{equation*}
    \Delta\rho\leq2,
    \qquad
    \Delta\eta\leq2,
    \qquad
    \Delta\sigma=2,
\end{equation*}
and therefore belongs to
$\mathfrak W_{\sigma+2,\eta+2,\rho+2}$.

\emph{Recursive insertions.}
It remains to treat
\eqref{eq:DtEF-recursive-insertion} and
\eqref{eq:DxEF-recursive-insertion}, which differentiate an inserted
vector. We handle these terms by induction on insertion depth, keeping
the coefficients generated by differentiation separate from the atoms.

For either differentiation operation, write
$
    (\mathscr D,\tau)
    =
    (\partial_t,0)
    \text{ or }
    (D_x[\cdot][\hat s],2).
$
For every scalar or vector atom $U$, and for the terminal vector
$U=\bar y_{x,t}$, we claim that there is a finite expansion
\begin{equation*}
    \mathscr D U
    =
    \sum_\nu d_\nu(t)\mathcal U_\nu(t,x)
\end{equation*}
such that
\begin{equation*}
\begin{aligned}
    \rho(\mathcal U_\nu)
    &\leq \rho(U)+2,
    \\
    |\partial_t^q d_\nu(t)|
    &\leq
    C_{\nu,q}
    e^{-\left(K(\mathcal U_\nu)-K(U)+\tau\right)t/2}
    b(t)^{q+2},
    \qquad q\in\mathbb N_0.
\end{aligned}
\end{equation*}
Here each $\mathcal U_\nu$ is a product of scalar atoms, with one
additional factor equal to a vector atom or $\bar y_{x,t}$ when $U$
is vector-valued. These monomials carry no deterministic coefficient.
The exponential factor records the change in insertion count from
$U$ to $\mathcal U_\nu$.

For the terminal vector, the claim follows from
\eqref{eq:Nt-weighted-class} and
\eqref{eq:Dx-ybar-score-weighted}, since
$K(\bar y_{x,t})=0$ and $\rho(\bar y_{x,t})=1$.
Atoms of depth zero have no inserted vectors, so their derivatives
are covered by the two tables above.

Suppose the claim holds for atoms of depth less than $r$, and
consider an atom of depth $r\geq1$. Each insertion $U_j$ is either
the terminal vector or a vector atom of smaller depth. The induction
hypothesis therefore applies to $\mathscr D U_j$.

Consider one term $d(t)\mathcal V$ in this expansion and write
\begin{equation*}
    \mathcal V
    =
    \left(\prod_\ell\mathcal S_\ell\right)\mathcal W,
\end{equation*}
where $\mathcal W$ is a vector atom or the terminal vector.
Its contribution to the differentiated insertion is
\begin{equation*}
    \langle Z,d(t)\mathcal V\rangle
    =
    d(t)
    \left(\prod_\ell\mathcal S_\ell\right)
    \langle Z,\mathcal W\rangle.
\end{equation*}
The scalar factors are independent of the outer integration
variable and can be taken outside the expectation. Thus this term
is again a coefficient times a scalar or vector monomial.

All other factors in the outer atom remain unchanged. Consequently,
the changes in its total insertion count and radius degree are
exactly
\begin{equation*}
    \Delta K=K(\mathcal V)-K(U_j),
    \qquad
    \Delta\rho=\rho(\mathcal V)-\rho(U_j)\leq2.
\end{equation*}
The induction hypothesis bounds the generated coefficient by
\begin{equation*}
    |\partial_t^q d(t)|
    \leq
    C_q e^{-(\Delta K+\tau)t/2}b(t)^{q+2}.
\end{equation*}
This is precisely the coefficient bound required for the
differentiated outer atom. Hence the recursive terms satisfy the
same relative weight estimates as the nonrecursive terms.
The induction is complete.

\emph{Deterministic coefficients and products.}
It remains to differentiate the deterministic coefficient of an elementary
monomial. If $c(t)$ satisfies the defining bound of
$\mathfrak W_{\sigma,\eta,\rho}$, then
\begin{equation*}
    |\partial_t^q c'(t)|
    =
    |\partial_t^{q+1}c(t)|
    \leq
    C_q
    e^{-(K(\mathcal B)+\sigma)t/2}
    b(t)^{\eta+q+1}.
\end{equation*}
Thus coefficient differentiation gives only one additional power of
$b(t)$, which is contained in the time target with index $\eta+2$ because
$b(t)\geq1$. Spatial differentiation does not act on $c(t)$.

The product rule, together with additivity of $K$ and $\rho$ and
multiplicativity of the coefficient profiles, now proves
\eqref{eq:weighted-time-rule} and
\eqref{eq:weighted-transport-rule} for elementary monomials. Linearity
extends the result to finite sums.

Since
$
    L
    =
    -\partial_t
    +
    \frac12D_x[\cdot][\hat s]
$
and
$
    \mathfrak W_{\sigma+2,\eta+2,\rho+2}
    \subseteq
    \mathfrak W_{\sigma,\eta+2,\rho+2},
$
the two differentiation rules imply
\eqref{eq:count rule for L}.
\end{proof}

For stability and the Runge--Kutta analysis, we also need mixed
space--time derivatives. The next lemma treats derivatives contracted
with the position vector $x$ and with fixed unit directions, keeping
track of their different weight increments. The device of the proof is
to treat each fixed direction $w$ as an additional \emph{labelled
insertion} of depth zero, so that the differentiation identities of
Section~\ref{subsec:sec3.2-algebraic-normal-form-closure} apply
verbatim, and to substitute $w=x$ only after all derivatives have
been taken. The two kinds of direction then cost different weights: a
unit direction enters only through $\langle Z,w\rangle$, bounded by
$2R$, and changes the indices by $(\Delta\sigma,\Delta\eta,\Delta\rho)
=(1,1,1)$; the position direction turns $\langle Z,w\rangle$ into the
centered insertion $\langle Z,x\rangle$, which raises the insertion
count $K$ by one and costs $(\Delta\sigma,\Delta\eta,\Delta\rho)
=(-1,0,1)$ relative to the unit case. This is the origin of the
different powers of $b(t)$ and $1+R$ attached to $k$ and $\ell$ in
\eqref{eq:higher-order-mixed-directional-bound}.

\begin{lemma}[Higher-order directional derivative bounds]
\label{lem:sec3-weighted-directional-bound}
Suppose that $P_0$ is supported in $B_R(0)$ and satisfies
Assumption~\ref{ass:uniform-edge-doubling}. Let
$\mathcal B\in\mathfrak W_{\sigma,\eta,\rho}$.
Then, for every $q,k,\ell\in\mathbb N_0$ and every fixed
$w_1,\ldots,w_\ell\in\mathbb R^d$ with $|w_j|\leq1$,
\begin{align}
&
\left|
\partial_t^q
D_x^{k+\ell}\mathcal B(t,x)
\left[
    \underbrace{x,\ldots,x}_{k\ \textrm {times}},
    w_1,\ldots,w_\ell
\right]
\right|
\nonumber\\
&\qquad\leq
C_{\mathcal B,q,k,\ell}
e^{-(\sigma+\ell)t/2}
b(t)^{\eta+2q+k+\ell}
(1+R)^{\rho+2q+2k+\ell}.
\label{eq:higher-order-mixed-directional-bound}
\end{align}
In particular, taking $k=0$ gives
\begin{equation}
    \left\lVert
        \partial_t^qD_x^\ell\mathcal B(t,x)
    \right\rVert_{\mathrm{op}}
    \leq
    C_{\mathcal B,q,\ell}
    e^{-(\sigma+\ell)t/2}
    b(t)^{\eta+2q+\ell}
    (1+R)^{\rho+2q+\ell}.
    \label{eq:higher-order-operator-bound}
\end{equation}
\end{lemma}

\begin{proof}
We first account for the time derivatives. Repeated application
of \eqref{eq:weighted-time-rule} gives
\begin{equation*}
    \partial_t^q\mathcal B
    \in
    \mathfrak W_{\sigma,\eta+2q,\rho+2q}.
\end{equation*}
Set $n=k+\ell$ and introduce fixed directions
$w_1,\ldots,w_n$, initially with $|w_j|\leq1$, and write
$\mathbf w=(w_1,\ldots,w_n)$.

\emph{Direction-labelled classes.}
We enlarge the atom construction of
Section~\ref{subsec:sec3.2-algebraic-normal-form-closure} by admitting
each $w_j$ as an additional insertion, with
\begin{equation*}
    \operatorname{dep}(w_j)=0,
    \qquad
    K(w_j)=0,
    \qquad
    \rho(w_j)=0,
\end{equation*}
so that atoms may contain factors $\langle Z,w_j\rangle$ in addition
to $\langle Z,x\rangle$, $|Z|^m$, and $\langle Z,U_i\rangle$. The
vector factor of a weighted monomial remains either $\bar y_{x,t}$ or
a vector atom. We denote by
\begin{equation*}
    \mathfrak W^{\mathbf w}_{\sigma,\eta,\rho}
\end{equation*}
the weighted class of Definition~\ref{def:weighted-normal-form-class}
built on this enlarged set of atoms, with the same indices
$(\sigma,\eta,\rho)$ and the same coefficient bounds.

Since $|w_j|\leq1$ and $|\langle Z,w_j\rangle|\leq2R$,
the proof of Lemma~\ref{lem:size-centered-atom} also applies
to these enlarged classes. Thus
\eqref{eq:weighted-class-pointwise-bound} remains valid,
with constants uniform over the indicated directions for
fixed normal-form structure and coefficient bounds.
We keep the direction vectors as labelled arguments of
the atom expansion; the differentiation identities themselves
hold for arbitrary values of these arguments.

We first consider one spatial derivative in a fixed direction
$w$. Setting $v=w$ in \eqref{eq:D_x E[F] form} gives the
four differentiation mechanisms listed in
Table~\ref{tab:fixed-direction-differentiation}.
Overall signs and fixed numerical factors do not affect
the weight counts.

\begin{table}[htbp]
\centering
\small
\setlength{\arraycolsep}{3pt}
\renewcommand{\arraystretch}{1.35}
$
\begin{array}{@{}c|c|c|c|c@{}}
\text{source}
&
\text{derivative hits}
&
\text{old factor}
&
\text{new factor}
&
(\Delta K,\Delta\sigma,\Delta\eta,\Delta\rho)
\\ \hline

\eqref{eq:DxEF-posterior}
&
\text{posterior density}
&
1
&
\text{add }\lambda(t)\langle Z,w\rangle
&
(0,1,1,1)
\\ \hline

\begin{gathered}
    \eqref{eq:DxEF-leading-Z},\
    \eqref{eq:DxEF-center-x},
    \\[-1mm]
    \eqref{eq:DxEF-radius},\
    \eqref{eq:DxEF-center-insertion}
\end{gathered}
&
\text{a centered }Z
&
Z
&
\mathsf K_{x,t}w
=
\lambda(t)\mathcal M_{0,0}[w]
&
(0,1,1,1)
\\ \hline

\eqref{eq:DxEF-direct-x}
&
\begin{gathered}
    \text{the explicit }x\\
    \text{in }\langle Z,x\rangle
\end{gathered}
&
\langle Z,x\rangle
&
\langle Z,w\rangle
&
(-1,1,0,-1)
\\ \hline

\eqref{eq:DxEF-recursive-insertion}
&
\text{an insertion }U_j
&
U_j
&
D_xU_j[w]
&
\text{controlled inductively}
\end{array}
$
\caption{Differentiation mechanisms and index increments
for a spatial derivative in a fixed direction $w$.}
\label{tab:fixed-direction-differentiation}
\end{table}
In the first row, differentiation of the posterior density
introduces
$
    \lambda(t)\langle Z,w\rangle.
$
The fixed insertion does not increase the normal count $K$.
Hence $\lambda(t)=e^{-t/2}b(t)$ contributes one unit of
residual decay and one power of $b(t)$, while
$\langle Z,w\rangle$ contributes one radius degree.
The coefficient bounds for all time derivatives follow from
\eqref{eq:partial t estimate for a and lambda}.

The same increments occur in the centering terms, since
\begin{equation*}
    D_xZ[w]
    =
    -\mathsf K_{x,t}w
    =
    -\lambda(t)\mathcal M_{0,0}[w].
\end{equation*}
The atom $\mathcal M_{0,0}[w]$ has insertion count zero
and radius degree two, so replacing one centered factor
$Z$ increases the radius degree by one.

Direct differentiation of the explicit $x$ replaces
$\langle Z,x\rangle$ by $\langle Z,w\rangle$.
This removes one normal insertion and reduces the radius
degree by one. The coefficient is unchanged, so the
decrease in $K$ increases the residual decay index by one.
This gives the third row of the table.

For the recursive term
\eqref{eq:DxEF-recursive-insertion}, we use the relative
coefficient argument from the proof of
Lemma~\ref{lem:weighted-differentiation-rules}.
More precisely, for a coefficient-free scalar or vector
atom $U$, we prove an expansion
\begin{equation*}
    D_xU[w]
    =
    \sum_\nu d_\nu(t)\mathcal V_\nu(t,x),
\end{equation*}
where the coefficient-free monomials $\mathcal V_\nu$
have the same scalar or vector type as $U$ and satisfy
\begin{equation*}
\begin{aligned}
    \rho(\mathcal V_\nu)
    &\leq \rho(U)+1,
    \\
    |\partial_t^r d_\nu(t)|
    &\leq
    C_{\nu,r}
    e^{-\left(K(\mathcal V_\nu)-K(U)+1\right)t/2}
    b(t)^{r+1},
    \qquad r\in\mathbb N_0.
\end{aligned}
\end{equation*}
The nonrecursive rows satisfy these estimates.
For scalar atoms, the leading-$Z$ contribution
\eqref{eq:DxEF-leading-Z} is omitted.

The terminal cases follow from
\begin{equation*}
    D_x\bar y_{x,t}[w]
    =
    \lambda(t)\mathcal M_{0,0}[w],
    \qquad
    D_xw_i[w]=0.
\end{equation*}
Atoms of depth zero are covered by the nonrecursive rows.
At greater depth, apply the induction hypothesis to each
differentiated insertion $U_j$.
For a generated term $d(t)\mathcal V$, write
\begin{equation*}
    \mathcal V
    =
    \left(\prod_\nu\mathcal S_\nu\right)\mathcal W,
\end{equation*}
where $\mathcal W$ is a vector atom or the posterior mean.
Reinsertion gives
\begin{equation*}
    \langle Z,d(t)\mathcal V\rangle
    =
    d(t)
    \left(\prod_\nu\mathcal S_\nu\right)
    \langle Z,\mathcal W\rangle.
\end{equation*}
The scalar factors can be taken outside the outer expectation.
The changes in the total insertion count and radius degree
are therefore exactly those of the differentiated insertion,
and the generated coefficient remains $d(t)$.
This completes the insertion-depth induction.

The product rule extends these estimates to weighted
monomials. Spatial differentiation does not act on their
time-dependent coefficients, and the bounds on the new
coefficients are preserved under multiplication.
Using \eqref{eq:weighted-class-monotonicity}, we obtain
\begin{equation}
    D_x[\cdot][w_j]:
    \mathfrak W^{\mathbf w}_{\sigma,\eta,\rho}
    \longrightarrow
    \mathfrak W^{\mathbf w}_{\sigma+1,\eta+1,\rho+1}.
    \label{eq:higher-order-unit-direction-rule}
\end{equation}

We now differentiate successively in the fixed directions
$w_1,\ldots,w_n$, using each direction label once.
Repeated application of
\eqref{eq:higher-order-unit-direction-rule} gives
\begin{equation}\label{eq:sec3-4 D_x^n dt^q}
    D_x^n\partial_t^q\mathcal B(t,x)
    [w_1,\ldots,w_n]
    \in
    \mathfrak W^{\mathbf w}_{
        \sigma+n,\,
        \eta+2q+n,\,
        \rho+2q+n}.
\end{equation}

We retain the direction labels in this finite atom expansion.

\emph{Claim: every monomial in the expansion
\eqref{eq:sec3-4 D_x^n dt^q} contains each direction label
$w_1,\ldots,w_n$ exactly once, and its time-dependent coefficient is
independent of the direction vectors.}
Indeed, the initial field $\partial_t^q\mathcal B$ contains none of
these labels. Each new differentiation introduces its direction
exactly once, through a factor $\langle Z,w_j\rangle$, possibly inside
a nested atom, and the product rule together with the four mechanisms
of Table~\ref{tab:fixed-direction-differentiation} preserves the
occurrences of all earlier labels in each nonzero term. The
coefficients arise only from $c(t)$, $\lambda(t)$, and the
coefficients of $\hat s$, none of which involves $\mathbf w$.

We next substitute
$
    w_{\ell+1}=\cdots=w_{\ell+k}=x
$
in the atom expansion~\eqref{eq:sec3-4 D_x^n dt^q}, after all derivatives have been taken.
Each substitution changes one factor
$\langle Z,w_j\rangle$ into $\langle Z,x\rangle$.
At the atom level, this is the identity
\begin{equation*}
    \left.
        \mathcal M_{a,m}[w_j,U_1,\ldots,U_s]
    \right|_{w_j=x}
    =
    \mathcal M_{a+1,m}[U_1,\ldots,U_s],
\end{equation*}
and the same identity holds for scalar atoms.
It also applies inside nested atoms.

The substitution increases $K$ by 1.
The number of ordinary insertions decreases by 1,
while the contribution $2K$ to the radius degree increases
by 2. Hence the total radius degree increases by 1.
Since the coefficient is unchanged, the residual decay
index decreases by 1, and the singularity index is
unchanged. For each term under consideration, the weight
change is therefore
\begin{equation}
    (K,\sigma',\eta',\rho')
    \longmapsto
    (K+1,\sigma'-1,\eta',\rho'+1).
    \label{eq:higher-order-x-direction-rule}
\end{equation}

Each of the $k$ substituted labels occurs exactly once.
Applying \eqref{eq:higher-order-x-direction-rule} to these
labels therefore yields
\begin{equation*}
\begin{aligned}
    D_x^{k+\ell}\partial_t^q\mathcal B(t,x)
    \left[
        w_1,\ldots,w_\ell,
        \underbrace{x,\ldots,x}_{k\ \textrm{ times}}
    \right]\in
    \mathfrak W^{(w_1,\ldots,w_\ell)}_{
        \sigma+\ell,\,
        \eta+2q+k+\ell,\,
        \rho+2q+2k+\ell}.
\end{aligned}
\end{equation*}
Here only the remaining fixed directions occur as additional
insertions; the substituted directions have become normal
factors in the original atom construction.

Applying the weighted pointwise estimate~\eqref{eq:weighted-class-pointwise-bound} proves
\eqref{eq:higher-order-mixed-directional-bound}, using
commutation of time and spatial derivatives and symmetry
of the spatial derivative slots.
The constants are uniform over $|w_1|,\ldots,|w_\ell|\leq1$.
Taking $k=0$ and the supremum over these directions gives
\eqref{eq:higher-order-operator-bound}.
\end{proof}

We now return to the qualitative decomposition~\eqref{eq:deco for L^j x} and use the weighted differentiation rules to quantify its time-dependent coefficients and centered remainder.

\begin{proposition}[Quantitative normal-form decomposition]
\label{prop:quantitative-normal-form-decomposition}
The decomposition \eqref{eq:deco for L^j x} may be chosen so that, for every
$j\geq1$ and $q\in\mathbb N_0$,
\begin{equation}\label{eq:Aj-Atilde-estimates}
\begin{aligned}
|\partial_t^qA_j(t)|
&\leq
C_{j,q}e^{-t}b(t)^{j+q},
\\
|\partial_t^q\widetilde A_j(t)|
&\leq
C_{j,q}e^{-t/2}b(t)^{j+q}.
\end{aligned}
\end{equation}
Moreover,
\begin{equation*}
R_1=0,
\qquad
R_j\in
\mathfrak W_{1,2j-1,2j-1}
\quad
\text{for }j\geq2.
\end{equation*}
The coefficient bounds, including those in the weighted
representation of $R_j$, depend only on $j$ and the order
of time differentiation.
\end{proposition}

\begin{proof}
At level $j=1$,
$A_1=a/2$, $\widetilde A_1=\lambda/2$, and $R_1=0$.
Applying $L$ to \eqref{eq:deco for L^j x} gives
\begin{equation}\label{eq:normal-form-recursions}
\begin{aligned}
    A_{j+1}
    &=
    -A_j'+\frac12aA_j,
    \\
    \widetilde A_{j+1}
    &=
    -\widetilde A_j'
    +
    \frac12\lambda A_j,
    \\
    R_{j+1}
    &=
    LR_j
    +
    \widetilde A_jL\bar y_{x,t}.
\end{aligned}
\end{equation}
Suppose the coefficient estimates hold at order $j$ for
every $q\in\mathbb N_0$. The product rule and
\eqref{eq:partial t estimate for a and lambda} give
\begin{equation*}
\begin{aligned}
    |\partial_t^q(aA_j)(t)|
    &\leq
    C_{j,q}e^{-2t}b(t)^{j+q+1},
    \\
    |\partial_t^q(\lambda A_j)(t)|
    &\leq
    C_{j,q}e^{-3t/2}b(t)^{j+q+1}.
\end{aligned}
\end{equation*}
Together with the induction bounds for $A_j'$ and
$\widetilde A_j'$, the first two recursions in
\eqref{eq:normal-form-recursions} prove
\eqref{eq:Aj-Atilde-estimates} at order $j+1$.

For the centered remainder,
\eqref{eq:L-ybar-weighted} and
\eqref{eq:Aj-Atilde-estimates} imply
$
    \widetilde A_jL\bar y_{x,t}
    \in
    \mathfrak W_{1,j+2,3}.
$
In particular,
$R_2\in\mathfrak W_{1,3,3}$.
Now suppose that
$R_j\in\mathfrak W_{1,2j-1,2j-1}$ for some $j\geq2$.
The material differentiation rule
\eqref{eq:count rule for L} gives
\begin{equation*}
    LR_j
    \in
    \mathfrak W_{1,2j+1,2j+1}.
\end{equation*}
Also, by \eqref{eq:weighted-class-monotonicity},
\begin{equation*}
    \widetilde A_jL\bar y_{x,t}
    \in
    \mathfrak W_{1,j+2,3}
    \subseteq
    \mathfrak W_{1,2j+1,2j+1},
\end{equation*}
since $j+2\leq2j+1$ and $3\leq2j+1$.
The last recursion in \eqref{eq:normal-form-recursions}
therefore proves the remainder estimate at order $j+1$.

The atom structure produced by this recursion depends only
on $j$. All time-dependent coefficients are generated from
$a$, $\lambda$, and their derivatives, so their bounds
depend only on $j$ and the number of time derivatives.
\end{proof}
Combining this decomposition with
\eqref{eq:weighted-class-pointwise-bound} and
\eqref{eq:higher-order-operator-bound} yields the main regularity result of this section.
The centered normal form propagates the geometric control
of posterior fluctuations to arbitrary material order,
giving global pointwise bounds for $L^jx$ and spatially
uniform bounds for $D_xL^jx$ directly from the support
and doubling assumptions on the possibly singular target $P_0$.
The constants are independent of dimension, and the exponential
time weights yield integrable spatial Lipschitz bounds on
$[\delta,\infty)$ for every $\delta>0$.
These estimates complete the regularity analysis needed
for the subsequent high-order convergence bounds.

\begin{theorem}[Weighted estimates for flow-generated fields]
\label{thm:weighted-estimates-Ljx}
Suppose that $P_0$ is supported in $B_R(0)$ and satisfies
Assumption~\ref{ass:uniform-edge-doubling}. For every integer $j\geq1$, every $t>0$, and every
$x\in\mathbb R^d$,
\begin{align}
    &|L^jx(t,x)|
    \leq
    C_j
    \left[
        e^{-t}(1-e^{-t})^{-j}|x|
        +
        e^{-t/2}
        \left(
            \frac{1+R}{1-e^{-t}}
        \right)^{2j-1}
    \right],
    \label{eq:truncation j-th estimate}
    \\
    &\left\lVert{}D_x(L^jx)(t,x)\right\rVert_{\mathrm{op}}
    \leq
    C_je^{-t}
    \left(
        \frac{1+R}{1-e^{-t}}
    \right)^{2j}.
    \label{eq:j-th lipschitz estimate}
\end{align}
The constant $C_j$ depends only on $j$ and the doubling
constant $D_{\mathrm{UD}}$; in particular, it is independent
of the ambient dimension.

If $t\geq\delta>0$, the same bounds hold with
$(1-e^{-t})^{-1}$ replaced by $(1-e^{-\delta})^{-1}$. In particular, for $0<\delta\leq1$ and $t\geq\delta$,
\begin{align}
    &|L^jx(t,x)|
    \leq
    C_j
    \left[
        e^{-t}\delta^{-j}|x|
        +
        e^{-t/2}
        \left(
            \frac{1+R}{\delta}
        \right)^{2j-1}
    \right],
    \label{eq:truncation j-th estimate+early stopping}
    \\
    &\left\lVert{}D_x(L^jx)(t,x)\right\rVert_{\mathrm{op}}
    \leq
    C_je^{-t}
    \left(
        \frac{1+R}{\delta}
    \right)^{2j}.
    \label{eq:j-th lipschitz estimate+early stopping}
\end{align}
\end{theorem}

In contrast with the bounded-field setting of classical ODE-solver
analysis, the growth in $|x|$ is confined to a single affine term with
an explicit decaying coefficient, while all spatial derivatives are
bounded uniformly in $x$; the convergence analysis of
Section~\ref{sec:high-order-convergence} is built on this splitting.

\begin{proof}
We start from the decomposition \eqref{eq:deco for L^j x} given by Lemma~\ref{lem:centered-moment-normal-form-Ljx},
\begin{equation*}
L^jx
=
A_j(t)x
+
\widetilde A_j(t)\bar y_{x,t}
+
R_j(t,x),
\end{equation*}
Proposition~\ref{prop:quantitative-normal-form-decomposition} and
\eqref{eq:weighted-class-pointwise-bound} give
\begin{equation*}
|R_j(t,x)|
\leq
C_je^{-t/2}b(t)^{2j-1}(1+R)^{2j-1},
\end{equation*}
where the statement is trivial for $j=1$. Since $|\bar y_{x,t}|\leq R$, the coefficient estimates
\eqref{eq:Aj-Atilde-estimates} imply
\begin{equation*}
\begin{aligned}
    |L^jx(t,x)|
    \leq C_j\bigl[
    &e^{-t}b(t)^j|x|
        +
        e^{-t/2}Rb(t)^j+
        e^{-t/2}b(t)^{2j-1}(1+R)^{2j-1}
    \bigr].
\end{aligned}
\end{equation*}
For $j\geq1$, we have $j\leq2j-1$ and $b(t)\geq1$, so
$
    Rb(t)^j
    \leq
    (1+R)^{2j-1}b(t)^{2j-1}.
$
Absorbing the posterior-mean term into the remainder bound
proves \eqref{eq:truncation j-th estimate}.

For the spatial derivative, apply
\eqref{eq:higher-order-operator-bound} with $q=0$ and
$\ell=1$ to
$R_j\in\mathfrak W_{1,2j-1,2j-1}$ for $j\geq2$
and to $\bar y_{x,t}\in\mathfrak W_{0,0,1}$. This gives
\begin{equation*}
\begin{aligned}
    \sup_{|w|\leq1}|D_xR_j(t,x)[w]|
    &\leq
    C_je^{-t}b(t)^{2j}(1+R)^{2j},
    \\
    \sup_{|w|\leq1}|D_x\bar y_{x,t}[w]|
    &\leq
    Ce^{-t/2}b(t)(1+R)^2.
\end{aligned}
\end{equation*}
Differentiating the normal-form decomposition,
\begin{equation*}
    D_x(L^jx)[w]
    =
    A_jw
    +
    \widetilde A_jD_x\bar y_{x,t}[w]
    +
    D_xR_j[w].
\end{equation*}
Taking the supremum over $|w|\leq1$ and using
\eqref{eq:Aj-Atilde-estimates}, we obtain
\begin{equation*}
\begin{aligned}
    \lVert{}D_x(L^jx)(t,x)\rVert_{\mathrm{op}}
    \leq C_je^{-t}\bigl[
        &b(t)^j
        +
        b(t)^{j+1}(1+R)^2
        \\
        &+
        b(t)^{2j}(1+R)^{2j}
    \bigr].
\end{aligned}
\end{equation*}
Since $b(t)\geq1$, $1+R\geq1$, and $j+1\leq2j$,
the last term controls the first two. This proves
\eqref{eq:j-th lipschitz estimate}.

Finally, for $t\geq\delta>0$,
$
    b(t)\leq b(\delta).
$
For $0<\delta\leq1$, the inequality
$1-e^{-\delta}\geq(1-e^{-1})\delta$ further gives
$
    b(\delta)
    \leq
    1/{\delta (1-e^{-1})}.
$
Substitution into the preceding estimates proves
\eqref{eq:truncation j-th estimate+early stopping}
and \eqref{eq:j-th lipschitz estimate+early stopping}.
\end{proof}

\section{Convergence of higher-order schemes}
\label{sec:high-order-convergence}

We now apply the regularity estimates of
Section~\ref{sec:sec3-centered-posterior-moment-normal-form}
to the convergence of higher-order discretizations of the reverse
heat flow.
We first establish Wasserstein bounds for the truncated
Taylor scheme, then total-variation bounds for both Taylor
and fixed explicit Runge--Kutta schemes.
The resulting estimates also give the number of time steps
needed to guarantee a prescribed accuracy of the associated flow-based generative models.

Throughout this section, the schemes use the exact modified
score $\hat s$. The global bounds account for both Gaussian
initialization and numerical discretization.

\subsection{Wasserstein convergence of the truncated Taylor scheme}
\label{sec:global-convergence-taylor}

We couple the numerical trajectory with the exact reverse
probability flow.
Theorem~\ref{thm:weighted-estimates-Ljx} controls the local
Taylor remainder and the Lipschitz constant of each numerical
step. A discrete Gr\"onwall argument then combines these bounds
into a global $W_2$ error estimate.

Fix an integer $p\geq1$, an early-stopping time
$0<\delta\leq1$, and a terminal forward time $T>\delta$.
Write
\begin{equation*}
    b(t):=(1-e^{-t})^{-1},
    \qquad
    \Theta_R(t):=(1+R)b(t),
    \qquad
    \Lambda_{R,\delta}
    :=
    \frac{1+R}{\delta},
\end{equation*}
so that $\Theta_R$ is decreasing and
$\Theta_R(t)\le\Theta_R(\delta)\le(1-e^{-1})^{-1}\Lambda_{R,\delta}$
for $t\ge\delta$.
For a positive integer $N$, define the uniform reverse-time grid
\begin{equation*}
    h=\frac{T-\delta}{N},
    \qquad
    u_n=nh,
    \qquad
    t_n=T-u_n,
    \qquad
    n=0,\ldots,N.
\end{equation*}
Thus $u_n$ increases from $0$ to $T-\delta$, while the
corresponding forward diffusion time $t_n$ decreases from
$T$ to $\delta$.

Let $(Y_u)$ be the exact reverse flow in
\eqref{eq:backward-ode}, initialized with $Y_0\sim P_T$.
Its marginals satisfy $Y_{u_n}\sim P_{t_n}$, and in particular
$Y_{u_N}\sim P_\delta$.
Using the Taylor map defined in \eqref{eq:taylor-scheme},
let
\begin{equation*}
    Y_{n+1}^h
    =
    \Phi_h^{(p)}(u_n,Y_n^h),
    \qquad
    n=0,\ldots,N-1,
\end{equation*}
with
\begin{equation*}
    Y_0^h\sim\gamma_d:=N(0,I_d),
    \qquad
    Q_n^{(p)}:=\mathcal L(Y_n^h).
\end{equation*}

Finally, we define the moment bound,
\begin{equation*}
    \mathfrak m_\delta
    :=
    \sup_{\delta\leq t\leq T}
    \left(
        \int_{\mathbb R^d}|x|^2\,P_t(dx)
    \right)^{1/2},
\end{equation*}
which controls the term proportional to $|x|$
in \eqref{eq:truncation j-th estimate+early stopping},
applied with $j=p+1$ along the exact trajectory.

\begin{theorem}[Global convergence of the truncated Taylor scheme]
\label{thm:global convergence truncated Taylor}
Suppose that $P_0$ is supported in $B_R(0)$ and satisfies
Assumption~\ref{ass:uniform-edge-doubling}. Let $p\geq1$ be an integer, let $0<\delta\leq1$, and assume
that the step size satisfies
\begin{equation}\label{eq:taylor-step-stability-condition}
    h\Lambda_{R,\delta}^2\leq1.
\end{equation}
Then
\begin{equation}\label{eq:global-Taylor-bound}
\begin{aligned}
W_2\bigl(P_\delta,Q_N^{(p)}\bigr)
\leq
C_p
\exp\left(
    C_p\frac{(1+R)^2}{\delta}
\right)
\Bigg[
    &e^{-T/2}(R^2+d)^{1/2}
    \\
    &+
    h^p
    \left(
        \delta^{-(p+1)}\mathfrak m_\delta
        +
        \Lambda_{R,\delta}^{2p+1}
    \right)
\Bigg],
\end{aligned}
\end{equation}
where $C_p$ depends only on $p$ and $D_{\mathrm{UD}}$.

Moreover,
$\mathfrak m_\delta\leq R+\sqrt d$, and hence
\begin{equation}\label{eq:bounded support global Taylor error}
    W_2\bigl(P_\delta,Q_N^{(p)}\bigr)
    \leq
    C_{p,R,\delta}
    \left[
        e^{-T/2}(R^2+d)^{1/2}
        +
        h^p(1+\sqrt d)
    \right],
\end{equation}
where $C_{p,R,\delta}$ may also depend on $D_{\mathrm{UD}}$ but is independent of
the ambient dimension.

\end{theorem}

\begin{proof}
Choose an optimal coupling of
$Y_0\sim P_T$ and $Y_0^h\sim\gamma_d$.
Evolve $Y_u$ by the exact reverse flow and $Y_n^h$ by
the Taylor scheme on the same probability space, and set
\begin{equation*}
    e_n
    :=
    \left\lVert{}Y_{u_n}-Y_n^h\right\rVert_{L^2}.
\end{equation*}
Then $e_0=W_2(P_T,\gamma_d)$.
We first estimate the local Taylor remainder along the
exact trajectory, then control the propagation of the
coupled error.

\emph{Local truncation error.}
The integral Taylor expansion \eqref{eq:taylor-x} gives
\begin{equation*}
    Y_{u_{n+1}}
    =
    \Phi_h^{(p)}(u_n,Y_{u_n})
    +
    R_{p+1}^x(u_n,h),
\end{equation*}
where
\begin{equation*}
    R_{p+1}^x(u_n,h)
    =
    \int_0^h
    \frac{(h-r)^p}{p!}
    (L^{p+1}x)(t_n-r,Y_{u_n+r})\,dr.
\end{equation*}
For $0\leq r\leq h$, the corresponding forward time
satisfies
$
    t_n-r\in[t_{n+1},t_n]\subseteq[\delta,T].
$
Since $Y_{u_n+r}\sim P_{t_n-r}$, the definition of
$\mathfrak m_\delta$ gives
$
    \lVert{}Y_{u_n+r}\rVert_{L^2}
    \leq
    \mathfrak m_\delta.
$

Apply \eqref{eq:truncation j-th estimate+early stopping}
with $j=p+1$ and use Minkowski's inequality.
Bounding $(h-r)^p$ by $h^p$ and changing variables to
the forward time $t=t_n-r$ then gives
\begin{equation}\label{eq:local-Taylor-error-global-section}
\begin{aligned}
    \left\lVert{}R_{p+1}^x(u_n,h)\right\rVert_{L^2}
    \leq
    C_ph^p
    \int_{t_{n+1}}^{t_n}
    \left[
        e^{-t}\delta^{-(p+1)}\mathfrak m_\delta
        +
        e^{-t/2}\Lambda_{R,\delta}^{2p+1}
    \right]dt.
\end{aligned}
\end{equation}
The intervals $[t_{n+1},t_n]$ partition $[\delta,T]$.
Summing \eqref{eq:local-Taylor-error-global-section}
therefore yields
\begin{equation}\label{eq:sum-local-Taylor-error}
\begin{aligned}
    \sum_{n=0}^{N-1}
    \left\lVert{}R_{p+1}^x(u_n,h)\right\rVert_{L^2}
    &\leq
    C_ph^p
    \int_\delta^T
    \left[
        e^{-t}\delta^{-(p+1)}\mathfrak m_\delta
        +
        e^{-t/2}\Lambda_{R,\delta}^{2p+1}
    \right]dt
    \\
    &\leq
    C_ph^p
    \left[
        \delta^{-(p+1)}\mathfrak m_\delta
        +
        \Lambda_{R,\delta}^{2p+1}
    \right].
\end{aligned}
\end{equation}
The two time integrals are bounded by $1$ and $2$,
respectively. Thus summing the local defects introduces
no additional factor growing with $T$.

\emph{One-step stability.}
By \eqref{eq:j-th lipschitz estimate} and the
definition of the Taylor map, for every $x,y\in\mathbb R^d$,
\begin{equation*}
    |\Phi_h^{(p)}(u_n,x)-\Phi_h^{(p)}(u_n,y)|
    \leq
    \left[
        1+e^{-t_n}
        \sum_{j=1}^p\frac{C_j}{j!}
        \bigl(h\Theta_R(t_n)^2\bigr)^j
    \right]|x-y|.
\end{equation*}
Since $t_n\ge\delta$, the step-size condition
\eqref{eq:taylor-step-stability-condition} gives
$h\Theta_R(t_n)^2\le(1-e^{-1})^{-2}h\Lambda_{R,\delta}^2\le(1-e^{-1})^{-2}$,
hence
\begin{equation*}
    \sum_{j=1}^p\frac{C_j}{j!}
    \bigl(h\Theta_R(t_n)^2\bigr)^j
    \leq
    h\Theta_R(t_n)^2
    \sum_{j=1}^p\frac{C_j}{j!}(1-e^{-1})^{-2(j-1)}
    \leq C_p h\Theta_R(t_n)^2.
\end{equation*}
Consequently,
\begin{equation}
    |\Phi_h^{(p)}(u_n,x)-\Phi_h^{(p)}(u_n,y)|
    \leq
    \bigl(1+C_p h\,e^{-t_n}\Theta_R(t_n)^2\bigr)
    |x-y|.
\end{equation}

Subtracting the numerical update from the exact Taylor
expansion and taking $L^2$ norms therefore yields
\begin{equation}\label{eq:error-recursion-Taylor}
    e_{n+1}
    \leq
    \bigl(1+C_p h\,e^{-t_n}\Theta_R(t_n)^2\bigr)e_n
    +
    \lVert{}R_{p+1}^x(u_n,h)\rVert_{L^2}.
\end{equation}
This separates the propagation of the existing error from
the local defect introduced at the next step.

\emph{Accumulation of the error.}
Applying the discrete Gr\"onwall inequality to
\eqref{eq:error-recursion-Taylor} yields
\begin{equation*}
    e_N
    \leq
    \exp\left(
        C_p h
        \sum_{n=0}^{N-1}e^{-t_n}\Theta_R(t_n)^2
    \right)
    \left[
        e_0+
        \sum_{n=0}^{N-1}
        \lVert{}R_{p+1}^x(u_n,h)\rVert_{L^2}
    \right].
\end{equation*}
Since $t_{n+1}=t_n-h$ and $t\mapsto e^{-t}b(t)^2$ is decreasing,
\begin{equation*}
    h\sum_{n=0}^{N-1}e^{-t_n}b(t_n)^2
    \leq
    \sum_{n=0}^{N-1}
    \int_{t_{n+1}}^{t_n}e^{-t}b(t)^2\,dt
    =
    \int_\delta^T e^{-t}b(t)^2\,dt
    =
    b(\delta)-b(T)
    \leq
    \frac{1}{(1-e^{-1})\,\delta},
\end{equation*}
where we used $\frac{d}{dt}b(t)=-e^{-t}b(t)^2$ and
$1-e^{-\delta}\ge(1-e^{-1})\delta$ for $0<\delta\le1$.
Thus the accumulated stability factor is bounded by
$\exp\bigl(C_p(1+R)^2\delta^{-1}\bigr)$, independently
of $T$ and $N$. Combining this with
\eqref{eq:sum-local-Taylor-error} gives
\begin{equation}\label{eq:global strong error refined}
    e_N
    \leq
    C_p\exp\left(C_p(1+R)^2\delta^{-1}\right)
    \left[
        e_0+
        h^p\left(
            \delta^{-(p+1)}\mathfrak m_\delta
            +\Lambda_{R,\delta}^{2p+1}
        \right)
    \right].
\end{equation}

It remains to bound the initialization error. The OU
semigroup contracts $W_2$ and leaves $\gamma_d$ invariant, so
\begin{equation*}
    e_0
    =
    W_2(P_T,\gamma_d)
    \leq
    e^{-T/2}W_2(P_0,\gamma_d)
    \leq
    e^{-T/2}(R^2+d)^{1/2}.
\end{equation*}
For the last inequality, we couple $X_0\sim P_0$ independently
with $Z\sim\gamma_d$ and use
$\mathbb E|X_0-Z|^2=\mathbb E|X_0|^2+d\leq R^2+d$.

At the final time, $Y_{u_N}\sim P_\delta$ and
$Y_N^h\sim Q_N^{(p)}$. The coupling inequality and
\eqref{eq:global strong error refined} therefore imply
\begin{equation*}
\begin{aligned}
    W_2(P_\delta,Q_N^{(p)})
    &\leq e_N \\
    &\leq
    C_p\exp\left(C_p(1+R)^2\delta^{-1}\right)
    \left[
        e^{-T/2}(R^2+d)^{1/2}
        +
        h^p\left(
            \delta^{-(p+1)}\mathfrak m_\delta
            +\Lambda_{R,\delta}^{2p+1}
        \right)
    \right],
\end{aligned}
\end{equation*}
which proves \eqref{eq:global-Taylor-bound}.

Finally, the OU representation
$X_t=e^{-t/2}X_0+\sqrt{1-e^{-t}}\,Z$ gives
\begin{equation*}
    \left(\mathbb E|X_t|^2\right)^{1/2}
    \leq
    e^{-t/2}\left(\mathbb E|X_0|^2\right)^{1/2}
    +
    \sqrt{1-e^{-t}}\sqrt d
    \leq R+\sqrt d.
\end{equation*}
Hence $\mathfrak m_\delta\leq R+\sqrt d$. Substituting this
bound into \eqref{eq:global-Taylor-bound} and absorbing the
factors depending only on $p$, $R$, $\delta$, and
$D_{\mathrm{UD}}$ into $C_{p,R,\delta}$ proves
\eqref{eq:bounded support global Taylor error}.
\end{proof}

\subsection{Total-variation convergence of Taylor and Runge--Kutta schemes}
\label{sec:tv-comparison}

The Wasserstein analysis of
Section~\ref{sec:global-convergence-taylor} couples trajectories over
the whole reverse interval. Because $W_2$ is a metric on the state
space, each step multiplies the propagated error by the Lipschitz
constant of the step map, which produces the exponential prefactor in
Theorem~\ref{thm:global convergence truncated Taylor}. Total variation
is defined on measures alone: a common measurable pushforward does not
increase it, irrespective of the Lipschitz constant. We therefore
compare the laws transported by the exact and numerical \emph{one-step
maps}, bound their discrepancy by a change of variables along the
interpolation $y\mapsto y+\theta e(y)$
(Lemma~\ref{lem:endpoint-transport-bound}), and sum the local errors
along the reverse flow. The resulting prefactors are polynomial in
$\Lambda_{R,\delta}$.

The one-step viewpoint also covers explicit Runge--Kutta (RK) methods,
which evaluate only the velocity field and never form the material
derivatives $L^jx$ used by the Taylor scheme. For these we first
propagate the mixed directional estimates of
Lemma~\ref{lem:sec3-weighted-directional-bound} through the stages by
an affine-plus-remainder decomposition of each stage derivative
(Lemma~\ref{lem:weighted-rk-stage-regularity}); this is the step at
which a componentwise analysis would pay a factor of $d$ per stage, and
where the operator-norm form of
Lemma~\ref{lem:sec3-weighted-directional-bound} keeps the constants
dimension-free. The Taylor scheme requires only
Theorem~\ref{thm:weighted-estimates-Ljx}. Both local estimates are then
converted into one-step total-variation bounds and summed along the
reverse flow.

We continue to use $u$ for reverse time and $t$ for forward
OU time, and write
\begin{equation*}
    v(t,x):=\frac12\hat s(t,x),
\end{equation*}
and recall $b(t)=(1-e^{-t})^{-1}$ and $\Theta_R(t)=(1+R)b(t)$ from
Section~\ref{sec:global-convergence-taylor}.
For a step satisfying $0\le u<u+h\le T-\delta$, set
\begin{equation*}
    t:=T-u,
    \qquad
    s:=t-h,
    \qquad
    b_*:=b(s),
    \qquad
    \Theta_*:=\Theta_R(s).
\end{equation*}
Thus $t$ and $s$ are the forward times at the beginning and
end of the reverse step. The exact local reverse flow,
denoted by $\phi_{u,r}$ for $0\le r\le h$, satisfies
\begin{equation*}
    \partial_r\phi_{u,r}(x)
    =
    v\bigl(t-r,\phi_{u,r}(x)\bigr),
    \qquad
    \phi_{u,0}(x)=x.
\end{equation*}
Since $t-r\in[s,t]$, monotonicity gives
\begin{equation*}
    b(t-r)\le b_*,
    \qquad
    \Theta_R(t-r)\le\Theta_*,
    \qquad 0\le r\le h.
\end{equation*}

Throughout this subsection, $C_p$ denotes a generic finite
constant depending only on $p$, the doubling constant
$D_{\mathrm{UD}}$, and the fixed Runge--Kutta tableau when
applicable. We impose the one-step condition
\begin{equation}
    h\Theta_*^2\le c_p,
    \label{eq:section5-one-step-condition}
\end{equation}
where $c_p>0$ has the same parameter dependence as $C_p$ and is
decreased finitely many times in the proofs below; we keep the same
symbol throughout. In particular, these constants are independent of
$d$, $T$, $u$, $h$, and $N$.

Since $b_*\le\Theta_*^2$, this condition implies, for
$0\le r\le h$ and every $q\ge0$,
\begin{equation}
    r b_*^{q+1}\le c_p b_*^q,
    \qquad
    r\Theta_*^{q+2}\le c_p\Theta_*^q.
    \label{eq:section5-absorption}
\end{equation}
We will use these inequalities to control the additional
factors arising in the stage and flow estimates.

\subsubsection{Propagation of the weighted regularity through Runge--Kutta stages}

Differentiating a Runge--Kutta stage produces mixed time and
spatial derivatives of $v$, with derivatives of the stage
states inserted as spatial directions.
Lemma~\ref{lem:sec3-weighted-directional-bound} controls
contractions with the current evaluation point, which we
will use to bound the spatial growth of these terms.

Consider a fixed explicit $S$-stage Runge--Kutta method of
classical order $p$, with tableau $(a_{ij},b_i,c_i)$ and
stage nodes $c_i\in[0,1]$. For a reverse step starting at
$u$, write $t=T-u$ and define, for $0\le r\le h$,
\begin{align}
    X_i(r,x)
    &=x+r\sum_{j<i}a_{ij}K_j(r,x),
    \label{eq:rk-stage-interpolation-1}
    \\
    K_i(r,x)
    &=v\bigl(t-c_ir,X_i(r,x)\bigr),
    \qquad i=1,\ldots,S,
    \label{eq:rk-stage-interpolation-2}
    \\
    \Psi_{u,r}^{\mathrm{RK}}(x)
    &=x+r\sum_{i=1}^S b_iK_i(r,x).
    \label{eq:rk-stage-interpolation-3}
\end{align}
Here $r$ varies the step size while $u$ and $x$ remain fixed;
$\Psi_{u,h}^{\mathrm{RK}}$ is the numerical one-step map.
The assumption $c_i\in[0,1]$ ensures that every stage time
$t-c_ir$ lies in $[s,t]$, so the time bounds introduced above
apply at every stage.

The next lemma shows that derivatives of the stage slopes
grow at most linearly in $|x|$, while their Jacobians admit
bounds independent of $x$ and the ambient dimension.

\begin{lemma}[Weighted Runge--Kutta stage regularity]
\label{lem:weighted-rk-stage-regularity}
Suppose that $P_0$ is supported in $B_R(0)$ and satisfies
Assumption~\ref{ass:uniform-edge-doubling}. Fix an explicit
Runge--Kutta tableau with $c_i\in[0,1]$ and an integer
$p\ge1$. Let $0\le u<u+h\le T-\delta$ satisfy
\eqref{eq:section5-one-step-condition}, and use the notation
$s=T-u-h$, $b_*=b(s)$, and $\Theta_*=\Theta_R(s)$.

Then, for every stage $i$, every $x\in\mathbb R^d$,
every $0\le r\le h$, and every integer $0\le m\le p+1$,
\begin{align}
    &|\partial_r^mK_i(r,x)|
    \le C_p\left[
        e^{-s}b_*^{m+1}|x|
        +e^{-s/2}\Theta_*^{2m+1}
    \right],
    \label{eq:weighted-rk-stage-value}
    \\
    &\lVert{}D_x\partial_r^mK_i(r,x)\rVert_{\mathrm{op}}
    \le C_pe^{-s}\Theta_*^{2m+2}.
    \label{eq:weighted-rk-stage-jacobian}
\end{align}
For every integer $1\le m\le p+1$, the stage states satisfy
\begin{align}
    &|\partial_r^mX_i(r,x)|
    \le C_p\left[
        e^{-s}b_*^m|x|
        +e^{-s/2}\Theta_*^{2m-1}
    \right],
    \label{eq:weighted-rk-state-value}
    \\
    &\lVert{}D_x\partial_r^mX_i(r,x)\rVert_{\mathrm{op}}
    \le C_pe^{-s}\Theta_*^{2m}.
    \label{eq:weighted-rk-state-jacobian}
\end{align}
Moreover,
\begin{align}
    &|X_i(r,x)-x|
    \le C_pr\left[
        e^{-s}b_*|x|+e^{-s/2}\Theta_*
    \right],
    \label{eq:rk-stage-near-identity-value}
    \\
    &\lVert{}D_xX_i(r,x)-I\rVert_{\mathrm{op}}
    \le C_pr e^{-s}\Theta_*^2.
    \label{eq:rk-stage-near-identity-jacobian}
\end{align}
\end{lemma}

\begin{proof}
We will prove a slightly stronger statement by induction over the stages.
For every already constructed stage $j$ and $0\le m\le p+1$, we construct a decomposition
\begin{equation*}
    \partial_r^mK_j(r,x)=\beta_{j,m}(r)x+\kappa_{j,m}(r,x),
\end{equation*}
with
\begin{equation}
\begin{aligned}
    |\beta_{j,m}(r)|&\le C_pe^{-s}b_*^{m+1},\\
    |\kappa_{j,m}(r,x)|&\le C_pe^{-s/2}\Theta_*^{2m+1},\\
    \lVert{}D_x\kappa_{j,m}(r,x)\rVert_{\mathrm{op}}
        &\le C_pe^{-s}\Theta_*^{2m+2}.
\end{aligned}
\label{eq:rk-induction-profile}
\end{equation}
These bounds immediately imply
\eqref{eq:weighted-rk-stage-value}--\eqref{eq:weighted-rk-stage-jacobian}.

For the first stage, $X_1(r,x)=x$, so the estimates for
the stage state are immediate. To estimate
$K_1(r,x)=v(t-c_1r,x)$, write
\begin{equation*}
    v(t,x)
    =
    \frac12a(t)x+\frac12v_c(t,x),
    \qquad
    v_c(t,x):=\lambda(t)\bar y_{x,t}.
\end{equation*}
Then the required decomposition is given by
\begin{equation*}
    \beta_{1,m}(r)
    :=\frac12(-c_1)^m\partial_\tau^m a(t-c_1r),
    \qquad
    \kappa_{1,m}(r,x)
    :=\frac12(-c_1)^m\partial_\tau^m v_c(t-c_1r,x).
\end{equation*}
The coefficient estimate~\eqref{eq:partial t estimate for a and lambda} for $a$ gives the first bound for $\beta_{1,m}$ in
\eqref{eq:rk-induction-profile}. Since
$v_c\in\mathfrak W_{1,1,1}$, Lemma~\ref{lem:sec3-weighted-directional-bound}
with $q=m$, $k=0$ and, respectively, $\ell=0$ and $\ell=1$ gives the
remaining two bounds for $\kappa_{1,m}$. Thus the induction profile holds for stage $j=1$.

Fix $i>1$ and assume it holds for every $j<i$. We first derive the structure
of the stage state directly from \eqref{eq:rk-stage-interpolation-1}. For
$m=0$,
\begin{equation*}
\begin{aligned}
    X_i(r,x)
    =x+r\sum_{j<i}a_{ij}
      \bigl(\beta_{j,0}(r)x+\kappa_{j,0}(r,x)\bigr)=\gamma_i(r)x+\xi_{i,0}(r,x),
\end{aligned}
\end{equation*}
where
\begin{equation*}
    \gamma_i(r):=1+r\sum_{j<i}a_{ij}\beta_{j,0}(r),
    \qquad
    \xi_{i,0}(r,x):=r\sum_{j<i}a_{ij}\kappa_{j,0}(r,x).
\end{equation*}
By the induction hypothesis, $\beta_{j,0}(r)$ and $\kappa_{j,0}(r,x)$ satisfy~\eqref{eq:rk-induction-profile}. Hence, we obtain the $m=0$ bounds for $X_i(r,x)$,
\begin{equation}\label{eq:m=0 bounds for X_i(r,x)}
\begin{aligned}
        &|\gamma_i(r)-1|\le C_pr e^{-s}b_*,\\
    &|\xi_{i,0}(r,x)|\le C_pr e^{-s/2}\Theta_*,\\
    &\lVert{}D_x\xi_{i,0}(r,x)\rVert_{\mathrm{op}}
    \le C_pr e^{-s}\Theta_*^2.
\end{aligned}
\end{equation}
This gives \eqref{eq:rk-stage-near-identity-value} and
\eqref{eq:rk-stage-near-identity-jacobian}.

For $1\le m\le p+1$, differentiating
\eqref{eq:rk-stage-interpolation-1} gives
\begin{equation*}
    \partial_r^mX_i
    =m\sum_{j<i}a_{ij}\partial_r^{m-1}K_j
     +r\sum_{j<i}a_{ij}\partial_r^mK_j.
\end{equation*}
Using the induction decomposition, we write
\begin{equation*}
    \partial_r^mX_i(r,x)=\alpha_{i,m}(r)x+\xi_{i,m}(r,x),
\end{equation*}
where
\begin{equation*}
\begin{aligned}
    \alpha_{i,m}(r)
    &:=m\sum_{j<i}a_{ij}\beta_{j,m-1}(r)
       +r\sum_{j<i}a_{ij}\beta_{j,m}(r),\\
    \xi_{i,m}(r,x)
    &:=m\sum_{j<i}a_{ij}\kappa_{j,m-1}(r,x)
       +r\sum_{j<i}a_{ij}\kappa_{j,m}(r,x).
\end{aligned}
\end{equation*}
The induction bounds~\eqref{eq:rk-induction-profile} and the absorption rules
\eqref{eq:section5-absorption} yield
\begin{equation}
\begin{aligned}
    |\alpha_{i,m}(r)|&\le C_pe^{-s}b_*^m,\\
    |\xi_{i,m}(r,x)|&\le C_pe^{-s/2}\Theta_*^{2m-1},\\
    \lVert{}D_x\xi_{i,m}(r,x)\rVert_{\mathrm{op}}
        &\le C_pe^{-s}\Theta_*^{2m}.
\end{aligned}
\label{eq:rk-state-affine-profile}
\end{equation}
Since $D_x\partial_r^mX_i=\alpha_{i,m}I+D_x\xi_{i,m}$ and
$b_*^m\le\Theta_*^{2m}$, this proves
\eqref{eq:weighted-rk-state-value}--\eqref{eq:weighted-rk-state-jacobian}.

The derivatives of $v$ in the stage expansion are evaluated
at $y=X_i(r,x)$. To apply
\eqref{eq:higher-order-mixed-directional-bound}, we therefore
rewrite each inserted stage derivative as a scalar multiple
of $y$ plus a bounded remainder.

The preceding bounds and
\eqref{eq:section5-one-step-condition} allow us to choose
$c_p$ sufficiently small so that, uniformly in $r$ and $x$,
\begin{equation*}
    |\gamma_i(r)-1|\le\frac12,
    \qquad
    \lVert{}D_xX_i(r,x)-I\rVert_{\mathrm{op}}\le\frac12.
\end{equation*}
For each fixed $r$, the map $x\mapsto X_i(r,x)-x$ is
therefore globally $1/2$-Lipschitz. For every
$y\in\mathbb R^d$, the equation $X_i(r,x)=y$ is equivalent to
\begin{equation*}
    x=y-\bigl(X_i(r,x)-x\bigr),
\end{equation*}
which has a unique solution by the contraction mapping
theorem. Thus $X_i(r,\cdot)$ is globally invertible, and
\begin{equation*}
    |\gamma_i(r)|^{-1}\le2,
    \qquad
    \lVert(D_xX_i(r,x))^{-1}\rVert_{\mathrm{op}}\le2.
\end{equation*}
We now rewrite the derivatives already computed at fixed
$x$ in terms of $y$. Since
$
    y=\gamma_i(r)x+\xi_{i,0}(r,x),
$
we have
$
    x=\gamma_i(r)^{-1}
    \bigl(y-\xi_{i,0}(r,x)\bigr).
$
Substituting this into the preceding affine decomposition gives
\begin{equation}
    \partial_r^mX_i(r,x)
    =\widetilde\alpha_{i,m}(r)y+\widetilde\xi_{i,m}(r,y),
    \label{eq:rk-X-current-point-decomposition}
\end{equation}
where, with $x=X_i(r,\cdot)^{-1}(y)$,
\begin{equation*}
    \widetilde\alpha_{i,m}(r)
    :=\frac{\alpha_{i,m}(r)}{\gamma_i(r)},
    \qquad
    \widetilde\xi_{i,m}(r,y)
    :=\xi_{i,m}(r,x)
      -\widetilde\alpha_{i,m}(r)\xi_{i,0}(r,x).
\end{equation*}
Using \eqref{eq:rk-state-affine-profile}, \eqref{eq:m=0 bounds for X_i(r,x)}, the inverse bound $|\gamma_i(r)|^{-1}\le2$, and step absorption rules \eqref{eq:section5-absorption}, we obtain
\begin{equation}
\begin{aligned}
    |\widetilde\alpha_{i,m}(r)|&\le C_pe^{-s}b_*^m,\\
    |\widetilde\xi_{i,m}(r,y)|&\le C_pe^{-s/2}\Theta_*^{2m-1},\\
    \lVert{}D_y\widetilde\xi_{i,m}(r,y)\rVert_{\mathrm{op}}
        &\le C_pe^{-s}\Theta_*^{2m}.
\end{aligned}
\label{eq:rk-current-point-profile}
\end{equation}

We now differentiate \eqref{eq:rk-stage-interpolation-2}. Fa\`a di Bruno's 
formula expresses $\partial_r^mK_i$ as a finite sum of terms
\begin{equation*}
     C(-c_i)^q
    (\partial_t^qD_y^\ell v)(t-c_ir,y)
    \bigl[
        \partial_r^{m_1}X_i(r,x),\ldots,
        \partial_r^{m_\ell}X_i(r,x)
    \bigr],
    \qquad y=X_i(r,x),
\end{equation*}
where $q+m_1+\cdots+m_\ell=m$ and $m_j\ge1$; the powers of $c_i$ and
combinatorial coefficients are absorbed into $C_p$.

Recall \[v(t,y)
    =
    \frac12a(t)y+\frac12v_c(t,y).\] Consider first the contribution of $v_c$.
Expand each inserted stage derivative using
\eqref{eq:rk-X-current-point-decomposition}.
For a subset $I\subseteq\{1,\ldots,\ell\}$, select
$\widetilde\alpha_{i,m_j}y$ when $j\in I$ and
$\widetilde\xi_{i,m_j}$ when $j\in I^c$, and set
\begin{equation*}
    k:=|I|,
    \qquad
    M_I:=\sum_{j\in I}m_j,
    \qquad
    M_{I^c}:=\sum_{j\in I^c}m_j.
\end{equation*}
In particular,
$
    M_I\ge k,
$ and $
    q+M_I+M_{I^c}=m.
$
After normalizing each nonzero remainder as
$\widetilde\xi_{i,m_j}=|\widetilde\xi_{i,m_j}|w_j$, $|w_j|=1$, the
corresponding contribution is bounded by
\begin{equation*}
\begin{aligned}
&C_p\left(\prod_{j\in I}|\widetilde\alpha_{i,m_j}|\right)
\left(\prod_{j\notin I}|\widetilde\xi_{i,m_j}|\right)\\
&\qquad\times
\left|
\partial_\tau^qD_y^\ell v_c(t-c_ir,y)
\bigl[\underbrace{y,\ldots,y}_{k\ \mathrm{times}},
      w_j,\ j\notin I\bigr]
\right|.
\end{aligned}
\end{equation*}
Since $v_c\in\mathfrak W_{1,1,1}$, we apply
\eqref{eq:higher-order-mixed-directional-bound} with
$k$ position directions and $\ell-k$ unit directions.
Using $t-c_ir\ge s$ gives
\begin{equation*}
\left|
\partial_\tau^qD_y^\ell v_c(t-c_ir,y)
\bigl[\underbrace{y,\ldots,y}_{k\ \mathrm{times}},
      w_j,\ j\notin I\bigr]
\right|
\le C_p e^{-(1+\ell-k)s/2}
  b_*^{1+2q+\ell}(1+R)^{1+2q+\ell+k}.
\end{equation*}
We combine this bound with
\eqref{eq:rk-current-point-profile}.
The scalar coefficients contribute $b_*^{M_I}$, while
the remainder factors contribute
$\Theta_*^{2M_{I^c}-(\ell-k)}$.
Since $\Theta_*=(1+R)b_*$, the total power of $b_*$ is
\begin{equation*}
\begin{aligned}
    1+2q+\ell+M_I+2M_{I^c}-(\ell-k)
=
    2m+1-(M_I-k)
\le 2m+1.
\end{aligned}
\end{equation*}
The total power of $1+R$ is
\begin{equation*}
\begin{aligned}
    1+2q+\ell+k+2M_{I^c}-(\ell-k)
    =
    2m+1-2(M_I-k)
\le 2m+1.
\end{aligned}
\end{equation*}
The exponential decay is stronger
than $e^{-s/2}$, so every centered Fa\`a di Bruno term is bounded by
$C_pe^{-s/2}\Theta_*^{2m+1}$. Summing the finitely many terms preserves this bound.

For the affine part $\frac12a(\tau)y$, only $\ell=0$ and $\ell=1$ occur.
The term with $\ell=0$ is
\[
    \frac12(-c_i)^m
    (\partial_t^ma)(t-c_ir)y.
\]
For $\ell=1$, we have $q+m_1=m$ with $m_1\ge1$.
Using \eqref{eq:rk-X-current-point-decomposition},
the corresponding terms take the form
\begin{equation*}
    C(-c_i)^q(\partial_t^qa)(t-c_ir)
    \left[
        \widetilde\alpha_{i,m_1}(r)y
        +
        \widetilde\xi_{i,m_1}(r,y)
    \right],
\end{equation*}
Using the coefficient estimate~\eqref{eq:partial t estimate for a and lambda} for $a$ and~\eqref{eq:rk-current-point-profile} for $\widetilde\alpha$ and $\widetilde\xi$, we may
therefore collect all affine contributions into $A_{i,m}(r)y$ and all
remaining contributions into $B_{i,m}(r,y)$, with
\begin{equation}
\begin{aligned}
    \partial_r^mK_i(r,x)&=A_{i,m}(r)y+B_{i,m}(r,y),\\
    |A_{i,m}(r)|&\le C_pe^{-s}b_*^{m+1},\\
    |B_{i,m}(r,y)|&\le C_pe^{-s/2}\Theta_*^{2m+1}.
\end{aligned}
\label{eq:rk-K-current-profile}
\end{equation}

To estimate $D_yB_{i,m}$, differentiate the multilinear expansion in a unit
direction. The derivative either lands on the outer mixed derivative of
$v_c$, on an explicit ambient direction $y$, or on a remainder
$\widetilde\xi_{i,m_j}$. In the first two cases
Lemma~\ref{lem:sec3-weighted-directional-bound} contributes one additional fixed
unit direction. In the third case, \eqref{eq:rk-current-point-profile}
replaces the factor $e^{-s/2}\Theta_*^{2m_j-1}$ by
$e^{-s}\Theta_*^{2m_j}$. The same bookkeeping therefore gives
\begin{equation}\label{eq:D_yB bound}
    \lVert{}D_yB_{i,m}(r,y)\rVert_{\mathrm{op}}
    \le C_pe^{-s}\Theta_*^{2m+2}.
\end{equation}

Finally, return to the initial point $x$ using
$y=\gamma_i(r)x+\xi_{i,0}(r,x)$. Define
\begin{equation*}
    \beta_{i,m}(r):=A_{i,m}(r)\gamma_i(r),
    \qquad
    \kappa_{i,m}(r,x)
    :=A_{i,m}(r)\xi_{i,0}(r,x)
      +B_{i,m}\bigl(r,X_i(r,x)\bigr).
\end{equation*}
Then $\partial_r^mK_i=\beta_{i,m}x+\kappa_{i,m}$. Since
$|\gamma_i|\le3/2$ and $\lVert{}D_xX_i\rVert_{\mathrm{op}}\le3/2$, \eqref{eq:rk-K-current-profile}, \eqref{eq:D_yB bound},
\eqref{eq:m=0 bounds for X_i(r,x)}, and \eqref{eq:section5-absorption} give exactly
\eqref{eq:rk-induction-profile} for stage $i$. This closes the induction and
proves the lemma.
\end{proof}

\subsubsection{Weighted local defects and endpoint transport}

We next compare the Taylor and Runge--Kutta maps with the
exact reverse flow. The following lemma provides the
exact-flow derivative bounds needed for the local
remainders, together with bounds on the inverse map
for the endpoint change of variables.

\begin{lemma}[Exact one-step flow bounds]
\label{lem:section5-exact-step-bounds}
Suppose that $P_0$ is supported in $B_R(0)$ and satisfies
Assumption~\ref{ass:uniform-edge-doubling}.
Fix an integer $p\ge1$, and let
$0\le u<u+h\le T-\delta$ satisfy
\eqref{eq:section5-one-step-condition}.

For every $0\le r\le h$, the map
$\phi_{u,r}:\mathbb R^d\to\mathbb R^d$ is a $C^1$
diffeomorphism. For all $x,y\in\mathbb R^d$,
\begin{equation*}
\begin{aligned}
    &|\phi_{u,r}(x)|
    \le C_p\left[
        |x|+r e^{-s/2}\Theta_*
    \right],
    \\
    &|\phi_{u,r}(x)-x|
    \le C_pr\left[
        e^{-s}b_*|x|
        +e^{-s/2}\Theta_*
    \right],\\
    &\lVert{}D_x\phi_{u,r}(x)\rVert_{\mathrm{op}}
    +
    \lVert{}D_y\phi_{u,r}^{-1}(y)\rVert_{\mathrm{op}}
    \le C_p.
    \end{aligned}
\end{equation*}
At the end of the step, the inverse map also satisfies
\begin{equation*}
    |\phi_{u,h}^{-1}(y)|
    \le C_p\left[
        |y|+h e^{-s/2}\Theta_*
    \right].
\end{equation*}
The derivatives with respect to $r$ satisfy
\begin{equation*}
\begin{aligned}
    &|\partial_r^{p+1}\phi_{u,r}(x)|
    \le C_p\left[
        e^{-s}b_*^{p+1}|x|
        +e^{-s/2}\Theta_*^{2p+1}
    \right],
    \\
    &\lVert{}D_x\partial_r^{p+1}\phi_{u,r}(x)\rVert_{\mathrm{op}}
    \le C_pe^{-s}\Theta_*^{2p+2}.
\end{aligned}
\end{equation*}
\end{lemma}

\begin{proof}
Since $v=Lx$, the estimates
\eqref{eq:truncation j-th estimate} and
\eqref{eq:j-th lipschitz estimate} with $j=1$ give,
for $0\le r\le h$ and every $z\in\mathbb R^d$,
\begin{equation*}
\begin{aligned}
    &|v(t-r,z)|
    \le C\left[
        e^{-s}b_*|z|+e^{-s/2}\Theta_*
    \right],
    \\
    &\lVert{}D_xv(t-r,z)\rVert_{\mathrm{op}}
    \le Ce^{-s}\Theta_*^2.
\end{aligned}
\end{equation*}
These uniform growth and Lipschitz bounds ensure that
solutions exist uniquely throughout the step in both
time directions. Smoothness at positive forward time
then implies that $\phi_{u,r}$ is a $C^1$ diffeomorphism.

The integral equation for the flow is
\begin{equation}\label{eq:integral equation for the flow}
    \phi_{u,r}(x)
    =
    x+\int_0^r
    v\bigl(t-\rho,\phi_{u,\rho}(x)\bigr)\,d\rho.
\end{equation}
Applying Gr\"onwall's inequality with the velocity bound gives
\begin{equation}\label{eq:trajectory bound}
\begin{aligned}
    |\phi_{u,r}(x)|
    &\le
    \exp\left(Cr e^{-s}b_*\right)
    \left[
        |x|+Cr e^{-s/2}\Theta_*
    \right]
    \\
    &\le
    C_p\left[
        |x|+r e^{-s/2}\Theta_*
    \right],
\end{aligned}
\end{equation}
where we used
$r e^{-s}b_*\le h\Theta_*^2\le c_p$.
Substituting this estimate back into the integral equation~\eqref{eq:integral equation for the flow}
yields
\begin{equation}\label{eq:displacement estimate}
\begin{aligned}
    |\phi_{u,r}(x)-x|
    &\le
    C_pr\left[
        e^{-s}b_*|x|+e^{-s/2}\Theta_*
    \right]
    +
    C_pr^2e^{-3s/2}b_*\Theta_*
    \\
    &\le
    C_pr\left[
        e^{-s}b_*|x|+e^{-s/2}\Theta_*
    \right].
\end{aligned}
\end{equation}
The last step follows from
\eqref{eq:section5-absorption} with $q=0$.

We next estimate the spatial Jacobian and its inverse.
Differentiating~\eqref{eq:integral equation for the flow} gives
\begin{equation*}
\begin{aligned}
    \partial_rD_x\phi_{u,r}(x)
    &=
    D_xv\bigl(t-r,\phi_{u,r}(x)\bigr)
    D_x\phi_{u,r}(x),
    \\
    \partial_r\bigl(D_x\phi_{u,r}(x)\bigr)^{-1}
    &=
    -\bigl(D_x\phi_{u,r}(x)\bigr)^{-1}
    D_xv\bigl(t-r,\phi_{u,r}(x)\bigr),
\end{aligned}
\end{equation*}
both with initial matrix $I$ at $r=0$.
The inverse matrix is the derivative of the inverse flow
at the corresponding endpoint. Gr\"onwall's inequality
applied to these two equations therefore gives,
uniformly in $x$ and $y$,
\begin{equation}\label{eq:Jacobian bound}
\begin{aligned}
    \lVert{}D_x\phi_{u,r}(x)\rVert_{\mathrm{op}}
    +
    \lVert{}D_y\phi_{u,r}^{-1}(y)\rVert_{\mathrm{op}}
    \le
    2\exp\left(Cr e^{-s}\Theta_*^2\right)\le C_p.
\end{aligned}
\end{equation}

For the value of the inverse map, let $y=\phi_{u,h}(x)$.
The displacement estimate~\eqref{eq:displacement estimate} gives
\begin{equation*}
    |x|
    \le
    |y|
    +
    C_ph e^{-s}b_*|x|
    +
    C_ph e^{-s/2}\Theta_*.
\end{equation*}
Choosing $c_p$ sufficiently small ensures that
$
    C_ph e^{-s}b_*
    \le C_ph\Theta_*^2
    \le\frac12.
$
Absorbing the term containing $|x|$ gives
\begin{equation*}
    |\phi_{u,h}^{-1}(y)|
    \le C_p\left[
        |y|+h e^{-s/2}\Theta_*
    \right].
\end{equation*}

It remains to estimate the higher derivatives.
Along the exact reverse flow, differentiation in $r$
is given by the material derivative $L$. Hence
\begin{equation*}
    \partial_r^{p+1}\phi_{u,r}(x)
    =
    (L^{p+1}x)
    \bigl(t-r,\phi_{u,r}(x)\bigr).
\end{equation*}
Applying \eqref{eq:truncation j-th estimate} with
$j=p+1$ and using the trajectory bound~\eqref{eq:trajectory bound} yields
\begin{equation*}
\begin{aligned}
    |\partial_r^{p+1}\phi_{u,r}(x)|
    &\le
    C_p\left[
        e^{-s}b_*^{p+1}|\phi_{u,r}(x)|
        +e^{-s/2}\Theta_*^{2p+1}
    \right]
    \\
    &\le
    C_p\left[
        e^{-s}b_*^{p+1}|x|
        +
        r e^{-3s/2}b_*^{p+1}\Theta_*
        +
        e^{-s/2}\Theta_*^{2p+1}
    \right]
    \\
    &\le
    C_p\left[
        e^{-s}b_*^{p+1}|x|
        +e^{-s/2}\Theta_*^{2p+1}
    \right].
\end{aligned}
\end{equation*}
Here the additional term is absorbed using
\eqref{eq:section5-absorption} with $q=p$:
\begin{equation*}
    r b_*^{p+1}\Theta_*
    \le c_p b_*^p\Theta_*
    \le c_p\Theta_*^{2p+1}.
\end{equation*}

Finally, differentiation with respect to the initial point
gives
\begin{equation*}
\begin{aligned}
    D_x\partial_r^{p+1}\phi_{u,r}(x)
    =
    D_x(L^{p+1}x)
    \bigl(t-r,\phi_{u,r}(x)\bigr)
    D_x\phi_{u,r}(x).
\end{aligned}
\end{equation*}
Combining \eqref{eq:j-th lipschitz estimate} with
$j=p+1$ and the Jacobian bound~\eqref{eq:Jacobian bound} proved above gives
\begin{equation*}
    \lVert{}D_x\partial_r^{p+1}\phi_{u,r}(x)\rVert_{\mathrm{op}}
    \le C_pe^{-s}\Theta_*^{2p+2}.
\end{equation*}
This completes the proof.
\end{proof}

We now use the flow and stage estimates to bound the local
errors of both schemes. Write
\begin{equation*}
    \Psi_{u,h}^{\mathrm T}(x)
    :=
    \Phi_h^{(p)}(u,x)
\end{equation*}
for the order-$p$ Taylor map, and let
$\Psi_{u,h}^{\mathrm{RK}}$ be the Runge--Kutta map defined
in \eqref{eq:rk-stage-interpolation-3}.
For either method, define the local error by
\begin{equation*}
    E_{u,h}(x)
    :=
    \Psi_{u,h}(x)-\phi_{u,h}(x).
\end{equation*}

\begin{lemma}[Weighted $C^1$ local defects]
\label{lem:sec5-weighted-C1-local-defects}
Suppose that $P_0$ is supported in $B_R(0)$ and satisfies
Assumption~\ref{ass:uniform-edge-doubling}.
Let $0\le u<u+h\le T-\delta$ satisfy
\eqref{eq:section5-one-step-condition}.
Then the order-$p$ Taylor scheme and the fixed explicit
Runge--Kutta method of classical order $p$ described above
satisfy, for every $x\in\mathbb R^d$,
\begin{align}
    |E_{u,h}(x)|
    &\le C_ph^{p+1}\left[
        e^{-s}b_*^{p+1}|x|
        +e^{-s/2}\Theta_*^{2p+1}
    \right],
    \label{eq:weighted-C1-defect-value}
    \\
    \lVert{}D_xE_{u,h}(x)\rVert_{\mathrm{op}}
    &\le C_ph^{p+1}e^{-s}\Theta_*^{2p+2}.
    \label{eq:weighted-C1-defect-jacobian}
\end{align}
\end{lemma}

\begin{proof}
For the Taylor scheme, the integral remainder gives
\begin{equation*}
    E_{u,h}^{\mathrm T}(x)
    =
    -\int_0^h
    \frac{(h-r)^p}{p!}
    \partial_r^{p+1}\phi_{u,r}(x)\,dr.
\end{equation*}
Applying Lemma~\ref{lem:section5-exact-step-bounds} to
the integrand proves the value estimate~\eqref{eq:weighted-C1-defect-value}. Differentiating
under the integral and applying the corresponding
Jacobian bound in Lemma~\ref{lem:section5-exact-step-bounds} proves the spatial derivative estimate~\eqref{eq:weighted-C1-defect-jacobian}.
In both cases, integration contributes the factor
\begin{equation*}
    \int_0^h\frac{(h-r)^p}{p!}\,dr
    =
    \frac{h^{p+1}}{(p+1)!}.
\end{equation*}

For the Runge--Kutta method, classical order $p$ gives
\begin{equation*}
    \left.
        \partial_r^m\Psi_{u,r}^{\mathrm{RK}}(x)
    \right|_{r=0}
    =
    \left.
        \partial_r^m\phi_{u,r}(x)
    \right|_{r=0},
    \qquad 0\le m\le p.
\end{equation*}
These equalities hold as identities in $x$, so they
remain valid after spatial differentiation.
Taylor's formula therefore yields
\begin{equation*}
    E_{u,h}^{\mathrm{RK}}(x)
    =
    \int_0^h
    \frac{(h-r)^p}{p!}
    \left[
        \partial_r^{p+1}\Psi_{u,r}^{\mathrm{RK}}(x)
        -
        \partial_r^{p+1}\phi_{u,r}(x)
    \right]dr.
\end{equation*}

Differentiating \eqref{eq:rk-stage-interpolation-3} gives
\begin{equation*}
\begin{aligned}
    \partial_r^{p+1}\Psi_{u,r}^{\mathrm{RK}}(x)
    ={}
    (p+1)\sum_{i=1}^S
    b_i\partial_r^pK_i(r,x)+
    r\sum_{i=1}^S
    b_i\partial_r^{p+1}K_i(r,x).
\end{aligned}
\end{equation*}
We apply \eqref{eq:weighted-rk-stage-value} and
\eqref{eq:weighted-rk-stage-jacobian} with
$m=p$ and $m=p+1$.
The terms multiplied by $r$ are controlled by
\eqref{eq:section5-absorption}, which gives
\begin{equation*}
\begin{aligned}
    r b_*^{p+2}
    &\le c_p b_*^{p+1},
    \\
    r\Theta_*^{2p+3}
    &\le c_p\Theta_*^{2p+1},
    \\
    r\Theta_*^{2p+4}
    &\le c_p\Theta_*^{2p+2}.
\end{aligned}
\end{equation*}
Consequently,
\begin{equation*}
\begin{aligned}
    |\partial_r^{p+1}\Psi_{u,r}^{\mathrm{RK}}(x)|
    &\le C_p\left[
        e^{-s}b_*^{p+1}|x|
        +e^{-s/2}\Theta_*^{2p+1}
    \right],
    \\
    \lVert{}D_x\partial_r^{p+1}
        \Psi_{u,r}^{\mathrm{RK}}(x)\rVert_{\mathrm{op}}
    &\le C_pe^{-s}\Theta_*^{2p+2}.
\end{aligned}
\end{equation*}
Combining these estimates and applying
Lemma~\ref{lem:section5-exact-step-bounds} for $\partial_r^{p+1}\phi_{u,r}(x)$ in the integral
remainder, and differentiating under the integral for
the spatial estimate, proves
\eqref{eq:weighted-C1-defect-value} and
\eqref{eq:weighted-C1-defect-jacobian}.
\end{proof}

To compare the endpoint laws, we now express the numerical
transport as a correction of the exact flow.

\begin{lemma}[Endpoint transport bound]
\label{lem:endpoint-transport-bound}
Suppose that $P_0$ is supported in $B_R(0)$ and satisfies
Assumption~\ref{ass:uniform-edge-doubling}.
Let $0\le u<u+h\le T-\delta$ satisfy
\eqref{eq:section5-one-step-condition}, and let
$\Psi_{u,h}$ be a $C^1$ one-step map whose local error
$E_{u,h}=\Psi_{u,h}-\phi_{u,h}$ satisfies
\eqref{eq:weighted-C1-defect-value} and
\eqref{eq:weighted-C1-defect-jacobian}.
Writing $t=T-u$ and $s=t-h$, we have
\begin{equation}
    \operatorname{TV}\left(
        (\Psi_{u,h})_\#P_t,P_s
    \right)
    \le C_ph^{p+1}\mathfrak D_p(s),
    \label{eq:one-step-weighted-TV}
\end{equation}
where, for $t>0$,
\begin{equation}
\begin{aligned}
    \mathfrak D_p(t)
    :={}&
    \sqrt d\,(R+\sqrt d)
    e^{-t}b(t)^{p+3/2}
    \\
    &+
    \sqrt{d\,b(t)}
    e^{-t/2}\Theta_R(t)^{2p+1}
    \\
    &+
    d\,e^{-t}\Theta_R(t)^{2p+2}.
\end{aligned}
\label{eq:weighted-TV-density}
\end{equation}
\end{lemma}

\begin{proof}
Define
$
    G:=\Psi_{u,h}\circ\phi_{u,h}^{-1},
    \
    e:=E_{u,h}\circ\phi_{u,h}^{-1},
$
so that $G(y)=y+e(y)$.
Since the exact flow satisfies
$
    (\phi_{u,h})_\#P_t=P_s,
$
we have
$
    (\Psi_{u,h})_\#P_t=G_\#P_s.
$
It therefore remains to compare $G_\#P_s$ with $P_s$.

We first estimate the correction in the endpoint variable.
Set $x=\phi_{u,h}^{-1}(y)$, so that $e(y)=E_{u,h}(x)$.
Lemma~\ref{lem:section5-exact-step-bounds} gives
\begin{equation}\label{eq:temp estimate in lem4.5}
\begin{aligned}
    &|x|
    \le C_p\left[
        |y|+h e^{-s/2}\Theta_*
    \right],
    \\
    &\lVert{}D_y\phi_{u,h}^{-1}(y)\rVert_{\mathrm{op}}
    \le C_p.
\end{aligned}
\end{equation}
Substituting the first estimate in~\eqref{eq:temp estimate in lem4.5} into
\eqref{eq:weighted-C1-defect-value} yields
\begin{equation}
\begin{aligned}
    |e(y)|
    &\le C_ph^{p+1}\left[
        e^{-s}b_*^{p+1}|y|
        +
        h e^{-3s/2}b_*^{p+1}\Theta_*
        +
        e^{-s/2}\Theta_*^{2p+1}
    \right]
    \\
    &\le C_ph^{p+1}\left[
        e^{-s}b_*^{p+1}|y|
        +
        e^{-s/2}\Theta_*^{2p+1}
    \right].
\end{aligned}
\label{eq:endpoint-correction-value}
\end{equation}
The last step uses
\eqref{eq:section5-absorption} with $q=p$, since
\begin{equation*}
    h b_*^{p+1}\Theta_*
    \le c_p b_*^p\Theta_*
    \le c_p\Theta_*^{2p+1}.
\end{equation*}

For the spatial derivative, the chain rule gives
$
    D_ye(y)
    =
    D_xE_{u,h}(x)
    D_y\phi_{u,h}^{-1}(y).
$
Applying \eqref{eq:weighted-C1-defect-jacobian} and the
inverse Jacobian bound in~\eqref{eq:temp estimate in lem4.5}, we obtain
\begin{equation}
    \lVert{}D_ye(y)\rVert_{\mathrm{op}}
    \le
    C_ph^{p+1}e^{-s}\Theta_*^{2p+2}.
    \label{eq:endpoint-correction-jacobian}
\end{equation}
By \eqref{eq:section5-one-step-condition} and
\eqref{eq:endpoint-correction-jacobian}, after decreasing $c_p$
if necessary so that $c_p\leq(4C_p)^{-1/(p+1)}$, we obtain
\begin{equation*}
\begin{aligned}
    \sup_{y\in\mathbb R^d}
    \lVert{}D_ye(y)\rVert_{\mathrm{op}}
\le
    C_pe^{-s}\bigl(h\Theta_*^2\bigr)^{p+1}
\le C_pc_p^{p+1}
    \le\frac14.
\end{aligned}
\end{equation*}
For $0\le\theta\le1$, define
$
    G_\theta(y):=y+\theta e(y).
$ Then $G_0(y)=y$ and $G_1(y)=G(y)$.
The equation $G_\theta(y)=z$ is equivalent to
$y=z-\theta e(y)$, whose right-hand side is a contraction.
Thus $G_\theta$ is a global $C^1$ diffeomorphism, with
\begin{equation}\label{eq:inverse bound}
    \lVert(D_yG_\theta(y))^{-1}\rVert_{\mathrm{op}}
    =
    \lVert(I+\theta D_ye(y))^{-1}\rVert_{\mathrm{op}}
    \le\frac43.
\end{equation}
For each $y$, the determinant $\det D_yG_\theta(y)$
is continuous in $\theta$, never vanishes, and equals
$1$ at $\theta=0$. Hence it is positive throughout
the interpolation.

Let $p_s$ denote the density of $P_s$. Changing variables
under $G=G_1$ gives
\begin{equation*}
    2\operatorname{TV}(P_s,G_\#P_s)
    =
    \int_{\mathbb R^d}
    \left|
        p_s(G(y))\det D_yG(y)-p_s(y)
    \right|\,dy.
\end{equation*}
To compare these two terms, set
$
    H_\theta(y)
    :=
    p_s(G_\theta(y))\det D_yG_\theta(y).
$
Differentiating with respect to $\theta$ and using
Jacobi's formula yields
\begin{equation*}
\begin{aligned}
    \partial_\theta H_\theta(y)
    ={}&
    \det D_yG_\theta(y)\,
    \nabla p_s(G_\theta(y))\cdot e(y)
    \\
    &+
    p_s(G_\theta(y))\det D_yG_\theta(y)
    \operatorname{tr}\left[
        (I+\theta D_ye(y))^{-1}D_ye(y)
    \right].
\end{aligned}
\end{equation*}
Since
\begin{equation*}
    H_1(y)-H_0(y)
    =
    \int_0^1\partial_\theta H_\theta(y)\,d\theta,
\end{equation*}
we apply the triangle inequality and then change variables
$z=G_\theta(y)$ for each $\theta$. This gives
\begin{equation*}
\begin{aligned}
    2\operatorname{TV}(P_s,G_\#P_s)
    \leq{}&
    \int_0^1\int_{\mathbb R^d}
    |\nabla p_s(z)|
    \,|e(G_\theta^{-1}(z))|
    \,dz\,d\theta
    \\
    &+
    \int_0^1\int_{\mathbb R^d}
    p_s(z)
    \left|
        \operatorname{tr}\left[
            (I+\theta D_ye(y))^{-1}D_ye(y)
        \right]
    \right|_{y=G_\theta^{-1}(z)}
    \,dz\,d\theta.
\end{aligned}
\end{equation*}
We estimate the density-gradient term and the Jacobian
trace term separately.

For the density-gradient term, set
\begin{equation*}
    A:=C_ph^{p+1}e^{-s}b_*^{p+1},
    \qquad
    B:=C_ph^{p+1}e^{-s/2}\Theta_*^{2p+1},
\end{equation*}
so that \eqref{eq:endpoint-correction-value} gives
$
    |e(y)|\le A|y|+B.
$
$c_p\leq(4C_p)^{-1/(p+1)}$ also ensures
\begin{equation*}
    A
    \le C_p(h\Theta_*^2)^{p+1}
    \le C_pc_p^{p+1}
    \le\frac14.
\end{equation*}
For $z=G_\theta(y)$, we have
\begin{equation*}
    |y|
    \le |z|+\theta|e(y)|
    \le |z|+A|y|+B.
\end{equation*}
Hence
\begin{equation*}
    |G_\theta^{-1}(z)|
    \le\frac{|z|+B}{1-A},
\end{equation*}
and substituting this into the bound for $e$ gives
\begin{equation*}
    |e(G_\theta^{-1}(z))|
    \le
    \frac{A|z|+B}{1-A}
    \le 2(A|z|+B),
    \qquad 0\le\theta\le1.
\end{equation*}
Using $\nabla p_s=p_s\nabla\log p_s$ and
Cauchy--Schwarz, we obtain
\begin{equation*}
\begin{aligned}
    \int_{\mathbb R^d}
    |\nabla p_s(z)|
    \,|e(G_\theta^{-1}(z))|\,dz
&\le
    2A\,\mathbb E\left[
        |\nabla\log p_s(X_s)|\,|X_s|
    \right]
    +
    2B\,\mathbb E|\nabla\log p_s(X_s)|
    \\
    &\le
    2\lVert\nabla\log p_s(X_s)\rVert_{L^2}
    \left[
        A\lVert{}X_s\rVert_{L^2}+B
    \right].
\end{aligned}
\end{equation*}
At time $s$, the OU representation is
$
    X_s=e^{-s/2}X_0+\sqrt{1-e^{-s}}\,Z,
$
where $Z\sim N(0,I_d)$ is independent of $X_0$.
Tweedie's formula gives
\begin{equation*}
    \nabla\log p_s(X_s)
    =
    -\sqrt{b_*}\,\mathbb E[Z\mid X_s].
\end{equation*}
Conditional Jensen's inequality yields the following score bound,
while the same representation controls the second moment:
\begin{equation*}
    \lVert\nabla\log p_s(X_s)\rVert_{L^2}
    \le\sqrt{d\,b_*},
    \qquad
    \lVert{}X_s\rVert_{L^2}\le R+\sqrt d.
\end{equation*}
Substituting these estimates and the definitions of $A$
and $B$ gives, uniformly in $\theta$,
\begin{equation}\label{eq:density-gradient term}
\begin{aligned}
    &\int_{\mathbb R^d}
    |\nabla p_s(z)|
    \,|e(G_\theta^{-1}(z))|\,dz
    \\
    &\qquad\le
    C_ph^{p+1}\left[
        \sqrt d\,(R+\sqrt d)
        e^{-s}b_*^{p+3/2}
        +
        \sqrt{d\,b_*}\,
        e^{-s/2}\Theta_*^{2p+1}
    \right].
\end{aligned}
\end{equation}

For the Jacobian trace term, we use
$|\operatorname{tr}M|\le d\lVert{}M\rVert_{\mathrm{op}}$,
the inverse bound~\eqref{eq:inverse bound} for $D_yG_\theta$, and
\eqref{eq:endpoint-correction-jacobian}:
\begin{equation*}
\begin{aligned}
    \left|
        \operatorname{tr}\left[
            (I+\theta D_ye(y))^{-1}D_ye(y)
        \right]
    \right|
\le
    \frac43d\,\lVert{}D_ye(y)\rVert_{\mathrm{op}}
    \le
    C_pd h^{p+1}e^{-s}\Theta_*^{2p+2}.
\end{aligned}
\end{equation*}
This estimate is uniform in $y$ and $\theta$.
Since $p_s$ integrates to one, it also bounds the
integrated trace contribution,
\begin{equation}\label{eq:Jacobian trace term}
    \int_{\mathbb R^d}
    p_s(z)
    \left|
        \operatorname{tr}\left[
            (I+\theta D_ye(y))^{-1}D_ye(y)
        \right]
    \right|_{y=G_\theta^{-1}(z)}
    \,dz\leq C_p d h^{p+1}e^{-s}\Theta_*^{2p+2}.
\end{equation}

Integrating~\eqref{eq:density-gradient term} and~\eqref{eq:Jacobian trace term} over $\theta\in[0,1]$ yields
\begin{equation*}
\begin{aligned}
    \operatorname{TV}(P_s,G_\#P_s)
    \le C_ph^{p+1}\Bigl[
        &\sqrt d\,(R+\sqrt d)
        e^{-s}b_*^{p+3/2}
        \\
        &+\sqrt{d\,b_*}\,
        e^{-s/2}\Theta_*^{2p+1}
        \\
        &+d\,e^{-s}\Theta_*^{2p+2}
    \Bigr]
    =
    C_ph^{p+1}\mathfrak D_p(s).
\end{aligned}
\end{equation*}
The density-gradient term produces the first two terms
in $\mathfrak D_p$, through the score and second-moment
bounds. The Jacobian trace produces the third. Since $(\Psi_{u,h})_\#P_t=G_\#P_s$, this proves
\eqref{eq:one-step-weighted-TV}.
\end{proof}

\subsubsection{Exact-score TV convergence and comparison with Runge--Kutta analyses}

We now apply Lemma~\ref{lem:endpoint-transport-bound} along the
reverse-time grid. Since total variation does not increase under a
common measurable pushforward, the global error is bounded by the
initialization error plus the sum of the one-step errors.
The time weights in \eqref{eq:weighted-TV-density} are integrable
on $[\delta,\infty)$, so this summation introduces no additional
factor proportional to the integration time.

\begin{theorem}[Exact-score TV convergence from Gaussian initialization]
\label{thm:weighted-exact-tv-bounds}
Suppose that $P_0$ is supported in $B_R(0)$ and satisfies
Assumption~\ref{ass:uniform-edge-doubling}.
Fix an integer $p\geq1$, an early-stopping time $0<\delta\leq1$,
and a terminal time $T\geq1$ with $T>\delta$.
For an integer $N\geq1$, set
\begin{equation*}
    h=\frac{T-\delta}{N},
    \qquad
    u_n=nh,
    \qquad
    t_n=T-u_n,
    \qquad n=0,\ldots,N.
\end{equation*}
Assume that
\begin{equation}
    h\Theta_R(\delta)^2\leq c_p
    \label{eq:tv-global-step-condition}
\end{equation}
for a sufficiently small constant $c_p>0$.

Consider either
\begin{enumerate}
\item a fixed explicit Runge--Kutta method of classical order $p$
      with stage nodes in $[0,1]$;
\item the order-$p$ truncated Taylor scheme of
      Section~\ref{sec:global-convergence-taylor}.
\end{enumerate}
Let $\Psi_{u_n,h}$ denote the corresponding numerical step map
on $[u_n,u_{n+1}]$, and define the numerical laws by
\begin{equation*}
    Q_0=\gamma_d:=N(0,I_d),
    \qquad
    Q_{n+1}=(\Psi_{u_n,h})_\#Q_n,
    \qquad n=0,\ldots,N-1.
\end{equation*}
Then
\begin{equation}
\begin{aligned}
    \operatorname{TV}(P_\delta,Q_N)
    \leq{}&
    C\left(e^{-T/2}R+\sqrt d\,e^{-T}\right)
    \\
    &+
    C_p d(1+R)^{2p+2}\delta^{-(2p+1)}h^p.
\end{aligned}
\label{eq:weighted-global-TV-gaussian-init}
\end{equation}
The constant $C$ is universal.
The constants $c_p$ and $C_p$ depend only on $p$, the doubling
constant in Assumption~\ref{ass:uniform-edge-doubling}, and the
fixed Runge--Kutta tableau when applicable.
\end{theorem}

The first term in \eqref{eq:weighted-global-TV-gaussian-init}
accounts for replacing $P_T$ by the standard Gaussian law and decays
exponentially in $T$. The second is the accumulated discretization
error, whose dependence on dimension and step size is $d\,h^p$ for
fixed target parameters and early-stopping time.

\begin{proof}
Set
$
    \Psi_n:=\Psi_{u_n,h},
    \
    e_n:=\operatorname{TV}(Q_n,P_{t_n}).
$
The exact reverse flow satisfies
$
    (\phi_{u_n,h})_\#P_{t_n}=P_{t_{n+1}}.
$
Since $t_{n+1}\geq\delta$ and $\Theta_R$ is decreasing,
\eqref{eq:tv-global-step-condition} gives
\begin{equation*}
    h\Theta_R(t_{n+1})^2
    \leq h\Theta_R(\delta)^2
    \leq c_p.
\end{equation*}
Thus Lemmas~\ref{lem:sec5-weighted-C1-local-defects}
and~\ref{lem:endpoint-transport-bound} apply to every step,
with endpoint forward time $s=t_{n+1}$.

By the triangle inequality and the fact that a common measurable
pushforward does not increase total variation,
\begin{equation*}
\begin{aligned}
    e_{n+1}
    &\leq
    \operatorname{TV}
    \left(
        (\Psi_n)_\#Q_n,
        (\Psi_n)_\#P_{t_n}
    \right)
    +
    \operatorname{TV}
    \left(
        (\Psi_n)_\#P_{t_n},
        P_{t_{n+1}}
    \right)
    \\
    &\leq
    e_n+C_p h^{p+1}\mathfrak D_p(t_{n+1}),
\end{aligned}
\end{equation*}
where the second step is by Lemma~\ref{lem:endpoint-transport-bound} and $\mathfrak D_p$ is defined in~\eqref{eq:weighted-TV-density}.

Iterating this inequality yields
\begin{equation*}
    e_N
    \leq
    e_0+C_p h^{p+1}
    \sum_{n=0}^{N-1}\mathfrak D_p(t_{n+1}).
\end{equation*}

We next compare the time sum with an integral.
The time-dependent factors in \eqref{eq:weighted-TV-density}
have the form
\begin{equation*}
    f_{a,q}(t):=e^{-at}b(t)^q,
    \qquad
    a\in\left\{\frac12,1\right\},
    \qquad
    1<q\leq2p+2.
\end{equation*}
For $t\geq\delta$,
\begin{equation*}
    \left|
        \frac{d}{dt}\log f_{a,q}(t)
    \right|
    =
    a+q e^{-t}b(t)
    \leq a+q b(\delta).
\end{equation*}
Consequently, for $t\in[t_{n+1},t_n]$,
$
    f_{a,q}(t_{n+1})
    \leq
    \exp\bigl(h[a+q b(\delta)]\bigr)f_{a,q}(t).
$
Since $1\leq b(\delta)\leq\Theta_R(\delta)^2$,
the step-size condition implies
$
    h[a+q b(\delta)]
    \leq(a+q)c_p.
$
The exponential factor is therefore bounded by $C_p$, and
integration over $[t_{n+1},t_n]$ gives
\begin{equation*}
    h f_{a,q}(t_{n+1})
    \leq
    C_p\int_{t_{n+1}}^{t_n}f_{a,q}(t)\,dt.
\end{equation*}
Applying this estimate to the three terms in
$\mathfrak D_p$ and summing over $n$, we obtain
\begin{equation*}
    h\sum_{n=0}^{N-1}\mathfrak D_p(t_{n+1})
    \leq
    C_p\int_\delta^T\mathfrak D_p(t)\,dt.
\end{equation*}
For the first and third terms of $\mathfrak D_p$, the identity
$b'(t)=-e^{-t}b(t)^2$ gives, for every $q>1$,
\begin{equation*}
    \int_\delta^T e^{-t}b(t)^q\,dt
    =
    \frac{b(\delta)^{q-1}-b(T)^{q-1}}{q-1}
    \leq
    \frac{b(\delta)^{q-1}}{q-1}.
\end{equation*}
For the middle term, write
\begin{equation*}
    \sqrt{d\,b(t)}e^{-t/2}\Theta_R(t)^{2p+1}
    =
    \sqrt d(1+R)^{2p+1}
    e^{-t/2}b(t)^{2p+3/2}.
\end{equation*}
Using $b(t)\asymp t^{-1}$ on $(0,1]$ and the boundedness of
$b$ on $[1,\infty)$, we find
\begin{equation*}
\begin{aligned}
    \int_\delta^T e^{-t/2}b(t)^{2p+3/2}\,dt
    &\leq
    C_p\left(
        \int_\delta^1 t^{-(2p+3/2)}\,dt+1
    \right)
    \\
    &\leq C_p b(\delta)^{2p+1/2}.
\end{aligned}
\end{equation*}
Substituting these bounds into \eqref{eq:weighted-TV-density}
yields
\begin{equation*}
\begin{aligned}
    \int_\delta^T\mathfrak D_p(t)\,dt
    &\leq
    C_p\sqrt d(R+\sqrt d)b(\delta)^{p+1/2}
    \\
    &\quad+
    C_p\sqrt d(1+R)^{2p+1}b(\delta)^{2p+1/2}
    \\
    &\quad+
    C_p d(1+R)^{2p+2}b(\delta)^{2p+1}
    \\
    &\leq
    C_p d(1+R)^{2p+2}b(\delta)^{2p+1},
\end{aligned}
\end{equation*}
where the last inequality uses $d\geq1$, $1+R\geq1$,
and $b(\delta)\geq1$.
It follows that
\begin{equation*}
    e_N
    \leq
    e_0+
    C_p d(1+R)^{2p+2}b(\delta)^{2p+1}h^p.
\end{equation*}

It remains to bound the initialization error
$e_0=\operatorname{TV}(\gamma_d,P_T)$.
The forward OU representation gives
\begin{equation*}
    X_T\mid X_0
    \sim
    N\left(e^{-T/2}X_0,(1-e^{-T})I_d\right).
\end{equation*}
By convexity of relative entropy and the Gaussian
relative-entropy formula,
\begin{equation*}
\begin{aligned}
    \operatorname{KL}(P_T\|\gamma_d)
    &\leq
    \mathbb E\left[
        \operatorname{KL}\left(
            N(e^{-T/2}X_0,(1-e^{-T})I_d)
            \,\|\,\gamma_d
        \right)
    \right]
    \\
    &=
    \frac12 e^{-T}\mathbb E|X_0|^2
    +
    \frac d2
    \left[-e^{-T}-\log(1-e^{-T})\right]
    \\
    &\leq
    C\left(e^{-T}R^2+d e^{-2T}\right).
\end{aligned}
\end{equation*}
Here we used $|X_0|\leq R$ almost surely and
\begin{equation*}
    -e^{-T}-\log(1-e^{-T})
    \leq C e^{-2T},
    \qquad T\geq1.
\end{equation*}
Pinsker's inequality therefore yields
$
    e_0
    \leq
    C\left(e^{-T/2}R+\sqrt d\,e^{-T}\right).
$
Combining the two error bounds and using
$b(\delta)\leq C\delta^{-1}$ for $0<\delta\leq1$, we conclude that
\begin{equation*}
\begin{aligned}
    \operatorname{TV}(P_\delta,Q_N)
    \leq
    C\left(e^{-T/2}R+\sqrt d\,e^{-T}\right)
+
    C_p d(1+R)^{2p+2}\delta^{-(2p+1)}h^p,
\end{aligned}
\end{equation*}
which proves \eqref{eq:weighted-global-TV-gaussian-init}.
\end{proof}

We now choose the terminal time and step size to reach a prescribed
accuracy in total variation.

\begin{corollary}[Dimension dependence of the exact-score TV complexity]
\label{cor:weighted-TV-complexity}
Under the assumptions of Theorem~\ref{thm:weighted-exact-tv-bounds},
let $0<\varepsilon\leq1$. Set
\begin{equation*}
    T=\log\frac{C_0(R^2+d)}{\varepsilon^2}
\end{equation*}
and
\begin{equation*}
    h_*=
    \kappa_p
    \min\left\{
        \Theta_R(\delta)^{-2},
        \left(
            \frac{\varepsilon\,\delta^{2p+1}}
                 {d(1+R)^{2p+2}}
        \right)^{1/p}
    \right\},
\end{equation*}
where $C_0\geq e^2$ is a sufficiently large universal constant
and $\kappa_p>0$ is sufficiently small, depending only on $p$,
the doubling constant, and the fixed Runge--Kutta tableau
when applicable.
Choose
\begin{equation*}
    N=\left\lceil\frac{T-\delta}{h_*}\right\rceil,
    \qquad
    h=\frac{T-\delta}{N}.
\end{equation*}
Then
\begin{equation*}
    \operatorname{TV}(P_\delta,Q_N)\leq\varepsilon.
\end{equation*}
For fixed $p$, $R$, $\delta$, the doubling constant, and the
Runge--Kutta tableau when applicable, the resulting step
complexity in the high-accuracy regime
\begin{equation}\label{eq:tv-high-accuracy-regime}
    \varepsilon
    \le
    d\,(1+R)^{2p+2}\,\delta^{-(2p+1)}\,\Theta_R(\delta)^{-2p},
\end{equation}
in which the accuracy constraint on $h_*$ is the binding one, is
\begin{equation}
    N
    =
    \widetilde O_{p,R,\delta}
    \left(
        d^{1/p}\varepsilon^{-1/p}
    \right).
    \label{eq:weighted-TV-complexity}
\end{equation}
\end{corollary}

\begin{proof}
Since $R^2+d\geq1$ and $\varepsilon\leq1$, the choice
$C_0\geq e^2$ ensures $T\geq2>\delta$.
Moreover,
\begin{equation*}
    e^{-T/2}R
    \leq\frac{\varepsilon}{\sqrt{C_0}},
    \qquad
    \sqrt d\,e^{-T}
    \leq\frac{\varepsilon}{C_0}.
\end{equation*}
Taking $C_0$ sufficiently large therefore makes the
initialization term in
\eqref{eq:weighted-global-TV-gaussian-init}
at most $\varepsilon/2$.

The definition of $N$ ensures $h\leq h_*$, so
\begin{equation*}
    h\Theta_R(\delta)^2\leq\kappa_p,\qquad
    C_p d(1+R)^{2p+2}\delta^{-(2p+1)}h^p
    \leq C_p\kappa_p^p\varepsilon.
\end{equation*}
Choosing $\kappa_p\leq c_p$ and
$C_p\kappa_p^p\leq1/2$ verifies
\eqref{eq:tv-global-step-condition} and makes the
discretization term at most $\varepsilon/2$.
Theorem~\ref{thm:weighted-exact-tv-bounds} then gives the
claimed accuracy.

Finally, the rounding in the definition of $N$ gives
\begin{equation*}
    N
    \leq
    1+\frac{T-\delta}{\kappa_p}
    \max\left\{
        \Theta_R(\delta)^2,
        \left(
            \frac{d(1+R)^{2p+2}}
                 {\varepsilon\,\delta^{2p+1}}
        \right)^{1/p}
    \right\}.
\end{equation*}
Under \eqref{eq:tv-high-accuracy-regime} the second term in the
maximum dominates.
Substituting the chosen value of $T$ and absorbing its
logarithmic dependence into $\widetilde O$ proves
\eqref{eq:weighted-TV-complexity}.
\end{proof}

\begin{remark}
\label{rem:dimension-sources}
The Runge--Kutta bounds of \cite{huang2025convergence,huang2025fast}
contain a leading exact-score discretization term of order
$d^{p+1}h^p$, corresponding to
$\widetilde O(d^{1+1/p}\varepsilon^{-1/p})$ steps for fixed order
and remaining parameters, whereas
Corollary~\ref{cor:weighted-TV-complexity} gives $d\,h^p$ and
$\widetilde O(d^{1/p}\varepsilon^{-1/p})$. The gap of $d^p$ has two
sources, both located in how the regularity of the score is
controlled.

First, in \cite{huang2025convergence} the score and its derivatives
are estimated through Tweedie's formula, which carries an explicit
factor $|x|$: $|D_x\log p_t(x)|\le b(t)(|x|+R)$ under compact support,
and the same $|x|$ persists in every derivative along a step.
Repeated differentiation of the stage map by the Fa\`a di Bruno
formula multiplies up to $p$ such factors, so the $p$-th derivative
along a step is bounded by an expression of the form
$C\bigl((\sqrt d+|x|)\sqrt d\,\kappa\bigr)^p$, and averaging $|x|^p$
over $P_t$ costs a further $d^{p/2}$. Second, the regularity
hypothesis there is componentwise: it bounds the individual entries of
the derivative tensors of the score, and each contraction of such a
tensor against a vector is a passage from $\ell^\infty$ to $\ell^2$
that costs a factor $\sqrt d$; an order-$p$ scheme performs $p$ such
contractions. Together these give $d^{p}$ in the local defect, and the
conversion to total variation adds one more power.

In the present analysis both sources are absent. The normal-direction
fluctuation bound of Theorem~\ref{thm:sec-3-normal-fluctuation-two-scale}
is uniform in $x$: no $|x|$ term appears in
$\mathbb E_{\mu_{x,t}}|\langle Z,x\rangle|^q$. Consequently the
decomposition $L^jx=A_j(t)x+\widetilde A_j(t)\bar y_{x,t}+R_j$ of
Proposition~\ref{prop:quantitative-normal-form-decomposition} confines
the dependence on $x$ to a single affine term with coefficient of
order $e^{-t}b(t)^j$, and the remainder is bounded uniformly in $x$.
The stage induction of
Lemma~\ref{lem:weighted-rk-stage-regularity} preserves this
affine-plus-remainder structure at every stage, so nesting does not
raise the degree in $|x|$ above one. Moreover, all estimates of
Section~\ref{sec:sec3-centered-posterior-moment-normal-form} are
operator-norm statements
(Lemma~\ref{lem:sec3-weighted-directional-bound}), so no
$\ell^\infty$-to-$\ell^2$ conversion occurs. As a result the local
defect bounds of Lemma~\ref{lem:sec5-weighted-C1-local-defects} carry
constants independent of $d$, and the dimension enters the proof at
exactly three points of Lemma~\ref{lem:endpoint-transport-bound}:
the second moment $\lVert X_s\rVert_{L^2}\le R+\sqrt d$, the score
bound $\lVert D_x\log p_s(X_s)\rVert_{L^2}\le\sqrt{d\,b(s)}$, and the
trace inequality $|\operatorname{tr}M|\le d\lVert M\rVert_{\mathrm{op}}$,
each contributing at most one power of $d$. We therefore list the
complexity bounds in Table~\ref{tab:intro-dimension-comparison}.
\end{remark}

\subsection{Numerical illustration}
\label{subsec:numerical-illustration}

We complement the theoretical results with controlled numerical experiments
designed to examine the dependence of the discretization error on both the
time step and the ambient dimension.  The purpose of these experiments is
not to assess the end-to-end performance of a learned diffusion model, but
rather to isolate the numerical error of the reverse probability-flow ODE
in a setting where the score and the reference flow can be computed to high
accuracy.

We consider the dimension-independent two-point target
\begin{equation*}
    P_0^{(d)}
    =
    \frac12\delta_{- e_1}
    +
    \frac12\delta_{ e_1},
\end{equation*}
embedded in $\mathbb R^d$.  This family has uniformly bounded support and
satisfies the uniform doubling assumption with a constant independent of
$d$; in fact, one may take $D_{\mathrm{UD}}\leq 2$.

For the forward OU process~\eqref{eq:standardOU}, the marginal law is the Gaussian mixture
\begin{equation*}
    P_t^{(d)}
    =
    \frac12
    N(-\alpha(t)e_1,\beta(t)^2 I_d)
    +
    \frac12
    N(\alpha(t)e_1,\beta(t)^2 I_d),
\end{equation*}
and its score admits the closed-form expression
\begin{equation*}
    D_x \log p_t(x)
    =
    -\frac{x}{\beta(t)^2}
    +
    \frac{\alpha(t) }{\beta(t)^2}
    \tanh\left(
        \frac{\alpha(t)  x_1}{\beta(t)^2}
    \right)e_1.
\end{equation*}
Consequently, the reverse probability-flow velocity used in our
parametrization is
\begin{equation*}
\begin{aligned}
    \frac12\bigl(x+D_x\log p_t(x)\bigr) 
    =
    -\frac{e^{-t}}{2(1-e^{-t})}x
    +
    \frac{ e^{-t/2}}{2(1-e^{-t})}
    \tanh\left(
        \frac{ e^{-t/2}x_1}{1-e^{-t}}
    \right)e_1 .
\end{aligned}
\end{equation*}
Thus no score approximation is involved in the experiment.

For the Runge--Kutta experiments, we use explicit Euler,
the explicit midpoint method, Kutta's third-order method,
and the classical fourth-order Runge--Kutta method for
$p=1,2,3,4$, respectively.
All stage nodes lie in $[0,1]$.
Both experiments use MATLAB with the random seed fixed to $1$.

To isolate the discretization error, both the exact and numerical reverse
flows are initialized from the same samples
\begin{equation*}
    X_T^{(i)}\sim P_T^{(d)}.
\end{equation*}
Let $Y^{(i)}$ denote the endpoint of the exact reverse flow
at time $\delta$. Write $X_{T,k}^{(i)}$ and $Y_k^{(i)}$ for the $k$-th
coordinates of $X_T^{(i)}$ and $Y^{(i)}$, respectively.
Writing $\mathcal F_t$ for the distribution function of the
first coordinate under $P_t^{(d)}$, the exact endpoint satisfies
\begin{equation*}
\begin{aligned}
    Y_1^{(i)}
    &=
    \mathcal F_\delta^{-1}
    \bigl(\mathcal F_T(X_{T,1}^{(i)})\bigr),
    \\
    Y_k^{(i)}
    &=
    \frac{\beta(\delta)}{\beta(T)}X_{T,k}^{(i)},
    \qquad k=2,\ldots,d.
\end{aligned}
\end{equation*}
In the Runge--Kutta experiments, the inverse distribution
function is evaluated by $90$ bisection iterations on
$[-20,20]$, with its probability argument clipped to
$[10^{-15},1-10^{-15}]$.
Let $Y_{p,h}^{(i)}$ denote the endpoint produced by a
$p$-th order numerical method with step size $h$.
We measure the error by the synchronous-coupling quantity
\begin{equation}\label{eq:numerical-coupled-error}
    \widehat E_{p,h,d}
    :=
    \left(
        \frac1M
        \sum_{i=1}^M
        \left|
            Y^{(i)}-Y_{p,h}^{(i)}
        \right|^2
    \right)^{1/2}.
\end{equation}
Since the exact and numerical flows
use the same initialization, \eqref{eq:numerical-coupled-error} measures
only the time-discretization error and excludes the separate initialization
error arising from replacing $P_T$ by the standard Gaussian.
\paragraph{Time-step convergence.}
We first study the convergence with respect to the step size. We set
$T=3$, $\delta=0.5$, $d=50$, and $M=10^4$, and use
\begin{equation*}
    N\in\{20,40,80,160,320,640,1280\},
    \qquad
    h=\frac{T-\delta}{N}.
\end{equation*}
Figure~\ref{fig:numerical-two-point}(a)--(b) shows the coupled error on a
log--log scale for the truncated Taylor schemes of orders $p=1,2,3,4$
and, for comparison, explicit Runge--Kutta schemes of the corresponding
orders. The fitted slopes reported in Figure~\ref{fig:taylor-stepsize}--\ref{fig:rk-stepsize}
are close to the corresponding orders $p=1,2,3,4$, in agreement with
the expected $p$-th order behavior
$\widehat E_{p,h,d}\asymp h^p$.
In particular, the Taylor results numerically confirm the high-order
time-discretization rate established by the preceding analysis.

\paragraph{Ambient-dimension scaling.}
Keeping $T=3$ and $\delta=0.5$, we next fix
$h=0.015625$ and $M=10^4$, and vary the ambient dimension over
\begin{equation*}
    d\in\{5,10,20,40,80,160,320,640,1280,2560,5120\}.
\end{equation*}
Figure~\ref{fig:numerical-two-point}(c)--(d) displays the resulting error
as a function of $d$. The fitted slopes reported in Figure~\ref{fig:numerical-two-point}(c)--(d)
are all close to $1/2$, consistent with the predicted
$\sqrt d$ dependence of the error
\(
    \widehat E_{p,h,d}\asymp d^{1/2}
\)
in the high-dimensional regime.

Taken together, the experiments exhibit the joint behavior
\begin{equation*}
    \widehat E_{p,h,d}
    \asymp
    C_p\,\sqrt d\,h^p
\end{equation*}
for this exactly tractable benchmark, in agreement with the
time-discretization and ambient-dimension dependence predicted by the
theory.
\begin{figure}[htbp]
    \centering

    \begin{subfigure}[t]{0.48\linewidth}
        \centering
        \includegraphics[width=0.94\linewidth]
        {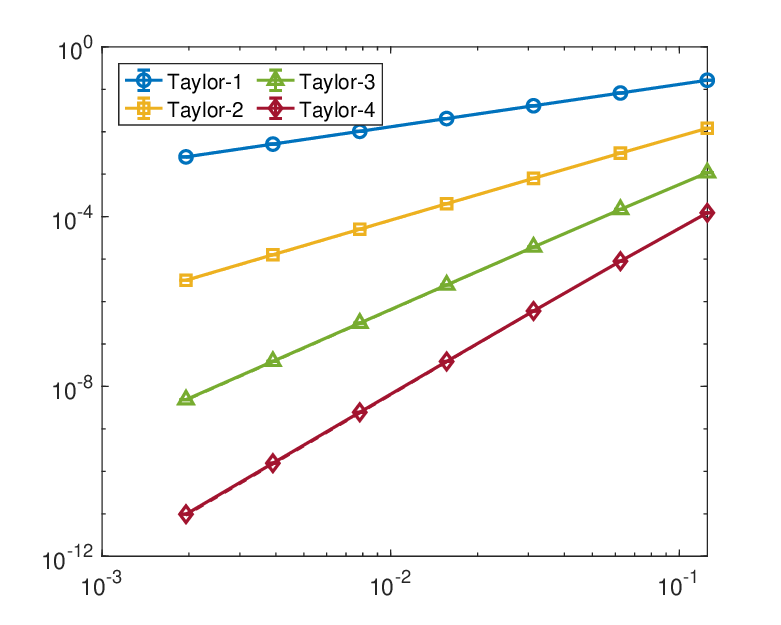}
        \caption{
    Taylor schemes: step-size convergence.
    Fitted slopes for $p=1,2,3,4$, respectively:
    $1.000$, $1.994$, $2.987$, $3.977$.
}
        \label{fig:taylor-stepsize}
    \end{subfigure}
    \hfill
    \begin{subfigure}[t]{0.48\linewidth}
        \centering
        \includegraphics[width=0.94\linewidth]
        {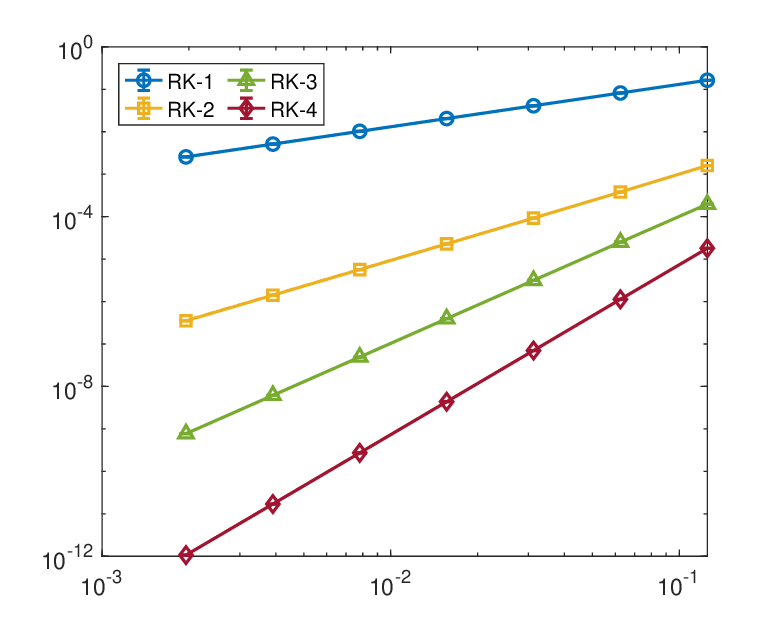}
       \caption{
    Runge--Kutta schemes: step-size convergence.
    Fitted slopes for orders $p=1,2,3,4$, respectively:
    $1.000$, $2.009$, $2.999$, $4.002$.
}
        \label{fig:rk-stepsize}
    \end{subfigure}

    \vspace{-0.2em}

    \begin{subfigure}[t]{0.48\linewidth}
        \centering
        \includegraphics[width=\linewidth]
        {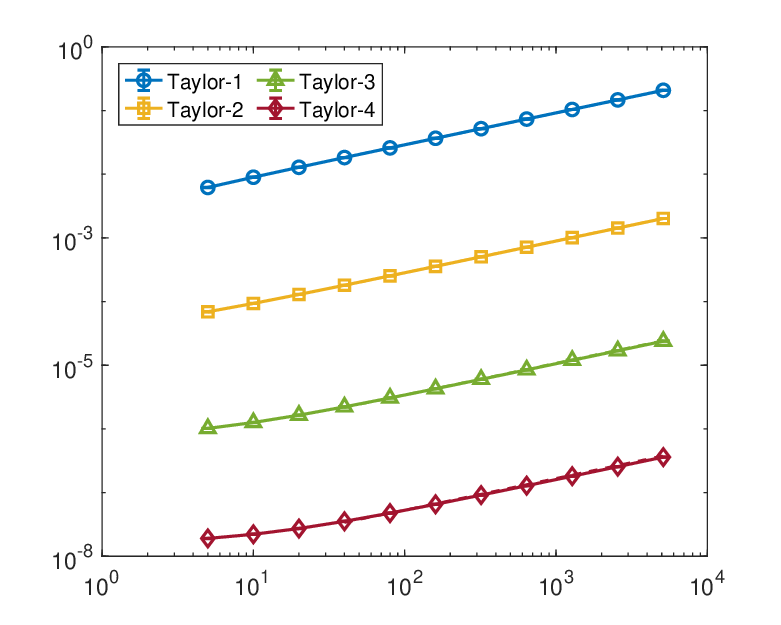}
        \caption{
    Taylor schemes: ambient-dimension scaling.
    Fitted slopes for $p=1,2,3,4$, respectively:
    $0.501$, $0.498$, $0.491$, $0.482$.
}
       
    \end{subfigure}
    \hfill
    \begin{subfigure}[t]{0.48\linewidth}
        \centering
        \includegraphics[width=\linewidth]
        {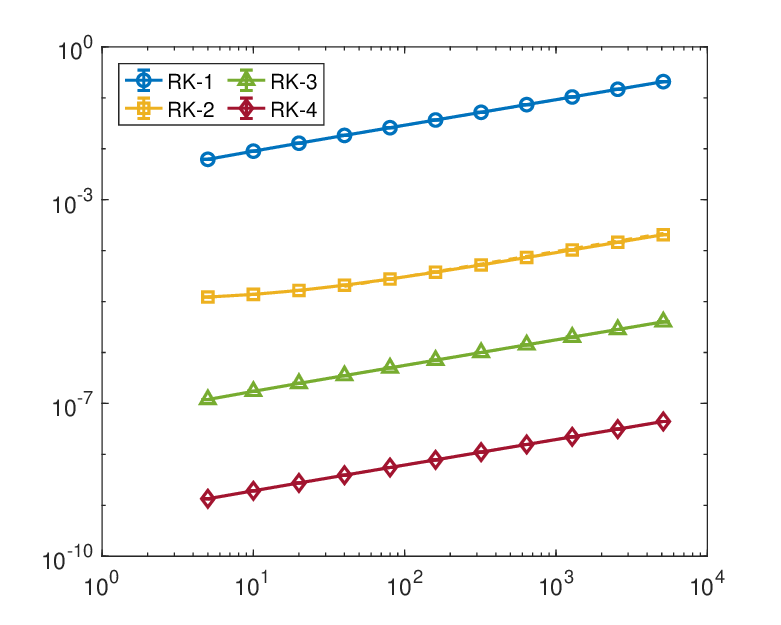}
       \caption{
    Runge--Kutta schemes: ambient-dimension scaling.
    Fitted slopes for orders $p=1,2,3,4$, respectively:
    $0.501$, $0.475$, $0.501$, $0.501$.
}
       
    \end{subfigure}

    \caption{
        Numerical illustration for the two-point target.
        Top row: coupled $L^2$ error as a function of the step size at
        fixed dimension $d=50$.
        Bottom row: coupled $L^2$ error as a function of the ambient
        dimension at fixed step size $h=0.015625$.
    }
    \label{fig:numerical-two-point}
\end{figure}

\section{Conclusion}

We derived arbitrary-order regularity estimates for the
Ornstein--Uhlenbeck reverse probability-flow ODE from compact
support and a dimension-uniform doubling condition on the
supporting-cap masses of the target distribution.
The key estimate controls centered posterior fluctuations
in the normal direction without growth in $|x|$.
We then used a centered posterior-moment representation closed
under material differentiation and a weighted calculus to
obtain the derivative bounds required by Taylor and
Runge--Kutta methods.
The constants in these bounds are independent of $d$ for
fixed support radius and doubling constant. 

For the order-$p$ truncated Taylor scheme, these estimates
give a $W_2$ discretization error of order $\sqrt d\,h^p$.
In total variation, both the order-$p$ Taylor scheme and
fixed explicit Runge--Kutta methods of order $p$ have
discretization error of order $d\,h^p$.
For fixed target and early-stopping parameters, the
corresponding step complexities are
$\widetilde O(d^{1/(2p)}\varepsilon^{-1/p})$ in $W_2$ and
$\widetilde O(d^{1/p}\varepsilon^{-1/p})$ in total variation.
In both metrics, the polynomial dimension exponent tends
to zero as the numerical order increases.

The analysis uses the exact score and retains the Gaussian
initialization error explicitly.
Future work could incorporate score-estimation error and
develop learnable approximations of the material derivatives
required by Taylor schemes.
Other directions include extending the analysis beyond the
doubling setting using weaker fluctuation estimates, and
developing analogous bounds for higher-order stochastic
reverse-time schemes. A further question is whether these estimates can support a
Bakry--{\'E}mery analysis of the associated time-inhomogeneous
reverse diffusion. Its instantaneous curvature matrix is
$\frac12I_d-2D_x(Lx)$, so our first spatial derivative estimate
provides a dimension-independent, time-dependent curvature
lower bound. The higher-order posterior-moment calculus may
help estimate the derivatives and contractions of the reverse
drift that arise in further $\Gamma$-calculations.
Deriving corresponding gradient estimates with controlled
dimension and time dependence remains a direction for
future work.
\appendix
\section{Examples and geometric interpretation of the doubling condition}
\label{app:edge-doubling}

This appendix gives sufficient conditions and examples for
Assumption~\ref{ass:uniform-edge-doubling}.
Recall that $F_\theta(u)$ denotes the $P_0$-mass within depth $u$
of the supporting hyperplane with outward normal $\theta$.
The assumption requires
\begin{equation*}
    F_\theta(2u)
    \leq
    D_{\mathrm{UD}}F_\theta(u),
    \qquad
    \theta\in\mathbb S^{d-1},
    \quad u>0.
\end{equation*}
For a family of targets in different ambient dimensions,
the same constant $D_{\mathrm{UD}}$ must work for every $d$.

We examine convex,
manifold, and atomic targets, keeping track of the parameters
that control the doubling constant.
We then discuss its dependence on the ambient dimension
and explain how local doubling can be extended to all scales.

\subsection{Targets of fixed intrinsic dimension}

For convex targets, the doubling constant can be controlled by
the intrinsic dimension and the density ratio.
This gives examples for which the constant remains bounded
as the ambient dimension increases.

\begin{proposition}[Convex targets of fixed intrinsic dimension]
\label{prop:intrinsic-convex-doubling}
Let $K\subset\mathbb R^d$ be a compact convex set whose affine
hull has dimension $k$.
Suppose that $P_0$ has a density $f$ with respect to the intrinsic
volume measure $\operatorname{vol}_k$ on $K$, satisfying
\begin{equation*}
    0<c_-\leq f(y)\leq c_+<\infty
\end{equation*}
for $\operatorname{vol}_k$-almost every $y\in K$.
Then
\begin{equation*}
    F_\theta(2u)
    \leq
    \frac{c_+}{c_-}2^k F_\theta(u)
\end{equation*}
for every $\theta\in\mathbb S^{d-1}$ and $u>0$.
\end{proposition}

\begin{proof}
The lower density bound implies that $\operatorname{supp}P_0=K$.
Fix $\theta\in\mathbb S^{d-1}$ and choose a supporting point
$y_\theta\in K$ such that
$
    \langle\theta,y_\theta\rangle
    =
    \max_{y\in K}\langle\theta,y\rangle.
$
Write
$
    C_\theta(u)
    :=
    \{y\in K:U_\theta(y)\leq u\}
$
for the supporting cap of depth $u$.
Consider the contraction
$
    H(y)
    :=
    y_\theta+\frac12(y-y_\theta).
$
Convexity gives $H(K)\subseteq K$, while the choice of
$y_\theta$ implies
\begin{equation*}
    U_\theta(H(y))
    =
    \frac12 U_\theta(y),
    \qquad y\in K.
\end{equation*}
Hence $H(C_\theta(2u))\subseteq C_\theta(u)$.
Since $H$ multiplies intrinsic $k$-dimensional volume by
$2^{-k}$,
\begin{equation*}
    2^{-k}\operatorname{vol}_k(C_\theta(2u))
    =
    \operatorname{vol}_k(H(C_\theta(2u)))
    \leq
    \operatorname{vol}_k(C_\theta(u)).
\end{equation*}
The density bounds therefore give
\begin{equation*}
\begin{aligned}
    F_\theta(2u)
    \leq
    c_+\operatorname{vol}_k(C_\theta(2u))
    \leq
    c_+2^k\operatorname{vol}_k(C_\theta(u))
    \leq
    \frac{c_+}{c_-}2^k F_\theta(u).
\end{aligned}
\end{equation*}
\end{proof}

Consequently, a family of such targets satisfies
dimension-uniform doubling whenever the intrinsic dimensions
$k$ and the density ratios $c_+/c_-$ remain uniformly bounded.

\paragraph{Smooth manifold targets.}
Let $M=\operatorname{supp}P_0$ be a compact smooth
$k$-dimensional manifold, and suppose that $P_0$ has a density
bounded above and below by positive constants with respect to
intrinsic volume.
Fix a direction $\theta$ and suppose that every maximizer of
$y\mapsto\langle\theta,y\rangle$ admits a local parametrization
$\Psi$, centered at that maximizer, such that
\begin{equation*}
    c_1|z|^2
    \leq
    U_\theta(\Psi(z))
    \leq
    c_2|z|^2
\end{equation*}
for $z$ sufficiently close to $0$, with $0<c_1\leq c_2$.
These quadratic bounds, together with the density and local
Jacobian bounds, give a cap mass comparable to $u^{k/2}$
near each maximizer.
The maximizers are isolated, so compactness implies that finitely
many such charts contain every sufficiently shallow cap.
Consequently,
\begin{equation*}
    F_\theta(u)\asymp u^{k/2}
\end{equation*}
for sufficiently small $u$, with comparison constants and a
validity scale that may depend on $\theta$.

A sufficient uniform version of this estimate is the following:
for every direction in which the supporting functional is
nonconstant on $M$, suppose there is a scale $u_\theta>0$ such that
\begin{equation*}
    c_-\left(\frac{u}{u_\theta}\right)^{k/2}
    \leq
    F_\theta(u)
    \leq
    c_+\left(\frac{u}{u_\theta}\right)^{k/2},
    \qquad 0<u\leq u_\theta,
\end{equation*}
where $0<c_-\leq c_+<\infty$ are independent of $\theta$.
For $u\geq u_\theta$, monotonicity gives
$c_-\leq F_\theta(u)\leq1$.
Combining these bounds at $u$ and $2u$ gives
\begin{equation*}
    D_{\mathrm{UD}}
    =
    \frac{\max\{1,c_+\}}{c_-}2^{k/2}.
\end{equation*}
The direction-dependent scale cancels in the doubling ratio.
Directions in which the supporting functional is constant have
$F_\theta(u)=1$ and already satisfy doubling.
The exponent $k/2$ here, as opposed to $k$ for the convex targets of
Proposition~\ref{prop:intrinsic-convex-doubling}, reflects quadratic
rather than linear contact of the support with its supporting
hyperplane; both yield doubling constants that depend on the intrinsic
dimension $k$ but not on the ambient dimension $d$, which is all that
Assumption~\ref{ass:uniform-edge-doubling} requires.

Consequently, these targets satisfy dimension-uniform doubling
when $k$ and $\max\{1,c_+\}/c_-$ remain uniformly bounded across
the family.
The uniform cap estimates must be verified from the geometry.
Smoothness alone does not guarantee them; degenerate contact
can change the cap exponent or prevent a power-law description.

\subsection{Finite atomic targets}

For finite atomic targets, the smallest atom weight gives a
simple bound on the doubling constant.

\begin{proposition}[Atomic targets]

Suppose that
\begin{equation*}
    P_0
    =
    \sum_{i=1}^n p_i\delta_{y_i},
    \qquad
    \sum_{i=1}^n p_i=1,
    \qquad
    p_i\geq p_*>0.
\end{equation*}
Then Assumption~\ref{ass:uniform-edge-doubling} holds with
$
    D_{\mathrm{UD}}=p_*^{-1}.
$
In particular, for the equally weighted empirical measure
on $n$ points, one may take $D_{\mathrm{UD}}=n$.
\end{proposition}

\begin{proof}
Fix $\theta\in\mathbb S^{d-1}$ and choose an index $i_\theta$
such that
\begin{equation*}
    \langle\theta,y_{i_\theta}\rangle
    =
    \max_{1\leq i\leq n}\langle\theta,y_i\rangle.
\end{equation*}
Then $U_\theta(y_{i_\theta})=0$, so this atom belongs to every
supporting cap of positive depth. Consequently,
\begin{equation*}
    F_\theta(u)
    \geq
    p_{i_\theta}
    \geq
    p_*,
    \qquad u>0.
\end{equation*}
Since $F_\theta(2u)\leq1$, we obtain
\begin{equation*}
    F_\theta(2u)
    \leq
    1
    \leq
    p_*^{-1}F_\theta(u).
\end{equation*}
\end{proof}

The bound is independent of the atom locations and the ambient
dimension. Thus a family of finite atomic targets satisfies
dimension-uniform doubling whenever the atom weights have a
common positive lower bound.
For empirical targets, the estimate retains a dependence on
the number of atoms and therefore does not provide uniformity
in the sample size.

\subsection{Dependence on the ambient dimension}

Even uniform distributions on smooth convex bodies can have
doubling constants that grow rapidly with the ambient dimension.

Consider the uniform distribution on the Euclidean unit ball
$B_1(0)\subset\mathbb R^d$.
By rotational symmetry, $F_\theta$ is independent of $\theta$,
so it suffices to consider $\theta=e_1$.
Let $\omega_m$ denote the volume of the unit ball in
$\mathbb R^m$, with $\omega_0=1$.
At depth $s=1-y_1$, the cross-section is a
$(d-1)$-dimensional ball of radius $\sqrt{2s-s^2}$.
Therefore,
\begin{equation*}
    F_\theta(u)
    =
    \frac{\omega_{d-1}}{\omega_d}
    \int_0^u (2s-s^2)^{(d-1)/2}\,ds,
    \qquad 0<u\leq2.
\end{equation*}
For each fixed $d$, this gives
\begin{equation*}
    F_\theta(u)
    \sim
    \frac{2^{(d+1)/2}\omega_{d-1}}
         {(d+1)\omega_d}
    u^{(d+1)/2},
    \qquad u\downarrow0.
\end{equation*}
Consequently,
\begin{equation*}
    \lim_{u\downarrow0}
    \frac{F_\theta(2u)}{F_\theta(u)}
    =
    2^{(d+1)/2}.
\end{equation*}
Every admissible doubling constant must therefore satisfy
\begin{equation*}
    D_{\mathrm{UD}}
    \geq
    2^{(d+1)/2}.
\end{equation*}
Thus this family fails the dimension-uniformity requirement in
Assumption~\ref{ass:uniform-edge-doubling}, despite its fixed
support radius, smooth convex boundary, and rotational symmetry.

More generally, if
$
    F_\theta(u)\sim c_\theta u^{\alpha_\theta}
$
as $u\downarrow0$, with $c_\theta>0$, then
$
    F_\theta(2u)/F_\theta(u)\to2^{\alpha_\theta}.
$
Hence any doubling constant satisfies
$D_{\mathrm{UD}}\geq2^{\alpha_\theta}$.
This gives a necessary restriction on the cap exponents: a
dimension-uniform doubling constant requires the local cap exponents
$\alpha_\theta$ to remain bounded uniformly in $\theta$ and $d$.

\subsection{Local versus all-scale doubling}

At fixed $t>0$ and direction $\theta$, increasing the
exponential-tilt strength favors smaller cap depths; see the
posterior representation in the proof of
Theorem~\ref{thm:sec-3-normal-fluctuation-two-scale}.
This explains the relevance of local cap estimates.
To extend such estimates to all scales, we also control the
cap mass at a positive depth.

Suppose that there exist $u_*>0$ and $D_{\mathrm{loc}}\geq1$
such that
\begin{equation}\label{eq:local-edge-doubling-appendix}
    F_\theta(2u)
    \leq
    D_{\mathrm{loc}}F_\theta(u),
    \qquad
    \theta\in\mathbb S^{d-1},
    \quad 0<u\leq u_*.
\end{equation}
Assume in addition that
\begin{equation*}
    \inf_{\theta\in\mathbb S^{d-1}}
    F_\theta(u_*)
    \geq
    m_*>0.
\end{equation*}
For $u\geq u_*$, monotonicity gives
$F_\theta(u)\geq F_\theta(u_*)\geq m_*$, and hence
\begin{equation*}
    F_\theta(2u)
    \leq
    1
    \leq
    m_*^{-1}F_\theta(u).
\end{equation*}
Combining this with
\eqref{eq:local-edge-doubling-appendix}, we obtain
Assumption~\ref{ass:uniform-edge-doubling} with
\begin{equation*}
    D_{\mathrm{UD}}
    =
    \max\{D_{\mathrm{loc}},m_*^{-1}\}.
\end{equation*}
Thus the resulting doubling constant is independent of the
ambient dimension whenever $D_{\mathrm{loc}}$ and $m_*^{-1}$
are uniformly bounded across the family of targets.

The same argument applies to direction-dependent cutoff scales
$u_\theta$, provided that the local doubling constant and the
lower bound on $F_\theta(u_\theta)$ remain uniform.
This allows local geometric estimates to be used even when
their range of validity varies with the direction.

The normal-fluctuation estimate in the main text applies to
every $x$ at each positive time, and therefore covers all
exponential-tilt strengths.
We use the all-scale formulation so that the exponential-tilt
lemma can be applied directly, with the required cap-mass
control expressed through $D_{\mathrm{UD}}$.

\section{Recursive conditional representation of the material derivatives}
\label{app:recursive-material-representation}

The analysis in Section~\ref{sec:sec3-centered-posterior-moment-normal-form} controls the material derivatives $L^j x$ through their centered
posterior normal form. There is also a complementary representation of
these quantities as recursively generated conditional expectations under
the forward Gaussian coupling. This representation is not used in the
regularity estimates or in the convergence proof; we record it here as a
structural description of the Taylor coefficients appearing in the
truncated Taylor scheme. Recall that
\begin{equation*}
    \mu_{x,t}(dy)
    =
    \frac{
        \exp\left(
            -\frac{|x-\alpha(t)y|^2}{2\beta(t)^2}
        \right)P_0(dy)
    }{
        \int
        \exp\left(
            -\frac{|x-\alpha(t)z|^2}{2\beta(t)^2}
        \right)P_0(dz)
    }.
\end{equation*}
For a sufficiently regular vector field $v=v(t,x)$, define
\begin{equation*}
 L_v
    :=
    -\partial_t+D_x[\cdot][v],
\end{equation*}
and set
\begin{equation}
\begin{aligned}
    \Xi(t,x,y;v)
    &:=
  L_v
    \left(
        -\frac{|x-\alpha(t)y|^2}{2\beta(t)^2}
    \right)
    \\
    &=
    -\frac{\alpha'(t)}{\beta(t)^2}
    \langle x-\alpha(t)y,y\rangle
    -
    \frac{\beta'(t)}{\beta(t)^3}
    |x-\alpha(t)y|^2
    \\
    &\qquad
    -
    \frac{
        \langle x-\alpha(t)y,v\rangle
    }{
        \beta(t)^2
    }.
\end{aligned}
\end{equation}
We abbreviate $\Xi_v(t,x,y):=\Xi(t,x,y;v)$.

 Let $\varphi(t,x,y)$ be sufficiently regular and integrable. Direct computation gives the following posterior differentiation formula.

\begin{equation}\label{eq:posterior-material-differentiation}
 L_v
    \mathbb E_{\mu_{x,t}}[\varphi]
    =
    \mathbb E_{\mu_{x,t}}
    \left[ L_v\varphi
    \right]
    +
    \mathbb E_{\mu_{x,t}}
    \left[
        \left(
            \varphi
            -
            \mathbb E_{\mu_{x,t}}[\varphi]
        \right)
        \Xi_v(t,x,Y)
    \right].
\end{equation}
We now specialize this identity to the reverse probability flow. Set 
\begin{equation*} F_j(t,x) := L^j x(t,x), \qquad j\geq1. \end{equation*} 
In particular,
\begin{equation*} F_1(t,x) = Lx(t,x) = \frac{1}{2}\hat{s}(t,x)=\frac12 \left( a(t)x+\lambda(t)\bar y_{x,t} \right). \end{equation*} 
 For formal vector variables $ v_1,\ldots,v_j\in\mathbb R^d $
and a function 
$ \Phi = \Phi(t,x,y;v_1,\ldots,v_{j-1}), $ 
define the lifted material derivative

\begin{equation}\label{eq:lifted-material-derivative} 
\mathfrak D_j\Phi := -\partial_t\Phi + D_x\Phi[v_1] + \sum_{r=1}^{j-1} D_{v_r}\Phi[v_{r+1}]. 
\end{equation} 
The role of the last terms is to encode the identities $ LF_r=F_{r+1} $ without differentiating the fields $F_r$ explicitly. Define the first target by 

\begin{equation} \mathcal T_1(t,x,y) := \frac12 \left( a(t)x+\lambda(t)y \right), \end{equation} 
and, recursively, for $j\geq1$, define 

\begin{equation}\label{eq:recursive-material-target-appendix} 
\begin{aligned} \mathcal T_{j+1} (t,x,y;v_1,\ldots,v_j) := &\mathfrak D_j \mathcal T_j (t,x,y;v_1,\ldots,v_{j-1}) \\ &+ \left( \mathcal T_j (t,x,y;v_1,\ldots,v_{j-1}) - v_j \right) \Xi_{v_1}(t,x,y). 
\end{aligned} 
\end{equation}
For every fixed $t>0$, these targets are explicit vector-valued polynomials in $x$, $y$, and the formal variables $v_r$. Indeed, $\Xi_v$ is quadratic in $(x,y,v)$, and \eqref{eq:recursive-material-target-appendix} shows inductively that $\mathcal T_j$ has total degree at most $2j-1$. 

\begin{proposition}[Recursive conditional representation] \label{prop:recursive-conditional-material} For every integer $j\geq1$, 

\begin{equation}\label{eq:recursive-conditional-material}  F_j(t,x) = \mathbb E_{\mu_{x,t}} \left[ \mathcal T_j \left( t,x,Y; F_1(t,x),\ldots,F_{j-1}(t,x) \right) \right]. 
\end{equation}
For $j=1$, the list of inserted fields is empty. \end{proposition} 

\begin{proof} The case $j=1$ follows immediately from the posterior-mean representation: 
\begin{equation*} \begin{aligned} \mathbb E_{\mu_{x,t}} \left[ \mathcal T_1(t,x,Y) \right] = \frac12 \left( a(t)x+\lambda(t)\bar y_{x,t} \right) = F_1(t,x). \end{aligned} 
\end{equation*}
Suppose that \eqref{eq:recursive-conditional-material} holds at order $j$, and abbreviate 

\begin{equation*} \widetilde{\mathcal T}_j(t,x,y) := \mathcal T_j \left( t,x,y; F_1(t,x),\ldots,F_{j-1}(t,x) \right). \end{equation*} Since $ LF_r=F_{r+1},$
the chain rule and \eqref{eq:lifted-material-derivative} give

\begin{equation*} L\widetilde{\mathcal T}_j = \left. \mathfrak D_j\mathcal T_j \right|_{v_r=F_r,\ 1\leq r\leq j}. \end{equation*} Moreover, by the induction hypothesis, 
\begin{equation*} \mathbb E_{\mu_{x,t}} \left[ \widetilde{\mathcal T}_j \right] = F_j(t,x). \end{equation*}
Applying~\eqref{eq:posterior-material-differentiation} with $v=F_1=Lx$ therefore yields \begin{equation*} 
F_{j+1} = LF_j = \mathbb E_{\mu_{x,t}} \left[ \left. \mathfrak D_j\mathcal T_j \right|_{v_r=F_r} \right] + \mathbb E_{\mu_{x,t}} \left[ \left( \widetilde{\mathcal T}_j-F_j \right) \Xi_{F_1} \right].
\end{equation*} By \eqref{eq:recursive-material-target-appendix}, the expression inside the two expectations is precisely \begin{equation*} \mathcal T_{j+1} \left( t,x,Y; F_1,\ldots,F_j \right). \end{equation*} This proves the assertion at order $j+1$. \end{proof}

\paragraph{Population-level learning targets.}
For a fixed time $t$ and assuming that $F_1,\ldots,F_{j-1}$ are known
exactly, define
\begin{equation*}
    G_j := \mathcal T_j \left( t,X_t,X_0; F_1(t,X_t),\ldots,F_{j-1}(t,X_t) \right).
\end{equation*}
Proposition~\ref{prop:recursive-conditional-material} gives
\begin{equation*}
    \mathbb E[G_j\mid X_t=x] = F_j(t,x).
\end{equation*}
Consequently, the standard conditional-expectation characterization of
least squares implies
\begin{equation*}
    F_j(t,\cdot) \in \operatorname*{argmin}_{f} \mathbb E \left[ |f(X_t)-G_j|^2 \right],
\end{equation*}
with uniqueness up to $P_t$-null sets. Thus, when the lower-order
material derivatives are available, the next material derivative admits
an exact population-level denoising target. This is analogous in spirit
to generalized Tweedie identities for higher-order spatial scores
\cite{meng2021estimating}, but the quantities here are the vector-valued
material derivatives directly required by the Taylor scheme.

Two features separate this population-level identity from an error
estimate within the present framework. First, the target $G_j$ is
defined with the exact lower-order fields $F_1,\ldots,F_{j-1}$; if these
are replaced by learned approximations, their errors enter
$\mathcal T_j$ nonlinearly through $\Xi$, so the identity does not yield
a simultaneously consistent objective for $(F_1,\ldots,F_j)$. Second,
and more fundamentally, least-squares training controls a learned $f_j$
only in $L^2(P_t)$, whereas the stability and total-variation arguments
of Section~\ref{sec:high-order-convergence} require bounds on $D_xf_j$
in operator norm, uniformly in $x$---precisely the regularity that
Theorem~\ref{thm:weighted-estimates-Ljx} establishes for the exact
fields $L^jx$ and that no $L^2$ objective enforces. The representation
therefore identifies what is to be learned at each order; estimating
the sampling error of the resulting scheme requires regularity
hypotheses on the learned fields that are outside the scope of this
paper.
\newpage

\bibliography{ref}

@inproceedings{ho2020denoising,
  title={Denoising diffusion probabilistic models},
  author={Ho, Jonathan and Jain, Ajay and Abbeel, Pieter},
  booktitle={Advances in Neural Information Processing Systems},
  volume={33},
  pages={6840--6851},
  year={2020}
}

@article{song2020score,
  title={Score-based generative modeling through stochastic differential equations},
  author={Song, Yang and Sohl-Dickstein, Jascha and Kingma, Diederik P and Kumar, Abhishek and Ermon, Stefano and Poole, Ben},
  journal={arXiv preprint arXiv:2011.13456},
  year={2020}
}

@article{kim2012generalization,
  title={A generalization of {Caffarelli's} contraction theorem via (reverse) heat flow},
  author={Kim, Young-Heon and Milman, Emanuel},
  journal={Mathematische Annalen},
  volume={354},
  number={3},
  pages={827--862},
  year={2012}
}

@article{lee2022convergence,
  title={Convergence for score-based generative modeling with polynomial complexity},
  author={Lee, Holden and Lu, Jianfeng and Tan, Yixin},
  journal={Advances in Neural Information Processing Systems},
  volume={35},
  pages={22870--22882},
  year={2022}
}

@article{chen2022sampling,
  title={Sampling is as easy as learning the score: theory for diffusion models with minimal data assumptions},
  author={Chen, Sitan and Chewi, Sinho and Li, Jerry and Li, Yuanzhi and Salim, Adil and Zhang, Anru R},
  journal={arXiv preprint arXiv:2209.11215},
  year={2022}
}

@article{meng2021estimating,
  title={Estimating high order gradients of the data distribution by denoising},
  author={Meng, Chenlin and Song, Yang and Li, Wenzhe and Ermon, Stefano},
  journal={Advances in Neural Information Processing Systems},
  volume={34},
  pages={25359--25369},
  year={2021}
}

@article{li2025faster,
  title={Faster diffusion models via higher-order approximation},
  author={Li, Gen and Zhou, Yuchen and Wei, Yuting and Chen, Yuxin},
  journal={arXiv preprint arXiv:2506.24042},
  year={2025}
}

@article{huang2025fast,
  title={Fast convergence for high-order {ODE} solvers in diffusion probabilistic models},
  author={Huang, Daniel Zhengyu and Huang, Jiaoyang and Lin, Zhengjiang},
  journal={arXiv preprint arXiv:2506.13061},
  year={2025}
}

@article{huang2025convergence,
  title={Convergence analysis of probability flow {ODE} for score-based generative models},
  author={Huang, Daniel Zhengyu and Huang, Jiaoyang and Lin, Zhengjiang},
  journal={IEEE Transactions on Information Theory},
  year={2025},
  publisher={IEEE}
}

@article{beyler2025convergence,
  title={Convergence of deterministic and stochastic diffusion-model samplers: A simple analysis in {Wasserstein} distance},
  author={Beyler, Eliot and Bach, Francis},
  journal={arXiv preprint arXiv:2508.03210},
  year={2025}
}

@article{wu2024stochastic,
  title={Stochastic {Runge--Kutta} methods: Provable acceleration of diffusion models},
  author={Wu, Yuchen and Chen, Yuxin and Wei, Yuting},
  journal={arXiv preprint arXiv:2410.04760},
  year={2024}
}

@article{pfarr2026higherorder,
  title={Analyzing the error of generative diffusion models: From {Euler--Maruyama} to higher-order schemes},
  author={Pfarr, Emanuel and Timofte, Radu and Werner, Frank},
  journal={arXiv preprint arXiv:2601.18425},
  year={2026}
}

@article{de2022convergence,
  title={Convergence of denoising diffusion models under the manifold hypothesis},
  author={De Bortoli, Valentin},
  journal={arXiv preprint arXiv:2208.05314},
  year={2022}
}

@inproceedings{pierretdiffusion,
  title={Diffusion models for Gaussian distributions: Exact solutions and Wasserstein errors},
  year={2025},
  author={Pierret, Emile and Galerne, Bruno},
  booktitle={Forty-second International Conference on Machine Learning}
}

@article{neufeld2026universal,
  title={Universal approximation results for neural networks with non-polynomial activation function over non-compact domains},
  author={Neufeld, Ariel and Schmocker, Philipp},
  journal={Analysis and Applications},
  volume={24},
  number={05},
  pages={1123--1173},
  year={2026},
  publisher={World Scientific}
}

@article{fournier2015rate,
  title={On the rate of convergence in Wasserstein distance of the empirical measure},
  author={Fournier, Nicolas and Guillin, Arnaud},
  journal={Probability theory and related fields},
  volume={162},
  number={3},
  pages={707--738},
  year={2015},
  publisher={Springer}
}

@article{gao2025wasserstein,
  title={Wasserstein Convergence Guarantees for a General Class of Score-Based Generative Models},
  author={Gao, Xuefeng and Nguyen, Hoang M and Zhu, Lingjiong},
  journal={Journal of machine learning research},
  year={2025},
  publisher={JMLR}
}

@inproceedings{silveribeyond,
  title={Beyond Log-Concavity and Score Regularity: Improved Convergence Bounds for Score-Based Generative Models in W2-distance},
  author={Silveri, Marta Gentiloni and Ocello, Antonio},
  booktitle={Forty-second International Conference on Machine Learning},
year = {2025}
}

@article{bruno2025wasserstein,
  title={Wasserstein Convergence of Score-based Generative Models under Semiconvexity and Discontinuous Gradients},
  author={Bruno, Stefano and Sabanis, Sotirios},
  journal={arXiv preprint arXiv:2505.03432},
  year={2025}
}

@inproceedings{chen2023improved,
  title={Improved analysis of score-based generative modeling: user-friendly bounds under minimal smoothness assumptions},
  author={Chen, Hongrui and Lee, Holden and Lu, Jianfeng},
  booktitle={Proceedings of the 40th International Conference on Machine Learning},
  pages={4735--4763},
  year={2023}
}

@inproceedings{benton2024nearly,
  title={Nearly $ d $-linear convergence bounds for diffusion models via stochastic localization},
  author={Benton, Joe and De Bortoli, Valentin and Doucet, Arnaud and Deligiannidis, George},
  booktitle={International Conference on Learning Representations},
  pages={36916--36936},
  year={2024}
}

@misc{conforti2024klconvergenceguaranteesscore,
      title={KL Convergence Guarantees for Score diffusion models under minimal data assumptions}, 
      author={Giovanni Conforti and Alain Durmus and Marta Gentiloni Silveri},
      year={2024},
      eprint={2308.12240},
      archivePrefix={arXiv},
      primaryClass={math.ST},
      url={https://arxiv.org/abs/2308.12240}, 
}

@article{xixianwasserstein,
  title={Wasserstein Bounds for generative diffusion models with Gaussian tail targets},
  author={Wang, Xixian and Wang, Zhongjian},
  journal={Transactions on Machine Learning Research},year={2026}
}

@article{haussmann1986time,
  title={Time reversal of diffusions},
  author={Haussmann, Ulrich G and Pardoux, Etienne},
  journal={The Annals of Probability},
  pages={1188--1205},
  year={1986},
  publisher={JSTOR}
}

@article{And_80,
    author = {B. Anderson},
    title = {Reverse-time diffusion equation models},
    journal = {Stoch. Process. Appl.},
  volume={12(3)},
   pages={313-326},
    year = {1982}
}

@article{meng2025pathway,
  title={Pathway to $ O (\sqrt {d}) $ Complexity bound under Wasserstein metric of flow-based models},
  author={Meng, Xiangjun and Wang, Zhongjian},
  journal={arXiv preprint arXiv:2512.06702},
  year={2025}
}

@book{hairer1993solving,
  title={Solving ordinary differential equations I: Nonstiff problems},
  author={Hairer, Ernst and Wanner, Gerhard and N{\o}rsett, Syvert P},
  year={1993},
  publisher={Springer}
}

@book{kloeden1992numerical,
  title     = {Numerical Solution of Stochastic Differential Equations},
  author    = {Kloeden, Peter E. and Platen, Eckhard},
  series    = {Applications of Mathematics},
  volume    = {23},
  year      = {1992},
  publisher = {Springer-Verlag},
  address   = {Berlin},
  doi       = {10.1007/978-3-662-12616-5}
}

@article{lu2023mathematical,
  title={Mathematical analysis of singularities in the diffusion model under the submanifold assumption},
  author={Lu, Yubin and Wang, Zhongjian and Bal, Guillaume},
  journal={arXiv preprint arXiv:2301.07882},
  year={2023}
}

@incollection{bakry1985diffusions,
  author    = {Bakry, Dominique and {\'E}mery, Michel},
  title     = {Diffusions hypercontractives},
  booktitle = {S{\'e}minaire de Probabilit{\'e}s XIX 1983/84},
  series    = {Lecture Notes in Mathematics},
  volume    = {1123},
  pages     = {177--206},
  publisher = {Springer},
  year      = {1985}
}

@book{bakry2014analysis,
  author    = {Bakry, Dominique and Gentil, Ivan and Ledoux, Michel},
  title     = {Analysis and Geometry of {Markov} Diffusion Operators},
  series    = {Grundlehren der mathematischen Wissenschaften},
  volume    = {348},
  publisher = {Springer},
  address   = {Cham},
  year      = {2014},
  doi       = {10.1007/978-3-319-00227-9}
}

@article{dziuk2013finite,
  author  = {Dziuk, Gerhard and Elliott, Charles M.},
  title   = {Finite element methods for surface {PDEs}},
  journal = {Acta Numerica},
  volume  = {22},
  pages   = {289--396},
  year    = {2013},
  doi     = {10.1017/S0962492913000056}
}

@article{talay1990expansion,
  author  = {Talay, Denis and Tubaro, Luciano},
  title   = {Expansion of the global error for numerical schemes solving stochastic differential equations},
  journal = {Stochastic Analysis and Applications},
  volume  = {8},
  number  = {4},
  pages   = {483--509},
  year    = {1990},
  doi     = {10.1080/07362999008809220}
}

@article{mattingly2010convergence,
  author  = {Mattingly, Jonathan C. and Stuart, Andrew M. and Tretyakov, Michael V.},
  title   = {Convergence of numerical time-averaging and stationary measures via {Poisson} equations},
  journal = {SIAM Journal on Numerical Analysis},
  volume  = {48},
  number  = {2},
  pages   = {552--577},
  year    = {2010},
  doi     = {10.1137/090770527}
}

\end{document}